\documentclass[12pt]{amsart}
\usepackage{amssymb, a4wide, mathdots,url,hyperref}

\newcommand{\grp}[1]{{\mathbf{#1}}}

\newcommand{\C}{\mathbb{C}}
\newcommand{\Z}{\mathbb{Z}}
\newcommand{\Q}{\mathbb{Q}}
\newcommand{\R}{\mathbb{R}}
\newcommand{\N}{\mathbb{N}}

\newcommand{\A}{\mathbb{A}}
\newcommand{\Ker}{\operatorname{Ker}}
\renewcommand{\Re}{\operatorname{Re}}

\newcommand{\orb}{\mathfrak{o}}
\newcommand{\OOO}{\mathcal{O}}             

\newcommand{\rdp}{\mathcal{A}^{rd}}
\newcommand{\grph}{\mathfrak{G}}
\newcommand{\Lie}{\operatorname{Lie}}
\newcommand{\Fam}{\mathcal{F}}
\newcommand{\srts}{\Delta}
\newcommand{\unr}{\operatorname{unr}}
\newcommand{\rd}[1]{C_{#1}(\grp U(\A)M\bs \grp G(\A))}
\newcommand{\rdsmth}[1]{C^{\infty}_{#1}(\grp U(\A)M\bs \grp G(\A))}
\newcommand{\univ}{\mathcal{U}}
\newcommand{\disc}{\operatorname{disc}}
\newcommand{\discdata}{\mathfrak{d}}
\newcommand{\Discdata}{\mathfrak{D}}

\newcommand{\Ind}{\operatorname{Ind}}

\newcommand{\AF}{\mathcal{A}^{mg}}
\newcommand{\AS}[1]{[\grp{#1}]}
\newcommand{\RAS}[2]{[\grp{#1}]_{#2}}
\newcommand{\AUT}{\mathcal{AUT}}
\newcommand{\iii}{\mathrm{i}}
\newcommand{\eps}{\epsilon}
\newcommand{\stng}{\operatorname{st}}
\newcommand{\eisen}{\mathcal{E}}
\newcommand{\Hom}{\operatorname{Hom}}
\renewcommand{\subset}{\subseteq}
\renewcommand{\supset}{\supseteq}
\newcommand{\rlvt}{{relevant}}

\newcommand{\bs}{\backslash}
\newcommand{\der}{\operatorname{der}}
\newcommand{\domain}{\mathfrak{D}}

\newcommand{\data}{\mathfrak{d}}
\newcommand{\Speh}{\operatorname{Speh}}
\newcommand{\pseudospace}{\mathfrak{P}}

\newcommand{\WR}{\mathcal{W}}
\newcommand{\diag}{\operatorname{diag}}
\newcommand{\Hconv}{H\operatorname{-conv}}
\newcommand{\dist}{\operatorname{-dist}}
\newcommand{\type}{\operatorname{-type}}

\newcommand{\GL}{\operatorname{GL}}
\newcommand{\Sp}{\operatorname{Sp}}
\newcommand{\SL}{\operatorname{SL}}
\newcommand{\rest}{\big|}
\newcommand{\cusp}{{\operatorname{cusp}}}
\newcommand{\abs}[1]{\left|{#1}\right|}
\newcommand{\aaa}{\mathfrak{a}}
\newcommand{\pr}{\operatorname{pr}}
\newcommand{\Res}{\operatorname{Res}}
\newcommand{\Sym}{\operatorname{Sym}}
\newcommand{\dsum}{\oplus}
\newcommand{\modulus}{\delta}
\newcommand{\trivchar}{{\bf 1}}
\newcommand{\siegel}{{\mathfrak{S}}}
\newcommand{\inj}{\iota}
\newcommand{\inn}{C}
\newcommand{\p}{{\mathfrak{p}}}   
\newcommand{\h}{\mathcal{H}}
\newcommand{\Ht}{H}
\newcommand{\clss}{\mathcal{C}}
\newcommand{\conj}[2]{[#1]_{#2}}           
\newcommand{\sprod}[2]{\left\langle#1,#2\right\rangle}
\newcommand{\norm}[1]{\lVert#1\rVert}
\newcommand{\dig}{d}
\newcommand{\conteisen}{\mathfrak{IE}}
\newcommand{\openstab}{L_H}

\newcommand{\cuspdatum}{\mathfrak{X}}
\newcommand{\cuspdata}{\mathfrak{E}}
\newcommand{\sings}{\operatorname{Sing}}
\newcommand{\singcls}{\mathfrak{C}}
\newcommand{\sing}{\mathfrak{S}}
\newcommand{\affines}{\mathfrak{A}}
\newcommand{\Ktypes}{\mathfrak{F}}
\newcommand{\alldata}{\mathfrak{B}}
\newcommand{\pointcls}{\mathfrak{Q}}
\newcommand{\imb}{{\imath}}

\newcommand{\sm}[4]{{\bigl(\begin{smallmatrix}{#1}&{#2}\\{#3}&{#4}
\end{smallmatrix}\bigr)}}

\newtheorem{theorem}{Theorem}[section]
\newtheorem{lemma}[theorem]{Lemma}
\newtheorem{proposition}[theorem]{Proposition}
\newtheorem{remark}[theorem]{Remark}
\newtheorem{conjecture}[theorem]{Conjecture}
\newtheorem{definition}[theorem]{Definition}
\newtheorem{corollary}[theorem]{Corollary}

\title{On the distinguished spectrum of $\Sp_{2n}$ with respect to $\Sp_n\times\Sp_n$}
\author{Erez Lapid}
\address{Department of Mathematics, Weizmann Institute of Science, Rehovot 7610001, Israel}
\email{erez.m.lapid@gmail.com}
\author{Omer Offen}
\address{Department of Mathematics, Technion -- Israel Institute of Technology , Haifa 3200003, Israel}
\email{offen@tx.technion.ac.il}
\date{\today}

\begin{document}

\setcounter{tocdepth}{1}

\begin{abstract}
Given a reductive group $G$ and a reductive subgroup $H$, both defined over a number field $F$,
we introduce the notion of the $H$-distinguished automorphic spectrum of $G$
and analyze it for the pairs $(\GL_{2n},\Sp_n)$ and $(\Sp_{2n},\Sp_n\times\Sp_n)$.
In the first case we give a complete description using results of Jacquet--Rallis, Offen and Yamana.
In the second case we give an upper bound, generalizing vanishing results of Ash--Ginzburg--Rallis
and a lower bound, extending results of Ginzburg--Rallis--Soudry.
\end{abstract}

\maketitle
\tableofcontents
\section{Introduction}
Let $G$ be a reductive group over a number field $F$ and let $H$ be a closed subgroup of $G$ defined over $F$.
Let $\A$ be the ring of adeles of $F$.
In the theory of automorphic forms one is often interested in period integrals
\begin{equation} \label{eq: Hperiod}
\int_{H(F)\bs H(\A)}\varphi(h)\ dh
\end{equation}
(assuming convergent) and in automorphic representations of $G(\A)$ on which such an integral is not identically zero.
In certain cases these representations, which are called $H$-distinguished, are characterized by functoriality and the period integral is related to
special values of $L$-functions.
In the analysis of these period integrals one is often lead to study non-convergent integrals which have to be suitably regularized.
More fundamentally, one may ask whether there is a sensible notion of the $H$-distinguished spectrum which captures both
the automorphic representations for which \eqref{eq: Hperiod} converges (and is non-zero) as well as others.
In this paper we propose a candidate for this space and study it in specific cases.
Namely, (fixing a central character which we suppress from the notation for simplicity)
we consider the orthogonal complement $L_{H\dist}^2(G(F)\bs G(\A))$ in $L^2(G(F)\bs G(\A))$
of the space of pseudo Eisenstein series $\varphi$ on $G(F)\bs G(\A)$ such that $\int_{H(F)\bs H(\A)}\varphi(hg)\ dh=0$ for all $g\in G(\A)$.
We will be interested in the spectral decomposition of $L_{H\dist}^2(G(F)\bs G(\A))$, and in particular in its discrete part $L_{\disc,H\dist}^2(G(F)\bs G(\A))$.
(It is possible to consider smooth rapidly decreasing functions instead of pseudo Eisenstein series but we do not know whether this gives rise to the same
space in general. At any rate, it seems that the choice above is the most convenient for computation.)
Arguably, the most curious phenomenon observed in the paper is that $L_{\disc,H\dist}^2(G(F)\bs G(\A))$ may contain an irreducible constituent for which
the integral \eqref{eq: Hperiod} is not convergent! (See Theorem \ref{thm: lwrbnd for distspec} and Remark \ref{rmk: not conv}.)

We will describe $L_{H\dist}^2(G(F)\bs G(\A))$ (and in particular, $L_{\disc,H\dist}^2(G(F)\bs G(\A))$) completely in the case where
$G=\GL_{2n}$ and $H=\Sp_n$ (the symplectic group of rank $n$).
Recall that in this case, by the results of M\oe glin--Waldspurger, the entire space $L_{\disc}^2(G(F)\bs G(\A))$
can be described explicitly in terms of the cuspidal representations of $\GL_m(\A)$
for all divisors $m$ of $2n$ \cite{MR1026752}. It turns out that $L_{\disc,H\dist}^2(G(F)\bs G(\A))$ is the contribution to $L_{\disc}^2(G(F)\bs G(\A))$
of all divisors $m$ of $n$. More generally, in terms of the Langlands decomposition of $L^2(G(F)\bs G(\A))$,
$L_{H\dist}^2(G(F)\bs G(\A))$ consists of the part whose discrete data belongs to $L_{\disc,M_H\dist}^2(M(F)\bs M(\A))$
where $M$ ranges over the Levi subgroups of the form $M=\GL_{2n_1}\times\dots\times\GL_{2n_k}$ and $M_H=\Sp_{n_1}\times\dots\times\Sp_{n_k}$
(so that $L_{\disc,M_H\dist}^2(M(F)\bs M(\A))=\otimes L_{\disc,\Sp_{n_i}\dist}^2(\GL_{2n_i}(F)\bs\GL_{2n_i}(\A))$
is described as above).
The main input for this case is the results of Jacquet--Rallis, the second-named author and Yamana about symplectic periods of automorphic forms
on $\GL_{2n}$ \cite{MR1142486, MR2248833, MR2254544, Ya}.
In fact, we can formulate the same result for a variant of $L_{H\dist}^2(G(F)\bs G(\A))$ where instead of pseudo Eisenstein series
one uses a much bigger space (Corollary \ref{cor: Spndist}).

The results for the pair $(\GL_{2n},\Sp_n)$ suggest a close connection between the distinguished spectrum and the automorphic spectrum of the group
$\GL_n$ through functoriality. However, it is not completely clear how to make this connection precise. (See Remark \ref{rem: functoriality}.)

A more interesting case is the pair $(G,H)=(\Sp_{2n},\Sp_n\times\Sp_n)$ which is the main focus of this paper.
In this case we do not know even a conjectural description of $L_{\disc,H\dist}^2(G(F)\bs G(\A))$.
However, we will be able to identify a certain subspace of $L_{\disc,H\dist}^2(G(F)\bs G(\A))$ which seems to be the most relevant for the descent
construction of Ginzburg--Rallis--Soudry (cf.~\cite{MR1740991}). (We will say more about that in a future paper.)
In particular, we find there representations for which \eqref{eq: Hperiod} does not converge.

In the opposite direction, by results of Jacquet--Rallis and Ash--Ginzburg--Rallis, both cases above are examples of pairs $(G,H)$ for which no
cuspidal representation of $G(\A)$ is $H$-distinguished \cite{MR1233493, MR1142486}.
Recall that $L^2(G(F)\bs G(\A))$ has a coarse decomposition $L^2(G(F)\bs G(\A))=\hat\dsum_{\cuspdatum}L^2_{\cuspdatum}(G(F)\bs G(\A))$
according to cuspidal data $\cuspdatum$.
We will show that for many cuspidal data $\cuspdatum$ we have $L^2_{\cuspdatum}(G(F)\bs G(\A))\cap L_{H\dist}^2(G(F)\bs G(\A))=0$,
extending the abovementioned vanishing results.
Moreover, for the remaining cuspidal data $\cuspdatum$ we will control the affine spaces $\sing$ which potentially contribute to
$L_{H\dist}^2(G(F)\bs G(\A))$ under the finer decomposition (due to Langlands)
\[
L^2_{\cuspdatum}(G(F)\bs G(\A))=\dsum_{\sing}L^2_{\cuspdatum}(G(F)\bs G(\A))_{\sing}
\]
according to intersections of singular hyperplanes (cf.~\cite{MR0579181, MR1361168}).
In the case $(G,H)=(\GL_{2n},\Sp_n)$ this analysis (together with the results of \cite{Ya})
is propitiously sufficient for the precise description of $L^2_{H\dist}(G(F)\bs G(\A))$.
This is to a large extent due to the simple description of $L^2_{\disc}(G(F)\bs G(\A))$.
In the case $(G,H)=(\Sp_{2n},\Sp_n\times\Sp_n)$ the upshot is unfortunately a bit technical to formulate and is stated as Theorem \ref{thm: reformulate upper bound}
in \S\ref{sec: main result}. At any rate, we do not expect that our result gives $L^2_{H\dist}(G(F)\bs G(\A))$ precisely in this case, but only an upper bound.

The main ingredient for our analysis is a formula for the $H$-period of pseudo Eisenstein series.
This kind of formula was considered for other pairs $(G,H)$ where $H$ is the fixed point subgroup of an involution
and is probably quite general \cite{MR1625060, MR2010737, MR2254544}.
It is based on an analysis of double cosets $P\bs G/H$ where $P$ is a parabolic subgroup of $G$
(starting with the fundamental results of Springer \cite{MR803346}).

Once again, the results suggest a relationship between the distinguished spectrum for the pair $(\GL_{2n},\Sp_n)$ and that of the pair
$(\GL_{2n},\GL_n\times\GL_n)$ via functoriality. However, the precise relationship requires further analysis and possibly additional
variants of the notion of distinguished spectrum.

In general, one may wonder whether $L_{H\dist}^2(G(F)\bs G(\A))$ admits a decomposition reminiscent to the Langlands decomposition
of $L^2(G(F)\bs G(\A))$, namely in terms of $L_{\disc,H_M\dist}^2(M(F)\bs M(\A))$ for suitable pairs $(M,H_M)$ where
$M$ is a Levi subgroup of $G$. We are not in a position to formulate a precise conjecture in general but we will do so in the
case $(G,H)=(\Sp_{2n},\Sp_n\times\Sp_n)$ (Conjecture \ref{conj: main}).

We mention that in the more general context of spherical varieties, spectral analysis of period integrals (as well as their local counterparts)
are studied in a recent work by Sakellaridis--Venkatesh \cite{1203.0039}.
However, their focus is somewhat different and in particular we do not know what role does the space $L_{H\dist}^2(G(F)\bs G(\A))$ play
in their theory, if any.

As alluded to above, our main result will be applied in a subsequent paper to analyze the descent
map of Ginzburg--Rallis--Soudry and its image, suggesting a way to study functoriality (in the generic case) without
using the converse theorem or the trace formula.

The structure of the paper is the following.
We start with general notation and auxiliary results (\S\ref{sec: notation}).
The first part of the paper (\S\ref{sec: double cosets}-\S\ref{sec: Hperiods}) is devoted to the
computation of the $H$-period of pseudo Eisenstein series.
To that end we first study the double cosets $P\bs G/H$ where $G=\Sp_{2n}$, $H=\Sp_n\times\Sp_n$
and $P$ is a parabolic subgroup of $G$ and single out the double cosets which ultimately contribute to the formula (\S\ref{sec: double cosets}).
We then study the main analytic object, namely the intertwining periods, their convergence and analytic properties (\S\ref{sec: intertwining periods}).
The formula for the $H$-period of pseudo Eisenstein series is finally derived in \S\ref{sec: Hperiods}.
In the second part of the paper (\S\ref{sec: dist spec}-\S\ref{sec: main result}) we apply this formula to the study of the $H$-distinguished spectrum.
We first define this notion (in a general context) and explain its relation to results of the first part (\S\ref{sec: dist spec}).
Then we analyze the pair $(\GL_{2n},\Sp_n)$ and provide complete results for this case (\S\ref{sec: GL2nSpn}).
Finally, we explicate the results of \S\ref{sec: dist spec} for the pair $(\Sp_{2n},\Sp_n\times\Sp_n)$ to provide an upper bound on the distinguished
spectrum in this case (\S\ref{sec: main result}). We also exhibit in \S\ref{sec: main result} a lower bound for the distinguished spectrum by showing that
it contains certain residual representations considered in \cite{MR1954940}.

\subsection*{Acknowledgement}
Part of this work was done while the authors participated in the program ``Research in Pairs''
in the Mathematisches Forschungsinstitut Oberwolfach. We are very grateful to the MFO for providing ideal working conditions
for collaboration.
The first-named author was partially supported by a grant from the Minerva Stiftung.
The second-named author was partially supported by grant \# 1394/12 from the Israel Science Foundation.

\section{Notation and preliminaries} \label{sec: notation}
\subsection{General notation}
Let $F$ be a number field and $\A=\A_F$ its ring of adeles.
In general, if $\grp X$ is an algebraic variety over $F$ we write $X=\grp X(F)$ for its $F$-points.
For an algebraic group $\grp Q$ defined over $F$ we denote by $X^*(\grp Q)$ the lattice of $F$-rational characters of $\grp Q$.
Let $\aaa_Q^*=X^*(\grp Q)\otimes_{\Z}\R$ and let $\aaa_Q=\Hom_{\R}(\aaa_Q^*,\R)$ be its dual vector space with the natural pairing
$\sprod{\cdot}{\cdot}=\sprod{\cdot}{\cdot}_Q$.
We endow $\aaa_Q$ and $\aaa_Q^*$ with Euclidean norms $\|\cdot\|$.
We denote by $\aaa_\C$ the complexification of a real vector space $\aaa$.
We also set
\[
\grp Q(\A)^1=\{q\in \grp Q(\A):\forall \chi\in X^*(\grp Q),\,\abs{\chi(q)}_{\A^*}=1\}.
\]
There is an isomorphism
\[
H_Q:\grp Q(\A)^1 \bs \grp Q(\A) \rightarrow \aaa_Q
\]
such that $e^{\sprod{\chi}{H_Q(q)}}=\abs{\chi(q)}_{\A^*}$, $\chi\in X^*(\grp Q)$, $q\in \grp Q(\A)$.

Let $\delta_Q$ denote the modulus function of $\grp Q(\A)$. It is a character of $\grp Q(\A)^1 \bs \grp Q(\A)$ and therefore there exists
$\rho_Q\in \aaa_Q^*$ such that
\[
\delta_Q(q)=e^{\sprod{2\rho_Q}{H_Q(q)}}, \ \ \ q\in \grp Q(\A).
\]

Let $\grp G$ be a reductive group over $F$ and $\grp{P_0}$ a minimal parabolic subgroup of $\grp G$ defined over $F$.
Fix a maximal $F$-split torus $\grp T$ of $\grp G$ contained in $\grp P_0$
and a maximal compact subgroup $K$ of $\grp G(\A)$ which is in good position with respect to $P_0$,
so that the Iwasawa decomposition $\grp G(\A)=\grp P_0(\A)K$ holds. We use it to extend the map
$H_0=H_{P_0}:\grp{P_0}(\A)\rightarrow\aaa_{P_0}$ to a right $K$-invariant function on $\grp G(\A)$.
Finally, we also fix a Siegel domain $\siegel_G$ for $G\bs \grp G(\A)$ and let $\siegel_G^1=\siegel_G\cap\grp G(\A)^1$ (cf.~\cite[I.2.1]{MR1361168}).

If $\Omega$ is a compact subset of $\grp G(\A)$ then we have
\begin{equation}\label{eq: exp bd}
\sup_{x\in\Omega, g\in \grp G(\A)}\norm{H_0(gx)-H_0(g)}=\sup_{x\in\Omega, k\in K}\norm{H_0(kx)}<\infty.
\end{equation}

Let $\grp{T_G}$ be the split part of the (Zariski) identity connected component of the center of $\grp G$.
Applying the imbedding $x\mapsto 1\otimes x: \R\to F_\infty=F\otimes_\Q \R\hookrightarrow\A$ we imbed $\grp{T_G}(\R)$ in
$\grp{T_G}(F_\infty) \hookrightarrow \grp{T_G}(\A)$
and denote by $A_G$ the image of the identity component $\grp{T_G}(\R)^\circ$ (in the usual topology) in $\grp{T_G}(\A)$.
Then $H_G:A_G \to \aaa_G$ is an isomorphism. Denote by $\nu\mapsto e^\nu$ its inverse.

It will be convenient to use the shorthand notation
\[
\AS{G}=A_GG\bs\grp G(\A).
\]
More generally, if $\grp{H}$ is a subgroup of $\grp{G}$ defined over $F$ then we set
\begin{equation} \label{def: AGH}
A_G^H=A_G\cap\grp H(\A)
\end{equation}
and
\begin{equation} \label{def: RAS}
\RAS HG=A_G^HH\bs\grp H(\A).
\end{equation}

For $n\in \N$ let $\grp{GL}_n$ be the general linear group of rank $n$.
For a matrix $g=(g_{i,j})\in \grp{GL}_n(\A)$ with $g^{-1}=((g^{-1})_{i,j})$ let
\[
\norm{g}=\norm{g}_{\grp{GL}_n(\A)}=\prod_v\max_{1\le i,\,j\le n}\{\abs{g_{i,j}}_v,\abs{(g^{-1})_{i,j}}_v\}
\]
where the product (here and elsewhere) ranges over all places $v$ of $F$.
Similarly, if $k$ is a local field with normalized absolute value $\abs{\cdot}_k$
we define $\|g\|=\max_{1\le i,\,j\le n}\{\abs{g_{i,j}}_k,\abs{(g^{-1})_{i,j}}_k\}$ for any $g\in \grp{GL}_n(k)$.
(Note that we use the notation $\norm{\cdot}$ in several settings. Hopefully this will be clear from the context.)

Fix a faithful $F$-rational representation $\rho:\grp G \rightarrow \grp{GL}_n$ and define
$\norm{g}_\rho=\norm{\rho(g)}_{\grp{GL}_n(\A)}$. Often, we omit the subscript $\rho$ if it is clear from the context.
We record some standard facts about $\norm{\cdot}$. (See \cite[Lemma I.2.2]{MR1361168} where the convention of $\norm{\cdot}_\rho$
is slightly different, but this entails little change.)
Henceforth, we use the notation $A\ll B$ to mean that there exists a constant $c$ such that $A\le cB$.
The constant $c$ is understood to be independent of the underlying parameters.
If we want to emphasize the dependence of $c$ on other parameters, say $T$, we will write $A\ll_TB$.
(We will suppress the implicit dependence on the group $\grp G$ and the representation $\rho$.)

\begin{subequations}
\begin{equation}
1\ll\norm{g}\text{ for all }g\in\grp G(\A).
\end{equation}
\begin{equation} \label{eq: submult}
\norm{g_1g_2}\ll\norm{g_1}\norm{g_2}\text{ for all }g_1,g_2\in \grp G(\A).
\end{equation}
\begin{equation}\label{eq: bdexpnm}
\norm{\Ht_0(g)}\ll 1+\log\norm{g} \text{ for all }g\in \grp G(\A).
\end{equation}
\begin{equation}\label{eq: bdsieg}
\log\norm{g} \ll 1+\norm{\Ht_0(g)}\text{ for all }g\in \grp \siegel_G^1.
\end{equation}
\begin{equation} \label{eq: gammag bigger}
\norm{g}\ll\norm{\gamma g}\text{  for any $g\in\siegel_G$ and $\gamma\in G$.}
\end{equation}
\begin{equation} \label{eq: agmult}
\text{There exists $N$ such that }\norm{a}\norm{g}\ll\norm{ag}^N\text{ for all }g\in \grp G(\A)^1,\ a\in A_G.
\end{equation}
\end{subequations}

Let $\Sigma=R(T,G)$ be the root system of $G$ with respect to $T$ and $\srts_0=\srts_0^G$
the basis of simple roots with respect to $P_0$, viewed as a subset of $\aaa_0^*$.
For $\alpha\in\Sigma$ we denote by $\alpha^\vee$ the corresponding coroot.
Recall that a standard (resp., semistandard) parabolic subgroup (defined over $F$) is one containing $\grp P_0$ (resp., $\grp T$).
The standard parabolic subgroups of $\grp G$ are parameterized by subsets of $\srts_0^G$.
A semistandard parabolic group $\grp P$ admits a unique Levi decomposition $\grp P=\grp M\ltimes\grp U$ where $\grp M\supset\grp T$.
We call these $\grp M$'s semistandard Levi subgroups (or standard, if $\grp P$ is standard).

For any standard parabolic subgroup $\grp P=\grp M\ltimes \grp U$ we have
\begin{equation} \label{eq: mumult}
\begin{split}
\norm{m}\ll\norm{mu}\text{ for all }m\in \grp M(\A), u\in \grp U(\A).
\end{split}
\end{equation}
(See proof of \cite[Lemma 6.1.1]{MR2010737} or of \cite[Lemme II.3.1]{MR1989693}.)

For a semi-standard parabolic subgroup $\grp P=\grp M\ltimes \grp U$ with semi-standard Levi subgroup $\grp M$ and unipotent radical $\grp U$ we have $\aaa_P^*=\aaa_M^*$.
If $\grp Q$ contains $\grp P$ then there is a unique Levi decomposition $\grp Q=\grp L\ltimes \grp V$ with $\grp L\supseteq \grp M$.
Thus, then $\aaa_L$ is a subspace of $\aaa_M$ and there is a canonical direct sum decomposition $\aaa_M=\aaa_L \dsum \aaa_M^L$.
The dual spaces satisfy the analogous properties.
In particular $\rho_P=\rho_Q+\rho_P^Q$ where $\rho_P^Q\in (\aaa_M^L)^*$ is the unique element such that
$\delta_{P\cap L}(p)=e^{\sprod{2\rho_P^Q}{H_{P}(p)}}$, $p\in(\grp P\cap \grp L)(\A)$.

Set $\aaa_0=\aaa_T$ and more generally, $\aaa_0^M=\aaa_T^M$ and similarly for the dual spaces. Set also $\rho_0=\rho_{P_0}$.
Recall that $\Ht_0:\grp G(\A)\rightarrow \aaa_0$ is defined via the Iwasawa decomposition.
Similarly we can define $H_P:\grp G(\A)\rightarrow\aaa_P=\aaa_M$.
We denote by $\Ht_0^M:\grp G(\A) \rightarrow \aaa_0^M$ the composition of $\Ht_0$ with the orthogonal projection to $\aaa_0^M$ and more generally,
by $\Ht_M^L:\grp G(\A) \rightarrow \aaa_M^L$ the composition of $\Ht_0^L$ with the orthogonal projection to $\aaa_M^L$.

\emph{Henceforth, unless otherwise mentioned all parabolic subgroups and Levi subgroups of $\grp G$ considered will be implicitly assumed to be standard (and defined over $F$).}

For a Levi subgroup $M$ of $G$ the root system $\Sigma^M=R(T,M)$ is a subsystem of $\Sigma$.
Let $\srts_0^M=\Sigma^M\cap \srts_0$ be the set of simple roots in $M$ with respect to $M\cap B$.

For a parabolic subgroup $P=M\ltimes U$ of $G$ let $\Sigma_M=R(T_M,G)\subseteq \aaa_M^*$, $\Sigma_P$ the subset of positive roots in $\Sigma_M$
with respect to $P$ and $\srts_P$ the non-zero projections to $\aaa_M^*$ of elements of $\srts_0$.
For $\alpha\in \Sigma_M$ we write $\alpha>0$ if $\alpha\in \Sigma_P$ and $\alpha<0$ otherwise.
Once again we denote by $\alpha^\vee$ the corresponding coroot (see \cite[I.1.11]{MR1361168}).
More generally, if $P\subset Q=L \ltimes V$ then we write $\srts_P^Q\subset\srts_P$ for the non-zero restrictions to $\aaa_M^*$ of elements of $\srts_0^L$.

Let $W=W^G=N_G(T)/C_G(T)$ be the Weyl group of $G$ with respect to $T$. (In the case where $G$ is split, $C_G(T)=T$.)
We assume that the fixed Euclidean structure on $\aaa_0$ is $W$-invariant.
We consider elements of $W$ as $C_G(T)$-cosets in $N_G(T)$.
In particular, for $w\in W$ we write $n\in w$ whenever $n\in N_G(T)$ represents $w$.
For a Levi subgroup $M$ let ${}_MW_M$ be the set of $w\in W$ such that $w$ has minimal length in $W^M wW^M$.
For any Levi subgroup $M'$ we write $W(M,M')$ for the set of $w\in W$ of minimal length in $wW^M$ such that $wMw^{-1}=M'$.
We also write $W(M)=\cup_{M'}W(M,M')$.
Note that if $w\in W(M,M')$ then $w^{-1}\in W(M',M)$ and if $w_1\in W(M_1,M_2)$ and $w_2\in W(M_2,M_3)$ then $w_2w_1\in W(M_1,M_3)$.
In particular, $W(M,M)$ is a subgroup of $W$, which we can identify with $N_G(M)/M$.

For any Levi subgroups $M\subset L$ we denote by $w_M^L$ the element of maximal length in $W(M)\cap W^L$.
In particular, if $M=C_G(T)$ we simply write $w_0^L$.

\subsection{Some auxiliary results}

Let $(V,\norm{\cdot})$ be a Euclidean space and $R>0$.
We denote by $\clss_R(V)$ the space of continuous functions $f:V\rightarrow\C$ such that
$f(v)e^{R\norm{v}}$ is bounded. Clearly, $\clss_{R'}(V)\subseteq \clss_R(V)$ for $R<R'$.

\label{sec: P^r}
For any $r>0$ we denote by $P^r(V^*)$ the space of holomorphic functions $\phi$ on $\{\lambda\in V^*_{\C}:\norm{\Re\lambda}<r\}$ such that
\[
\sup_{\lambda\in V^*_{\C}:\norm{\Re\lambda}<r}\abs{\phi(\lambda)}(1+\norm{\lambda})^N<\infty,\ \ N=1,2,\dots
\]
Later on we will also use the notation $P^r(V^*;W)$ to denote the space of $W$-valued functions satisfying the condition above,
where $W$ is a finite-dimensional vector space. It is isomorphic to $P^r(V^*)\otimes W$.

\begin{lemma} \label{lem: rdequ}
The following conditions are equivalent for a smooth function $f:V \rightarrow\C$.
\begin{enumerate}
\item For all $r<R$ and a differential operator $D$ on $V$ with constant coefficients $Df\in\clss_r(V)$.
\item For all $r<R$ the function $f(v)e^{r\sqrt{1+\norm{v}^2}}$ is a Schwartz function on $V$.
\item The Fourier transform
\[
\hat f(\lambda)=\int_Vf(v)e^{\sprod{\lambda}v}\ dv
\]
of $f$ admits holomorphic continuation to $\{\lambda\in V^*_{\C}:\norm{\Re\lambda}<R\}$
and belongs to $\cap_{r<R}P^r(V^*)$.
\end{enumerate}
\end{lemma}

\begin{proof}
The equivalence of the first two conditions follows from the elementary fact that all the derivatives
of the function $\sqrt{1+x^2}$ are bounded.

If $\abs{f(v)}\le Ce^{-r\norm{v}}$ and $\norm{\Re\lambda}\le r'<r$ then
\[
\int_V\abs{f(v)e^{\sprod{\lambda}v}}\ dv\le C\int_Ve^{-(r-r')\norm{v}}\ dv<\infty.
\]
In particular, $\hat f$ is holomorphic for $\norm{\Re\lambda}<r$ and bounded for $\norm{\Re\lambda}\le r'$.
Moreover, if all derivatives of $f$ satisfy $\abs{Df(v)}\ll_{f,D} e^{-r\norm{v}}$ then for any $n$
\[
\abs{\hat f(\lambda)}(1+\norm{\lambda})^n
\]
is bounded for $\norm{\Re\lambda}\le r'$ (since $\widehat{Df}=\hat D\hat f$ and $\hat D$ is an arbitrary polynomial).
Thus, the first property implies the third.

Conversely, if $f$ satisfies the third condition then by Fourier inversion and shift of contour
\[
f(v)=\int_{\Re\lambda=\lambda_0}\hat f(\lambda)e^{-\sprod{\lambda}v}\ \abs{d\lambda}
\]
for a suitably chosen Haar measure and any $\lambda_0$ such that $\norm{\lambda_0}<R$. Let $r<R$. Taking $\lambda_0$ such that $\norm{\lambda_0}=r$
and $\sprod{\lambda_0}v=r\norm{v}$, and using the bounds on $\hat f$ we get
\[
\abs{f(v)}\ll_{r,f}e^{-r\norm{v}}.
\]
Similarly for the derivatives of $f$.
\end{proof}

Fix a parabolic subgroup $\grp P=\grp M\ltimes\grp U$ of $\grp G$.
For any $f\in\clss_R(\aaa_0^M)$ we define
\[
\theta^M_f(g)=\sum_{\gamma\in P_0\cap M\bs M}e^{\sprod{\rho_0}{\Ht_0(\gamma g)}}f(\Ht_0^M(\gamma g)),\ \ \ g\in\grp G(\A).
\]
Whenever convergent, $\theta^M_f:A_G\grp U(\A) M\bs \grp G(\A)\rightarrow \C$ is a right $K$-invariant function satisfying
$\theta^M_f(ag)=e^{\sprod{\rho_P}{\Ht_0(a)}}\theta^M_f(g)$, $a\in A_M$.
\begin{lemma} \label{lem: pseudobnd}
For $R$ sufficiently large, the sum defining $\theta^M_f$ is absolutely convergent for any $f\in\clss_R(\aaa_0^M)$.
Moreover, for any $N>1$ there exists $R$ and $N'$ such that for any $f\in\clss_R(\aaa_0^M)$ we have
\begin{equation} \label{eq: bndpseudo}
\sup_{m\in\siegel_M^1}\abs{\theta^M_f(mg)}\norm{m}^N\ll_{N,f}\norm{g}^{N'}, \ \ \ g\in\grp G(\A).
\end{equation}
\end{lemma}

\begin{proof}
The first part follows from \cite[Proposition II.1.10]{MR1361168} (and will also follow from the argument below).
The relation \eqref{eq: bndpseudo} is right-$K$-invariant in $g$, and therefore we may assume that $g\in\grp P(\A)$.
Write $g=m'u'$ with $m'\in\grp M(\A)$ and $u'\in\grp U(\A)$.
Since $\theta^M_f(mg)=\theta^M_f(mm')$ for all $m\in\grp M(\A)$
we may assume by \eqref{eq: mumult} that $g=m'\in\grp M(\A)$.
By a similar reasoning, using \eqref{eq: agmult} and \eqref{eq: bdexpnm} we may assume that $g=m'\in\grp M(\A)^1$,
in which case we will show that we can take $N'=N$.
Let $m\in\siegel_M^1$ and let $\gamma\in M$ and $m_1\in\siegel_M^1$ be such that $mm'=\gamma m_1$.
By \eqref{eq: gammag bigger} and \eqref{eq: submult} we have
\[
\norm{m}\ll\norm{\gamma^{-1}m}=\norm{m_1(m')^{-1}}\ll \norm{m_1}\norm{m'}.
\]
Therefore, since $\theta^M_f(mm')=\theta^M_f(m_1)$ it suffices to consider the case $g=e$, i.e.,
to show that for any $N\gg 1$ there exists $R$ such that
\[
\sup_{m\in\siegel_M^1}\abs{\theta^M_f(m)}\norm{m}^N<\infty
\]
for any $f\in\clss_R(\aaa_0^M)$.
This follows from the inequality
\[
\norm{\Ht_0^M(m)}\ll 1+\norm{\Ht_0^M(\gamma m)}, \ \ \gamma\in M, m\in\siegel_M^1
\]
(\cite[Lemma 2.1]{MR3156857}) and the fact that there exists $N_1$ such that
\[
\#\{\gamma\in P_0\cap M\bs M:\norm{\Ht_0^M(\gamma m)}\le X\}\ll (e^X+\norm{m})^{N_1},\ \ X\ge 0, m\in\siegel_M^1
\]
(an easy consequence of \cite[Lemma 5.1]{MR518111}) together with \eqref{eq: bdsieg}.
\end{proof}

The following standard lemma is a variant of \cite[Proposition II.1.10]{MR1361168}.
For convenience we include a proof.
\begin{lemma} \label{lem: bndpseudo}
For any $N>0$ there exists $R>0$ such that
\[
\sup_{g\in\siegel_G^1}\sum_{\gamma\in P\bs G}\abs{\phi(\gamma g)}\norm{g}^N<\infty
\]
and in particular,
\[
\sup_{g\in \grp G(\A)} \sum_{\gamma\in P\bs G}\abs{\phi(\gamma g)}<\infty,
\]
for any function $\phi$ on $A_G\grp U(\A)M\bs \grp G(\A)$ satisfying
\begin{equation} \label{eq: rdbnd}
\sup_{m\in\siegel_M^1,a\in A_M,k\in K}\modulus_P(a)^{-\frac12}\abs{\phi(amk)}\norm{m}^te^{R\norm{\Ht_P^G(a)}}<\infty,\ \ t=1,2,3\dots
\end{equation}
\end{lemma}

\begin{proof}
It is enough to prove the lemma for $N\gg 1$. Let $f(v)=e^{-R\norm{v}}$ for $v\in\aaa_0^G$.
The condition on $\phi$ together with \eqref{eq: bdexpnm} implies
\[
\abs{\phi(g)}\ll_{\phi,R}e^{\sprod{\rho_P}{\Ht_P(g)}}f(\Ht_0^G(g)), \ \ g\in\siegel_M K.
\]
It follows that
\[
\abs{\phi(g)}\ll_{\phi,R}\sum_{\gamma\in P_0\cap M\bs M}e^{\sprod{\rho_0}{\Ht_0(\gamma g)}}f(\Ht_0^G(\gamma g)), \ \ g\in \grp G(\A).
\]
Therefore
\[
\sum_{\gamma\in P\bs G}\abs{\phi(\gamma g)}\ll_{\phi,R}\sum_{\gamma\in P_0\bs G}e^{\sprod{\rho_0}{\Ht_0(\gamma g)}}f(\Ht_0^G(\gamma g)), \ \ g\in \grp G(\A).
\]
The lemma now follows from Lemma \ref{lem: pseudobnd} with $M=G$.
\end{proof}

Let $\AF_P(G)$ be the space of continuous functions $\varphi$ on $\grp U(\A)M\bs \grp G(\A)$ of moderate growth such that
$\varphi(ag)=e^{\sprod{\rho_P}{\Ht_0(a)}}\varphi(g)$ for all $a\in A_M$, $g\in \grp G(\A)$.
Denote by $\rdp_P(G)$ the subspace of $\AF_P(G)$ consisting of $\varphi$ such that for all $N>0$
\[
\sup_{m\in \siegel_M^1, k\in K}\abs{\varphi(mk)}\norm{m}^N<\infty.
\]
For instance, it follows from \cite[Lemma I.2.10]{MR1361168} that $\rdp_P(G)$ contains the space of smooth functions $\varphi\in\AF_P(G)$
of uniform moderate growth such that $m\mapsto\modulus_P(m)^{-\frac12}\varphi(mg)$ is a cuspidal function on $\AS{M}$ for all $g\in \grp G(\A)$.

For $\varphi\in\AF_P(G)$ and $\lambda\in \aaa_{M,\C}^*$ let
\[
\varphi_\lambda(g)=e^{\sprod{\lambda}{H_P(g)}}\varphi(g),\ g\in \grp G(\A).
\]
Let $w\in W(M)$ and let $\grp{P'}=\grp{M'}\ltimes \grp{U'}$ be the parabolic subgroup of $\grp G$ such that $\grp{M'}=w\grp Mw^{-1}$.
For any $\varphi\in\AF_P(G)$ and $\lambda\in\aaa_{M,\C}^*$ the integral
\[
M(w,\lambda)\varphi(g)=e^{\sprod{-w\lambda}{\Ht_{P'}(g)}}\int_{\grp{U'}(\A)\cap w\grp U(\A)w^{-1}\bs \grp{U'}(\A)}\varphi_\lambda(w^{-1}ug)\ du
\]
converges provided that $\Re\sprod{\lambda}{\alpha^\vee}\gg1$, $\alpha\in\srts_P$
(cf.~proof of \cite[Proposition II.1.6]{MR1361168}).

\label{sec: rddef}

For any $R>0$ let $\rd{R}$ be the space of continuous functions $\phi$ on $A_G\grp U(\A)M\bs \grp G(\A)$ satisfying
\eqref{eq: rdbnd} such that $\phi(\cdot g)$ is a cuspidal function on $M\bs \grp M(\A)$ for all $g\in\grp G(\A)$.

For $R\gg1$ and any $\phi\in\rd{R}$ define
\[
\theta_\phi(g)=\sum_{\gamma\in P\bs G}\phi(\gamma g)
\]
which converges by Lemma \ref{lem: bndpseudo}.
For any $\lambda\in\aaa_{M,\C}^*$ with $\norm{\Re\lambda}<R$ we write
\[
\phi[\lambda](g)=e^{-\sprod{\lambda}{\Ht_P(g)}}\int_{A_G\bs A_M}e^{-\sprod{\lambda+\rho_P}{\Ht_P(a)}}\phi(ag)\ da.
\]
We have $\phi[\lambda]\in\rdp_P(G)$.

Let $\rdsmth{R}$ be the smooth part of $\rd{R}$, i.e., the space of smooth functions $\phi$ on $\grp U(\A)M\bs \grp G(\A)$ such that
$X*\phi\in\rd{R}$ for all $X\in\univ(\mathfrak{g})$ (the universal enveloping algebra of the Lie algebra of $G$).
Let $\phi\in\rdsmth{R}$. Then
\begin{equation} \label{eq: invMeltrns}
\phi(g)=\int_{\lambda_0+\iii(\aaa_M^G)^*}\phi[\lambda]_\lambda(g)\ d\lambda
\end{equation}
for any $\lambda_0\in(\aaa_M^G)^*$ with $\norm{\lambda_0}<R$. Moreover, it easily follows from Lemma \ref{lem: rdequ} (or more precisely, its proof)
that for any $R'<R$ and $N>0$ we have
\begin{equation} \label{eq: bndonFT}
\sup_{m\in\siegel_M^1,k\in K,\lambda\in(\aaa_M^G)_{\C}^*:\norm{\Re\lambda}\le R'}\abs{\phi[\lambda](mk)}(\norm{m}+\norm{\lambda})^N<\infty.
\end{equation}
Thus, we may think of $\phi\in\rdsmth{R}$ as a holomorphic map on $\{\lambda\in(\aaa_M^G)_{\C}^*:\norm{\Re\lambda}<R\}$
with values in $\rdp_P(G)$ satisfying \eqref{eq: bndonFT}.

\subsection{Symplectic groups}

For $n\in \N$ let
\[
\grp{Sp}_n=\{g\in\grp{GL}_{2n}:{}^t g J_n g=J_n\}
\]
be the symplectic group of rank $n$ where
\[
J_n=\sm{0}{w_n}{-w_n}{0}
\]
and $w_n=(\delta_{i,n+1-j})\in\GL_n$ is the permutation matrix with ones on the non-principal diagonal.
Let ${}^*$ be the automorphism of $\grp{GL}_n$ given by $g\mapsto g^*=w_n {}^t g^{-1}w_n$.
The imbedding $g\mapsto \diag(g,g^*):\grp{GL}_n\rightarrow\grp{Sp}_n$ identifies $\grp{GL}_n$
with the Siegel Levi subgroup of $\grp{Sp}_n$.

Let $\grp B_n$ be the Borel subgroup of $\grp{Sp}_n$ consisting of upper triangular matrices.
It has a Levi decomposition $\grp B_n=\grp T_n\ltimes\grp N_n$ where $\grp T_n$ is the subgroup of diagonal matrices and
$\grp N_n$ is the subgroup of upper unitriangular matrices in $\grp B_n$.
The parabolic and Levi subgroups of $\grp{Sp}_n$ are parameterized by tuples of non-negative integers of the form $\gamma=(n_1,\dots,n_k;r)$ where
$k,r\ge0$, $n_1,\dots,n_k>0$, and $n_1+\cdots+n_k+r=n$.
Explicitly, to such $\gamma$ we associate the parabolic subgroup $\grp P=\grp P_\gamma=\grp M\ltimes \grp U$ consisting of block
upper triangular matrices in $\grp{Sp}_n$ where
\[
\grp M=\grp M_\gamma=\{\diag(g_1,\dots,g_k,h,g_k^*,\dots,g_1^*):h\in\grp{Sp}_r,\,g_i\in\grp{GL}_{n_i},\,i=1,\dots,k\}.
\]
In particular, $\grp{Sp}_n=\grp P_{(;n)}$ whereas $\grp P_{(n;0)}$ is the Siegel parabolic subgroup of $\grp{Sp}_n$.

We denote by
\[
\inj_\gamma=\inj_M:\grp{GL}_{n_1}\times\dots\times\grp{GL}_{n_k}\times\grp{Sp}_r\rightarrow\grp M
\]
the isomorphism defined by
\[
\inj_M(g_1,\dots,g_k;h)=\diag(g_1,\dots,g_k,h,g_k^*,\dots,g_1^*).
\]
If $r=0$ we simply write $\inj_M(g_1,\dots,g_k)$.
Also, if $M$ is clear from the context we will suppress it from the subscript.

Let $\delta_n=\diag(1,-1,1,\dots,(-1)^{n-1})\in\GL_n$ and $\epsilon_n=\diag(\delta_n,\delta_n^*)\in\Sp_n$.

\subsection{The setup} \label{sec: setup}
From now on, unless otherwise specified, we fix $n\in \N$ and let $\grp G=\grp{Sp}_{2n}$, $\epsilon=\epsilon_{2n}$ and
$\grp H=\grp{C_G}(\epsilon)\simeq\grp{Sp}_n\times\grp{Sp}_n$, the centralizer of $\epsilon$ in $\grp G$.
(Occasionally, we will also use $n$ as a running variable. Hopefully this will not cause any confusion.)
We identify $\grp G/\grp H$ with the $\grp G$-conjugacy class $\grp X$ of $\epsilon$,
a closed subvariety of $\grp G$, via $g \grp H\mapsto g\epsilon g^{-1}$.
Note that $X=G/H$ because the first Galois cohomology of $\grp H$ is trivial.

We take $\grp P_0$ to be the Borel subgroup $\grp B=\grp  B_{2n}=\grp T\ltimes \grp N$ where $\grp T=\grp T_{2n}$ and $\grp N=\grp N_{2n}$. Note that
\[
\grp T=\{\diag(a_1,\dots,a_{2n},a_{2n}^{-1},\dots,a_1^{-1}):a_1,\dots,a_{2n}\in\mathbb{G}_m\}
\]
and $\aaa_T^*$ is naturally identified with $\R^{2n}$.

Let $\gamma=(n_1,\dots,n_k;r)$ with $n_1+\cdots+n_k+r=2n$. For $\grp M=\grp M_\gamma$ the space $\aaa_M^*\simeq \R^k$ is imbedded in $\aaa_T^*\simeq \R^{2n}$
as elements of the form
\[
(\overbrace{\lambda_1,\dots,\lambda_1}^{n_1},\dots,\overbrace{\lambda_k,\dots,\lambda_k}^{n_k},\overbrace{0,\dots,0}^r),\ \ \ (\lambda_1,\dots,\lambda_k)\in \R^k.
\]

Under the identification $\aaa_T^*\simeq\R^{2n}$ we have $\Sigma=\{\pm e_i \pm e_j:1\le i,\,j\le 2n\}\setminus\{0\}$ where $\{e_i:1\le i\le 2n\}$
is the standard basis of $\R^{2n}$. Also, $\srts_0=\{\alpha_1,\dots,\alpha_{2n}\}$ where
$\alpha_i=e_i-e_{i+1}$, $i=1,\dots,2n-1$ (the short simple roots) and $\alpha_{2n}=2e_{2n}$ (the long simple root).

\section{Double cosets} \label{sec: double cosets}
In this section we study the double cosets $P\bs G/H$ for any parabolic subgroup $P$ of $G$.
Equivalently, $P\bs G/H$ parameterizes the $P$-orbits in $X$ under conjugation.
For $g\in G$ and a subgroup $\grp Q$ of $\grp G$ defined over $F$ we denote by $\conj{g}{Q}$ the $Q$-orbit of $g$ under conjugation
and by $\grp Q_g=\grp{C_Q}(g)$ the centralizer of $g$ in $\grp Q$.

Recall the following elementary result (e.g., \cite[Theorem 1]{MR0027764}).
\begin{lemma}\label{lem: symp inv}
For any involution $g\in G$ there exists a unique decomposition $2n=p+q$ such that
$g\in [\inj(I_p,-I_q)]_G$, i.e., every involution in $G$ is $G$-conjugate to $\inj(I_p,-I_q)$ for unique $p$ and $q$.
Thus, two involutions in $G$ which are conjugate in $\GL_{4n}$ are conjugate in $G$.
\end{lemma}

\subsection{Borel orbits}

We start with the case $P=B$.

\begin{lemma}(cf. \cite[Lemma 4.1]{MR803346})\label{eq: Borel orbit}
The map $\conj{x}{B}\mapsto N_G(T)\cap\conj{x}{B}$ defines a bijection between the $B$-orbits
in $X$ and the $T$-orbits in $N_G(T)\cap X$.
\end{lemma}

The crux of the matter is to show that $N_G(T)\cap\conj{x}{B}$ is not empty. This is proved in \cite[Lemma 4.1]{MR803346} in the case of an algebraically closed field.
However, the proof carries over verbatim to our case. See \cite[Lemma 4.1.1]{MR2010737} for more details.
The fact that $N_G(T)\cap\conj{x}{B}$ is a unique $T$-orbit follows from the uniqueness in the Bruhat decomposition as in \cite[Proposition 4.1.1]{MR2010737}.
That the map is bijective is now straightforward.

Recall that any involution $w\in W$ can be written in the form $w=s_{\beta_1}\cdots s_{\beta_r}$
where $\beta_1,\dots,\beta_r$ are pairwise orthogonal roots and $s_\beta\in W$ is the reflection associated to a root $\beta\in \Sigma$
(\cite{MR647211}). Moreover,
\begin{equation} \label{eq: r is determined}
\text{$r$ is determined by the conjugacy class of $w$.}
\end{equation}
Let
\[
\WR=\{w\in W: w\cap X \ne \emptyset\}
\]
so that
\[
N_G(T)\cap X=\coprod_{w\in\WR} w\cap X.
\]
Clearly $\WR$ is a union of conjugacy classes of involutions in $W$.
We can describe the set $\WR$ explicitly.

\begin{definition}
An involution $w\in W$ is called \emph{minimal} if there exists a Levi
subgroup $M$ of $G$ such that $w=w_0^M$ and $w\alpha=-\alpha$ for all $\alpha\in \srts_0^M$.
\end{definition}

Recall that every involution is conjugate to a minimal one (cf. \cite[Proposition 3.3]{MR803346}).

For any $k=0,\dots,n$ let $L_k=M_{(2^{(k)},1^{(2n-2k)};0)}$ be the Levi subgroup of semisimple rank $k$ such that
$\srts_0^{L_k}=\{\alpha_1,\alpha_3,\dots,\alpha_{2k-1}\}$. (Here $a^{(r)}$ is the $r$-tuple $(a,\dots,a)$.) Note that $w_0^{L_k}$ is a minimal involution.
Let $\WR_k$ be the conjugacy class of $w_0^{L_k}$ in $W$.

\begin{lemma}
We have
\begin{equation}\label{eq: wk}
\WR_k=\{s_{\beta_1}\cdots s_{\beta_k}: \beta_1,\dots,\beta_k \text{ are pairwise strongly orthogonal
\emph{short} roots}\}
\end{equation} and
\begin{equation}\label{part: conj of min}
\WR=\coprod_{k=0}^n\WR_k.
\end{equation}
Moreover, for any $w\in\WR_k$ there are $2(n-k)\choose n-k$ $T$-orbits in $w\cap X$.
\end{lemma}

\begin{proof}
Note that $w_0^{L_k}=s_{\alpha_1} s_{\alpha_3}\cdots s_{\alpha_{2k-1}}$ and therefore every element of $\WR_k$
is a product of reflections associated to pairwise strongly orthogonal short roots.
We show by a simple induction on $k$ that if $\beta_1,\,\dots,\beta_k$ are pairwise strongly orthogonal
short roots then $ s_{\beta_1}\cdots s_{\beta_k}$ is $W$-conjugate to $w_0^{L_k}$.
The case $k=1$ is immediate from the fact that $W$ acts transitively on the short roots.
For $k>1$, after conjugating we may assume without loss of generality that $\beta_1=\alpha_1$.
The imbedding $x\mapsto \inj(I_2;x):\Sp_{2n-2}\rightarrow \Sp_{2n}$ induces an imbedding of Weyl groups $W^{\Sp_{2n-2}}\hookrightarrow W$.
The image of this imbedding commutes with $s_{\alpha_1}$ and, by strong orthogonality, contains $s_{\beta_2},\dots,s_{\beta_k}$.
The claim therefore follows by induction on $n$.

This shows \eqref{eq: wk}.
The disjointness of the $\WR_k$'s follows from \eqref{eq: r is determined}.
To show \eqref{part: conj of min} it is enough to show that every minimal involution $w\in\WR$
is conjugate to $w_0^{L_k}$ for some $k=0,\dots,n$ and that $w_0^{L_k}\cap X\ne \emptyset$.

Let $L=M_{(n_1,\dots,n_k;r)}$ be a Levi subgroup of $G$ such that $w_0^L\in\WR$ is a minimal involution.
Note first that $r=0$ since otherwise we would have an involution in $\Sp_r$ whose non-zero entries are on the non-principal diagonal,
which is clearly impossible.
Thus $\inj_L(w_{n_1},\dots,w_{n_k})\in w_0^L$.
It easily follows from the property $w_0^L\alpha=-\alpha$, $\alpha\in \srts_0^L$ that $n_1,\dots,n_k\le  2$ and that $w_0^L$ is conjugate to $w_0^{L_k}$ where
$k=\#\{i=1,\dots,k:n_i=2\}$. This shows that $\WR \subseteq \coprod_{k=0}^n\WR_k$.

To show that $\coprod_{k=0}^n\WR_k \subseteq \WR$ it is enough to see that $w_0^{L_k}\cap X\ne\emptyset$.
Note that
\begin{multline*}
w_0^{L_k}\cap X=\{\inj_{L_k}(\sm0{a_1^{-1}}{a_1}0,\dots,\sm0{a_k^{-1}}{a_k}0,b_1,\dots,b_{2(n-k)}):\\ a_1,\dots,a_k\in F^*,\,b_1,\dots,b_{2(n-k)}=\pm 1,\,\#\{i: b_i=1\}=n-k\}.
\end{multline*}

Let $\alpha=\sm0110$. For any subset $A$ of $\{1,\dots,2(n-k)\}$ of size $n-k$ let
\[
b_i=\begin{cases}1 & i\in A \\-1 & i\not\in A\end{cases}
\]
and
\[
n_A=\inj_{L_k}(\overbrace{\alpha,\dots,\alpha}^k,b_1,\dots,b_{2(n-k)}).
\]
Then $n_A\in w_0^{L_k}\cap X$ and \eqref{part: conj of min} follows. Furthermore every $T$-orbit in $w_0^{L_k}\cap X$
contains $n_A$ for a unique subset $A$ of size $n-k$. The Lemma follows.
\end{proof}

\subsection{P-orbits}
Consider now the case of a general parabolic subgroup of $G$.

For any parabolic subgroup $\grp P$ of $\grp G$ with a given Levi decomposition $\grp P=\grp M\ltimes \grp U$ denote by $\pr_M:\grp P \rightarrow \grp M$
the projection to the Levi part of $\grp P$.
Assume for the rest of this section that $\grp P=\grp M\ltimes \grp U$ is a parabolic subgroup of $\grp G$.

Recall that for any two (not necessarily standard) parabolic subgroups $\grp Q_i$, $i=1,2$ of $\grp G$ with Levi decompositions
$\grp Q_i=\grp L_i \ltimes \grp V_i$, $\pr_{L_1}(\grp Q_1\cap \grp Q_2)$ is a parabolic subgroup of $\grp L_1$.
For $w\in   {}_MW_M$ let
\[
\grp P(w)=\pr_M(\grp P\cap w\grp Pw^{-1})=\grp M\cap w\grp Pw^{-1}.
\]
Then $\grp P\cap w\grp Pw^{-1}=\grp P(w)(\grp U\cap w\grp Pw^{-1})$ and $\grp P(w)$
is a standard parabolic subgroup of $\grp M$ with Levi decomposition
\[
\grp P(w)=\grp M(w)\ltimes \grp U(w)\text{ where }\grp M(w)=\grp M\cap w\grp Mw^{-1}\text{ and }\grp U(w)=\grp M\cap w\grp Uw^{-1}.
\]

By the Bruhat decomposition, for $g\in G$ there exists a unique element
$w\in {}_MW_M$ such that $PwP=PgP$. Let $p\in P$ be such that $g\in pwP$. Then
\[
\grp P\cap g\grp Pg^{-1}=p(\grp P\cap w\grp Pw^{-1})p^{-1}.
\]
It follows that
\begin{equation}\label{eq: projm}
\pr_M(\grp P\cap g\grp Pg^{-1})=\pr_M(p)\grp P(w)\pr_M(p)^{-1}.
\end{equation}
In particular, the following conditions are equivalent
\begin{enumerate}
\item $\pr_M(\grp P\cap g\grp Pg^{-1})=\grp M$,
\item $(\grp P\cap g\grp Pg^{-1})\grp U=\grp P$,
\item $\grp P(w)=\grp M$,
\item $\grp M(w)=\grp M$,
\item $\grp U(w)=1$,
\item $w\grp M\subseteq \grp{N_G(M)}$.
\end{enumerate}
If these conditions are satisfied we say that $g\in G$ is $M$-admissible.
This condition depends only on $PgP$.

\begin{lemma}\label{lem: char adm}
An element $g\in G$ is $M$-admissible if and only if $g\in UN_G(M)U$.
\end{lemma}

\begin{proof}
If $g\in N_G(M)$ then clearly $M\subset P\cap gPg^{-1}$ and therefore $\pr_M(P\cap gPg^{-1})=M$.
Since $M$-admissibility depends only on $PgP$ it follows that every element of $ UN_G(M)U$ is $M$-admissible.
Conversely, suppose that $g$ is $M$-admissible and let $w\in {}_MW_M$ be such that $PgP=PwP$.
Then $w\subseteq N_G(M)$. Let $u_1,\,u_2\in U$, $n\in w$ and $m_1,\,m_2\in M$ be such that $g=u_1m_1 nm_2u_2$. Then $g\in u_1 nMu_2$.
\end{proof}

\begin{lemma}(cf. \cite[Proposition 4.2.1]{MR2010737}) \label{lem: intersection}
Let $x\in X$ and let $w\in {}_MW_M$ be such that $PxP=PwP$.
Then $wM(w)\cap \conj{x}P$ is non-empty.
\end{lemma}

\begin{proof}
Since $w$ is reduced and $PwP=PxP=(PxP)^{-1}=(PwP)^{-1}$ it follows that $w^2=1$.
Let $w'\in W$ be an element of minimal length in the image of $\conj{x}P\cap N_G(T)$
(a non-empty set by Lemma \ref{eq: Borel orbit}) under the natural map $N_G(T)\rightarrow W$.
Then $w'$ is an involution such that $Pw'P=PwP$ and therefore there exists a reduced expression
$w'=w_1 w'' w w_2$ with $w_1^{-1},\,w_2\in W^M$ both left $M(w)$-reduced and $w''\in W^{M(w)}$.
Such a decomposition is unique.
Since both $w$ and $w'$ are involutions we also have $w'=w_2^{-1}w(w'')^{-1}w_1^{-1}$.
It follows from the uniqueness of the decomposition that $w_2=w_1^{-1}$.
Thus, $w''w$ is $W^M$-conjugate to $w'$ and hence from the definition of $w'$, $w''w$ also has a representative in $\conj{x}P \cap N_G(T)$.
The minimality of $w'$ and the fact that $w_1 w'' w w_2$ is a reduced decomposition implies that $w'=w''w$.
This shows that there exists $y\in \conj{x}P \cap M(w)w$ as required.
\end{proof}

\begin{lemma}(cf. \cite[Proposition 4.2.2]{MR2010737}) \label{lem: std rep}
Let $w\in {}_MW_M$ and $x\in wM(w)\cap X$.
Then $\grp U(w)$ is a normal subgroup of $\pr_M(\grp P_x)$ contained in $\pr_M(\grp R(x))$ where $\grp R(x)$
is the unipotent radical of $\grp P_x$.
\end{lemma}

\begin{proof}
As in the proof of Lemma \ref{lem: intersection} we have $w^2=1$.
Note that $\grp P_x\subseteq \grp P\cap x\grp Px^{-1}=\grp P\cap w\grp Pw^{-1}$ (since $x\in wM(w)$) and therefore $\pr_M(\grp P_x)\subseteq \grp P(w)$.
Since $\grp U(w)$ is normal in $\grp P(w)$ it is enough to show that $\grp U(w)\subseteq \pr_M(\grp R(x))$.

Note that $\grp P\cap x\grp Px^{-1}=\grp M(w)\ltimes \grp Z$ is a Levi decomposition where
$\grp Z=\grp U(w)(\grp U\cap w\grp Pw^{-1})=\grp U(w)(\grp U\cap w\grp Mw^{-1})(\grp U\cap w\grp Uw^{-1})$ and we have
$x\grp M(w)x=\grp M(w)$ and $x\grp Zx=\grp Z$.
It follows that $\grp P_x=\grp M(w)_x\ltimes \grp Z_x$ and that $\grp R(x)=\grp Z_x$.

Let $u\in \grp U(w)$ and let $v=xux$. Then $v\in \grp U\cap w\grp Pw^{-1}\subseteq \grp Z$ and (since $u\in \grp M$ and $v\in \grp U$) also $u^{-1}vu\in \grp U$.
Therefore the commutator $z:=[v^{-1},u^{-1}]\in \grp U$. Thus, $xzx=[u^{-1},v^{-1}]=z^{-1}\in \grp U$ and therefore $z\in \grp{U'}:=\grp U\cap w\grp Uw^{-1}$.
Thus, $z$ satisfies the cocycle condition $z\theta(z)=1$ with respect to the involution $\theta(g)=xgx$ on $\grp{U'}$.
Since $\grp{U'}$ is a unipotent group we have $H^1(\langle\theta\rangle,\grp{U'})=1$, i.e., $z$ must be a co-boundary.
There exists therefore $u'\in \grp{U'}$ such that $z=u' \theta(u')^{-1}$.
Note that this means that $v^{-1}u^{-1} xux u= u' x{u'}^{-1}x$, i.e., that
$uv u'\in \grp Z_x$. But $vu'\in \grp U$ and therefore $\pr_M(uvu')=u$. The Lemma follows.
\end{proof}

Let $x\in X$. Recall that $\pr_M(\grp P\cap x\grp Px^{-1})$ is a parabolic subgroup of $\grp M$. Let $\grp U(x)$ be its unipotent radical
and as before let $\grp R(x)$ be the unipotent radical of $\grp P_x$.

\begin{lemma} \label{lem: nonadmissible}
Let $x\in X$.
\begin{enumerate}
\item \label{part: rad} The kernel of $\pr_M:\grp P_x\rightarrow \grp M$ is contained in $\grp R(x)$;
\item \label{part: norm} $\grp U(x)$ is a normal subgroup of $\pr_M(\grp P_x)$ contained in $\pr_M(\grp R(x))$;
\item \label{part: 0int}Let $\chi$ be a character of $\grp P_x(\A)^1\bs \grp P_x (\A)$.
Then for every function $f: \grp U(\A)M\bs \grp P(\A)\to \C$ such that
\[
\int_{U(x) \bs \grp U(x)(\A)} f(up)\ du=0,\ p\in \grp P(\A)
\]
we have
\[
\int_{P_x\bs \grp P_x(\A)} f(p)\chi(p)\ dp=0
\]
(provided that the integral converges).
In particular,
\[
\int_{P_x\bs \grp P_x(\A)} f(p)\modulus_{P_x}^{-1}(p)\ dp=0.
\]
\end{enumerate}
\end{lemma}

\begin{proof}
Since the kernel of $\pr_M:\grp P_x\rightarrow \grp M$ is contained in $\grp U$, it is a unipotent normal subgroup of $\grp P_x$. Part \eqref{part: rad} follows.
Let $w\in {}_MW_M$ be such that $PxP=PwP$ and let $y\in \conj{x}P\cap M(w)w$ (which exists by Lemma
\ref{lem: intersection}). Let $p\in P$ be such that $x=pyp^{-1}$. Then
\[
\pr_M(\grp P_x)=\pr_M(p)\pr_M(\grp P_y)\pr_M(p)^{-1} \ \ \ \text{and}\ \ \ \grp R(x)=p\grp R(y)p^{-1}.
\]
Note further that $\grp P\cap x\grp Px^{-1}=p(\grp P\cap y\grp P y^{-1})p^{-1}$ and therefore
\[
\pr_M(\grp P\cap x\grp Px^{-1})=\pr_M(p)\pr_M(\grp P\cap y\grp P y^{-1})\pr_M(p)^{-1}.
\]
Part \eqref{part: norm}  therefore follows from
Lemma \ref{lem: std rep}.

Let $\grp S=\pr_M(\grp P_x)$. Clearly $\grp R(x)(\A)\subseteq \grp P_x(\A)^1$, $\grp R(x)$ being a unipotent group.
Therefore, by part \eqref{part: rad} $\chi$ induces a quasi-character $\delta:S\bs \grp S(\A)\to \C^*$.
By the invariance properties of $f$ we have a normalization of measures such that
\[
\int_{P_x\bs \grp P_x(\A)} f(p)\chi(p)\ dp=\int_{S\bs \grp S(\A)} f(s) \delta(s)\ ds.
\]
From part \eqref{part: norm} $\grp U(x)$ is a normal unipotent subgroup of $\grp S$. Therefore we have
\[
\int_{S\bs \grp S(\A)} f(s) \delta(s)\ ds=\int_{\grp U(x)(\A)S\bs \grp S(\A)}  \int_{\AS{U(x)}} f(us)\ du\ \delta(s)    \ ds.
\]
Part \eqref{part: 0int} follows.
\end{proof}

\begin{lemma}\label{lem: factor adm}
For $x\in N_G(M)\cap X$ we have $\grp P_x=\grp M_x\ltimes \grp U_x$.
\end{lemma}

\begin{proof}
This follows from the fact that
\[
	\grp P\cap x\grp Px^{-1}=\grp M\ltimes (\grp U\cap x\grp Ux^{-1})
\]
is a Levi decomposition which is invariant under conjugation by $x$
and $\grp P_x=(\grp P\cap x\grp Px^{-1})_x$.
\end{proof}

We now describe explicitly the $M$-admissible $P$-orbits in $X$.

\begin{lemma}\label{lem: M orbit map}
The map $\conj{x}P\mapsto\conj{x}P\cap N_G(M)$ is a bijection between the $M$-admissible $P$-orbits in $X$
and the $M$-orbits in $N_G(M)\cap X$.
\end{lemma}

\begin{proof}
Let $x\in X$ be $M$-admissible and let $w\in {}_MW_M$ be such that $PxP=PwP$.
From the definition of $M$-admissibility it follows that $wM\subseteq N_G(M)$.
It also follows from Lemma \ref{lem: intersection} that without loss of generality we may assume $x\in wM$.
Let $y\in \conj{x}P\cap N_G(M)$. Recall that $N_G(M)$ is the disjoint union of $M\sigma$ over all $\sigma\in {}_MW_M$ such that
$\sigma M\sigma^{-1}=M$ and $G$ is the disjoint union of $P\sigma P$ over all $\sigma\in {}_MW_M$.
Since $\conj{y}P=\conj{x}P\subseteq PwP$, it follows that $y\in PwP\cap N_G(M)=Mw$, i.e., $Mx=My$.
Let $n\in \conj xP\cap N_G(T)$. Then since $PwP\cap N_G(T)\subseteq Mw$ we see that $n\in Mw$, i.e., $xn^{-1},\,yn^{-1}\in M$.

Let $p=mu\in P$ be such that $x=pyp^{-1}$, with $m\in M$ and $u\in U$.
Then
\begin{equation}\label{eq: over P}
xn^{-1}=p(yn^{-1})(np^{-1}n^{-1}).
\end{equation}
In particular, $np^{-1}n^{-1}\in P$ and therefore $p\in P\cap n^{-1}Pn=M(U\cap n^{-1}Un)$, i.e. $u\in U\cap n^{-1}Un$.
It follows that $\pr_M(np^{-1}n^{-1})=nm^{-1}n^{-1}$ and therefore
applying $\pr_M$ to \eqref{eq: over P} we get
\[
xn^{-1}=m(yn^{-1})(nm^{-1}n^{-1}).
\]
Therefore $x=mym^{-1}$. This shows that $\conj xP\mapsto \conj xP\cap N_G(M)$ is a well-defined map from $M$-admissible $P$-orbits
in $X$ to $M$-orbits in $N_G(M)\cap X$. It is clearly injective, and it is surjective from the definition of $M$-admissibility.
\end{proof}

Next we analyze the $M$-orbits in $N_G(M)\cap X$.
Recall that $W(M,M)$ is a subgroup of $W$ which can be identified with $N_G(M)/M$.
We denote the resulting isomorphism by $\imb_M: N_G(M)/M\simeq W(M,M)\hookrightarrow W$ (which we also view as a homomorphism
$N_G(M)\rightarrow W$).

The following definitions are given in \cite{MR2010737}.
\begin{definition}\label{def: adm min}
We denote by $W(M,M)_2$ the set of involutions in $W(M,M)$.
An element $w\in W(M,M)_2$ is $M$-minimal if it is of the form $w_M^L$ for some Levi subgroup $\grp L$ containing $\grp M$
and $w_M^L$ acts as $-1$ on $\aaa_M^L$.
\end{definition}

\begin{lemma}\label{lem: conj min} (\cite[Corollary 3.3.1]{MR2010737}).
For every $w\in W(M,M)_2$ there exists a Levi subgroup $M'$ and $\sigma\in W(M,M')$ such that $\sigma w\sigma^{-1}\in W(M',M')_2$ is $M'$-minimal.
\end{lemma}

Let
\[
\WR_M=\{w\in W(M,M):wM\cap X\ne\emptyset\}\subset W(M,M)_2
\]
so that
\begin{equation}\label{eq: w union}
N_G(M)\cap X=\coprod_{w\in \WR_M}wM \cap X.
\end{equation}
In other words $\WR_M=\imb_M(N_G(M)\cap X)$.

It is immediate from Lemma \ref{lem: M orbit map} that $x\in X$ is $M$-admissible if and only if $PxP=PwP$ where $w\in\WR_M$.

It is also clear that
\begin{equation}\label{eq: conj adm}
\text{if $w\in \WR_M$ and $\sigma\in W(M,M')$ then $\sigma w\sigma^{-1}\in \WR_{M'}$}.
\end{equation}


\begin{definition}\label{def: std basic}
A pair $(M,L)$ of Levi subgroups with $M\subset L$ is called \emph{standard \rlvt} if $M$
and $L$ are of the form $M=M_{(r_1,r_1,\dots,r_k,r_k,s_1,\dots,s_l,t_1,\dots,t_m;u)}$ and $L=M_{(2r_1,\dots,2r_k,s_1,\dots,s_l;v)}$
(with $k$, $l$, $m$, $u$ or $v$ possibly zero) where $t_1,\dots,t_m$ are even and $v=u+t_1+\dots+t_m$. Thus,
\[
M\simeq\overbrace{\GL_{r_1}\times\GL_{r_1}\times\dots\times\GL_{r_k}\times\GL_{r_k}}^{M_1}
\times\overbrace{\GL_{s_1}\times\dots\times\GL_{s_l}}^{M_2}\times\overbrace{\GL_{t_1}\times\dots\times\GL_{t_m}\times\Sp_u}^{M_3},
\]
and
\[
L\simeq\overbrace{\GL_{2r_1}\times\dots\times\GL_{2r_k}}^{L_1}\times\overbrace{\GL_{s_1}\times\dots\times\GL_{s_l}}^{L_2}\times\overbrace{\Sp_v}^{L_3}
\]
with $M_1\subset L_1$, $M_2=L_2$, $M_3\subset L_3$.
\end{definition}

For instance, the standard \rlvt\ pairs $(M,L)$ with $M=T$ are $(T,M_{(2^{(k)},1^{(2n-2k)};0)})$, $k=0,\dots,n$.

More generally, a pair $(M,L)$ consisting of a Levi subgroup $M$ and a \emph{semistandard} Levi subgroup $L$ containing $M$
is \rlvt\ if there exists $w\in W(M)$ such that $(wMw^{-1},wLw^{-1})$ is a standard \rlvt\ pair.

\begin{lemma}\label{lem: min is basic}
Let $M\subseteq L$ be Levi subgroups of $G$.
\begin{enumerate}
\item \label{part: min basic}
Assume that $w_M^L\in \WR_M$ is an $M$-minimal involution. Then there exists $\sigma\in W(L)\cap W^{M_{(2n;0)}}$ such that
$(\sigma M\sigma^{-1},\sigma L\sigma^{-1})$ is a standard \rlvt\ pair. In particular, $(M,L)$ is a \rlvt\ pair.
\item\label{part: basic min} If $(M,L)$ is a standard \rlvt\ pair then $w_M^L\in \WR_M$ is an $M$-minimal involution.
\end{enumerate}
\end{lemma}

\begin{proof}
For the first part, assume that $w_M^L\in \WR_M$ is $M$-minimal. Write
\[
M=M_{(n_1,\dots,n_a;u)}\simeq\GL_{n_1} \times \cdots \times\GL_{n_a}\times\Sp_u
\]
and
\[
L=M_{(m_1,\dots,m_b;v)}\simeq\GL_{m_1} \times \cdots \times\GL_{m_b}\times\Sp_v.
\]
The inclusion $M\subseteq L$ implies that $\{1,2,\dots,a\}$ can be partitioned into sets $S_1,\dots,S_b,R$ such that $s_i < s_{i+1}<r$
for every $i\le b-1$, $s_i\in S_i$, $s_{i+1}\in S_{i+1}$ and $r\in R$,
$\sum_{s\in S_i} n_s=m_i$ and $v=u+\sum_{r\in R} n_r$.
If $\abs{S_i}>2$ for some $i$ then it is easily observed that $w_M^L$ does not act as $-1$ on $\aaa_M^L$.
Therefore the $M$-minimality of $w_M^L$ implies that $\abs{S_i}\le 2$ for all $i=1,\dots,b$.
Note further that the $M$-admissibility of $w_M^L$ implies that if $S_i=\{s,s+1\}$ then $n_s=n_{s+1}$.
To any permutation $\sigma$ of $\{1,\dots,b\}$ corresponds a unique element (that we still denote by $\sigma$) of $W(L)\cap W^{M_{(2n;0)}}$ such that
$\sigma L\sigma^{-1}=M_{(m_{\sigma^{-1}(1)},\dots,m_{\sigma^{-1}(b)};v)}$.
If $k$ is the number of indices $i$ such that $\abs{S_i}=2$ let $\sigma$ be such that
$\abs{S_{\sigma^{-1}(i)}}=2$ if and only if $i=1,\dots,k$.
Then $\sigma w_M^L \sigma^{-1}=w_{M'}^{L'}$ where
$(M',L')=(\sigma M\sigma^{-1},\sigma L\sigma^{-1})$ has the form
\[
M'=\overbrace{\GL_{r_1}\times\GL_{r_1}\times\dots\times\GL_{r_k}\times\GL_{r_k}}^{M_1}
\times\overbrace{\GL_{s_1}\times\dots\times\GL_{s_l}}^{M_2}\times\overbrace{\GL_{t_1}\times\dots\times\GL_{t_m}\times\Sp_u}^{M_3},
\]
and
\[
L'=\overbrace{\GL_{2r_1}\times\dots\times\GL_{2r_k}}^{L_1}\times\overbrace{\GL_{s_1}\times\dots\times\GL_{s_l}}^{L_2}\times\overbrace{\Sp_v}^{L_3}.
\]

Assume now further that $w_M^L\in \WR_M$. It follows from \eqref{eq: conj adm} that $w_{M'}^{L'}\in \WR_{M'}$.
Recall the elements $\delta_n$ of $\GL_n$ defined in \S\ref{sec: notation}.
Write $\delta_{s_1+\cdots+s_l+u}=\diag(\gamma_1,\dots,\gamma_l,\gamma)$ where $\gamma_i\in\{\pm \delta_{s_i}\}$
and $\gamma\in \{\pm \delta_u\}$ and let $\beta=\diag(\gamma,\gamma^*)\in\Sp_u$ (in fact $\beta\in\{\pm\epsilon_u\}$).
Let
\[
n_0=\inj_{L'}(\sm0{I_{r_1}}{I_{r_1}}0,\dots,\sm0{I_{r_k}}{I_{r_k}}0,\gamma_1,\dots,\gamma_l;n_1)
\]
where
\[
n_1=\left(\begin{smallmatrix}
 & & & & & & \delta_{t_1} \\ & & & & & \iddots & \\ & & & & \delta_{t_m} & &  \\ & & & \beta & & & \\ & & -\delta_{t_m}^* & & & & \\ & \iddots & & & & &
 \\ -\delta_{t_1}^* & & & & & &
\end{smallmatrix}\right).
\]
Then $n_0\in w_{M'}^{L'}$. Therefore, there exists $m=\inj_{M'}(a_1,b_1,\dots,a_k,b_k,c_1,\dots,c_l,d_1,\dots,d_m;h)\in M'$
(with $a_i,\,b_i\in\GL_{r_i}$, $c_i\in\GL_{s_i}$, $d_i\in \GL_{t_i}$ and $h\in \Sp_u$) such that $mn_0\in X$. In particular, $(mn_0)^2=I_{4n}$ and therefore
\[
\sm0{d_i \delta_{t_i}}{-d_i^*\delta_{t_i}^*}0^2=I_{2t_i},
\]
i.e., $-(d_i\delta_{t_i})(d_i\delta_{t_i})^*=I_{t_i}$.
In other words $w_{t_i}d_i\delta_{t_i}$ is a non-degenerate skew-symmetric matrix and therefore $t_i$ is even.
This shows that $(M',L')$ is a standard \rlvt\ pair. Part \eqref{part: min basic} follows.

For the second part, it suffices to note that $n_0\in X$. This follows for instance from Lemma \ref{lem: symp inv}
and the fact that as an element of $\GL_{4n}$, the dimensions of the $\pm1$-eigenspaces of $n_0$ coincide.
\end{proof}

\begin{lemma}\label{lem: min M}
We have
\[
\WR_M=\{\sigma w_{M'}^{L'}\sigma^{-1}: (M',L')\text{ is a standard \rlvt\ pair}\text{ and } \sigma\in W(M',M)\}.
\]
\end{lemma}

\begin{proof}
Let $w\in \WR_M$. It follows from Lemma \ref{lem: conj min} that $w_1=\sigma_1 w\sigma_1^{-1}$ is $M_1$-minimal for some $\sigma_1\in W(M)$ where
$M_1=\sigma_1 M \sigma_1^{-1}$ and from \eqref{eq: conj adm} that $w_1\in \WR_{M_1}$.
It now follows from Lemma \ref{lem: min is basic} \eqref{part: min basic} that $\sigma_2 w_1\sigma_2^{-1}$ is of the form $w_{M'}^{L'}$ for some
$\sigma_2\in W(M_1)$ and a standard \rlvt\ pair $(M',L')$. It follows that $w=\sigma w_{M'}^{L'}\sigma^{-1}$ where
$\sigma=\sigma_1^{-1}\sigma_2^{-1}\in W(M',M)$.

The other inclusion follows from \eqref{eq: conj adm} and \ref{lem: min is basic} \eqref{part: basic min}.
\end{proof}

Let $\grp M$ be a Levi subgroup of $\grp G$ and $x\in N_G(M)\cap X$. The group $N_G(M)/M$ acts on $\aaa_M^*$ and in particular,
$x$ acts as an involution on $\aaa_M^*$ and decomposes it into a direct sum of the $\pm 1$-eigenspaces which we
denote by $(\aaa_M^*)_x^\pm$. (A similar decomposition applies to the dual space $\aaa_M=(\aaa_M)_x^+\dsum (\aaa_M)_x^-$.)
For any such $x$ let $\grp L=\grp L(x)$ be the intersection of all \emph{semistandard} Levi subgroups containing
$M$ and $x$. Then $\grp L$ is a semistandard Levi subgroup and we have $(\aaa_M^*)_x^+=\aaa_L^*$, or equivalently, $(\aaa_M^*)_x^-=(\aaa_M^L)^*$ (cf.~ \cite[p. 1299]{MR681738}).

\begin{definition}\label{def: stdx}
With the above notation we say that $x$ is $M$-minimal if $L(x)$ is standard.
Similarly, we say that $x$ is $M$-standard \rlvt\ if the pair $(M,L(x))$ is standard \rlvt\ (see Definition \ref{def: std basic}).
\end{definition}

\begin{remark}\label{rmk: Mmin}
If $w=\imb_M(x)\in \WR_M$ then $L(x)$ and the above decomposition of $\aaa_M^*$ depend only on $w$. Furthermore, $x$ is $M$-minimal if and only if
$w$ is an $M$-minimal involution, in which case $w=w_M^{L(x)}$.
\end{remark}

\begin{corollary}\label{cor: xadm}
Let $M$ be a Levi subgroup of $G$ and $x\in N_G(M)\cap X$. Then there exists $n\in N_G(T)$ such that $nMn^{-1}$ is a Levi subgroup of $G$,
$nxn^{-1}$ is $nMn^{-1}$-standard \rlvt\ and $L(nxn^{-1})=nL(x)n^{-1}$.
\end{corollary}

\begin{proof}
Let $w=\imb_M(x)\in \WR_M$ and let $(M',L')$ be a standard \rlvt\ pair and $\sigma\in W(M,M')$ be such that $w=\sigma^{-1}w_{M'}^{L'}\sigma$, as in
Lemma \ref{lem: min M}.
Let $n\in \sigma$ and set $x'=nxn^{-1}$. By definition $x'$ is $M'=nMn^{-1}$-standard \rlvt. Note that
$(\aaa_M)_x^{+}=\sigma((\aaa_{M'})_{x'}^{+})$,
$(\aaa_M)_x^{-}=\sigma((\aaa_{M'})_{x'}^{-})$
and therefore also $L(x')=\sigma L(x)\sigma^{-1}=n L(x)n^{-1}$.
\end{proof}

Recall the notation \eqref{def: AGH}.
\begin{lemma}\label{lem: exp iso}
For every $x\in N_G(M)\cap  X$ the restriction of $\Ht_M$ to $\grp M_x(\A)$ defines a surjective homomorphism
\[
\Ht_M:\grp M_x(\A)\rightarrow(\aaa_M)_x^+.
\]
Moreover, the restriction of $\Ht_M$ to $A_M^{M_x}$ defines an isomorphism
\[
\Ht_M:A_M^{M_x}\rightarrow(\aaa_M)_x^+.
\]
\end{lemma}

\begin{proof}
The second part follows from the fact that $xe^{\nu}x^{-1}=e^{x\nu}$ for any $\nu\in\aaa_M$.
The first part follows from the second part and the fact that $\Ht_M(\grp M_x(\A))\subset (\aaa_M)_x^+$, since
$\Ht_M(xmx^{-1})=x\Ht_M(m)$ for any $m\in \grp M(\A)$.
\end{proof}

In view of Lemma \ref{lem: exp iso}, for any $x\in N_G(M)\cap X$ let $\rho_x\in (\aaa_M^*)_x^+$ be the unique element such that
\begin{equation}\label{eq: rhox}
e^{\sprod{\rho_x}{H_M(a)}}=\modulus_{P_x}(a)\modulus_P(a)^{-\frac12} \ \text{ or equivalently}\ \ \
\modulus_{P_x}(a)=e^{\sprod{\rho_x+\rho_P}{H_M(a)}}, \ \ \ a\in A_M^{M_x}.
\end{equation}
Note that $\rho_x$ depends only on $[x]_M$.

\begin{remark}
The vector $\rho_x$ (with a slightly different convention) was encountered in the setup of \cite{MR2254544}.
It does not show up in the cases considered in \cite{MR2010737} by [ibid., Proposition 4.3.2].
Note that in our case $\modulus_{P_x}$ is non-trivial, in general, on $\grp M_x(\A)\cap\grp M(\A)^1$. In other words we will \emph{not} necessarily have
$\modulus_{P_x}(m)=e^{\sprod{\rho_x+\rho_P}{H_M(m)}}$ for all $m\in \grp M_x(\A)$.
This is in contrast with the cases considered in \cite{MR2010737} and \cite{MR2254544} where $\grp M_x(\A)\cap\grp M(\A)^1=\grp M_x(\A)^1$.
\end{remark}

\begin{lemma}\label{lem: stab vx}
Suppose that $x\in  N_G(M)\cap X$ is $M$-minimal and let $\grp L=\grp L(x)$. Let $\grp Q=\grp L\ltimes \grp V$ be the parabolic subgroup
of $\grp G$ with Levi subgroup $\grp L$.
Then $\grp U_x=\grp V_x$ and therefore $\modulus_{Q_x}\rest_{\grp P_x(\A)}=\modulus_{P_x}$.
\end{lemma}

\begin{proof}
Let $\grp{U_L}=\grp L\cap \grp U$ be the unipotent radical of the parabolic subgroup $\grp{P_L}:=\grp L\cap \grp P$ of $\grp L$ (with Levi subgroup $\grp M$).
Then $\grp U=\grp{U_L}\ltimes \grp V$. Note that $x\in L$ and therefore $x\grp Vx^{-1}=\grp V$. On the other hand, since $x\in w_M^L M$ we have
$x \grp{U_L} x^{-1}=\grp{U_L}^t$ (the image of $\grp{U_L}$ under transpose). It follows that if $u=u_1u_2\in \grp U_x$ with $u_1\in \grp{U_L}$ and $u_2\in \grp V$ then
$xu_1x^{-1}=uxu_2^{-1}x^{-1}\in \grp{U_L}^t\cap \grp U=1$ and therefore $u_1=e$ and $u\in \grp V_x$. Thus, $\grp U_x=\grp V_x$.
By Lemma \ref{lem: factor adm} we now have $\grp P_x=\grp M_x\ltimes \grp V_x$ whereas $\grp Q_x=\grp L_x\ltimes \grp V_x$. The rest of the Lemma follows.
\end{proof}

\subsection{Orbit representatives}\label{subsec: orbits}
Our purpose here is to give an explicit description of the fibers of $M$-orbits $[x]_M$ with $x\in wM\cap X$
that lie over an $M$-minimal involution $w\in\WR_M$. For each such orbit we choose a convenient representative $x$ and explicate the centralizer $M_x$.

Suppose first that $(M,L)$ is a standard \rlvt\ pair and use the notation in Definition \ref{def: std basic}.
Let $x\in w_M^L M\cap X \subseteq L$ and write
\[
x=\inj(x_1,x_2;x_3)
\]
where
\[
x_1=\diag(\sm{0}{y_1}{y_1^{-1}}{0},\dots,\sm{0}{y_k}{y_k^{-1}}{0})\in L_1,
\]
with $y_i\in\GL_{r_i}$,
\[
x_2=\diag(z_1,\dots,z_l)\in L_2
\]
with $z_i\in\GL_{s_i}$ which is $\GL_{s_i}$-conjugate to $\diag(I_{p_i},-I_{q_i})$ for some decomposition $p_i+q_i=s_i$ and
$x_3\in L_3$ is of the form
\[
x_3=\left(\begin{smallmatrix}
 & & & & & & a_1 \\ & & & & & \iddots & \\ & & & & a_m & &  \\ & & & h & & & \\ & & -a_m^* & & & & \\ & \iddots & & & & &
 \\ -a_1^* & & & & & &
\end{smallmatrix}\right)
\]
with $a_i\in\GL_{t_i}$ such that $w_{t_i}a_i$ is anti-symmetric and $h\in\Sp_u$ is an involution.
We have
\[
M_x=\inj((M_1)_{x_1},(M_2)_{x_2};(M_3)_{x_3}).
\]
Note that
$(M_1)_{x_1}$ is the product of $\GL_{r_i}$, $i=1,\dots,k$ embedded in $\GL_{r_i}\times\GL_{r_i}$;
$(M_2)_{x_2}$ is the product of centralizers of involutions in $\GL_{s_i}$, $i=1,\dots,l$;
$(M_3)_{x_3}$ is the product of symplectic groups in $\GL_{t_i}$, $i=1,\dots,m$
and a centralizer of an involution in $\Sp_u$.
More explicitly,
\[
(M_1)_{x_1}=\{\diag(g_1,y_1^{-1}g_1y_1,\dots, g_k,y_k^{-1}g_ky_k):g_i\in\GL_{r_i}\}.
\]
Note that after conjugation by an element of $M_1$ we may assume that $y_i=I_{r_i}$, $i=1,\dots,k$.
Similarly,
\[
(M_2)_{x_2}=\diag(C_{\GL_{s_1}}(z_1),\dots,C_{\GL_{s_l}}(z_l)).
\]
After conjugation in $M_2$ we may assume that $z_i=\diag(I_{p_i},-I_{q_i})$ and then
\[
C_{\GL_{s_i}}(z_i)=\GL_{p_i} \times\GL_{q_i}.
\]
Finally,
\[
(M_3)_{x_3}=\inj(\Sp(w_{t_1}a_1),\dots,\Sp(w_{t_m}a_m);C_{\Sp_u}(h)).
\]
After conjugation in $M_3$ we may assume that $a_i=\diag(I_{t_i/2},-I_{t_i/2})$, i.e., $w_{t_i}a_i=J_{t_i/2}$
and by Lemma \ref{lem: symp inv} that $h=\inj(I_p,-I_q)$ for some decomposition $u=p+q$ such that (since $x\in[\epsilon]_G$)
\begin{equation}\label{eq: signature cond}
p+\sum_{i=1}^l p_i=q+\sum_{i=1}^l q_i.
\end{equation}

To summarize, for $x\in w_M^L M\cap X$ (or $M$-orbit $[x]_M\subseteq w_M^L M\cap X$) we associate the data
\begin{equation}\label{eq: pq data}
\p=(p_1,q_1,\dots,p_l,q_l;p,q)
\end{equation}
 satisfying \eqref{eq: signature cond} and such that $u=p+q$ and $s_i=p_i+q_i$, $i=1,\dots,l$.
We further choose a convenient representative $x_\p\in [x]_M$ as follows.
For integers $s,p,q$ write $\dig_{2s}=\sm{I_s}{}{}{-I_s}$ and $\dig_{p,q}=\left(\begin{smallmatrix}I_p&&\\&-I_{2q}&\\&&I_p\end{smallmatrix}\right)$.
For $\alpha=(r_1,\dots,r_k)$ let $x_\alpha=\diag(\sm0{I_{r_1}}{I_{r_1}}0,\dots,\sm0{I_{r_k}}{I_{r_k}}0)$.
For $\beta=(p_1,q_1,\dots,p_l,q_l)$ let $y_\beta=\diag(I_{p_1},-I_{q_1},\dots,I_{p_l},-I_{q_l})$.
For $\gamma=(t_1,\dots,t_m;p,q)$ with $t_i$ even let
\begin{equation} \label{eq: zgamma}
z_\gamma=
\left(\begin{smallmatrix}
 & & & & & & \dig_{t_1} \\ & & & & & \iddots & \\ & & & & \dig_{t_m} & &  \\ & & & \dig_{p,q} & & & \\
  & & \dig_{t_m} & & & & \\ & \iddots & & & & &
 \\ \dig_{t_1} & & & & & &
\end{smallmatrix}\right).
\end{equation}
We set
\[
x_\p=\inj(x_\alpha,y_\beta;z_\gamma).
\]

Let
\[
H_\gamma=(M_3)_{z_\gamma}=\inj_{M_3}(\Sp_{t_1},\dots,\Sp_{t_m};H_{p,q})
\]
where $H_{p,q}=C_{\Sp_u}(d_{p,q})\simeq\Sp_p \times\Sp_q \hookrightarrow\Sp_u$.
We have
\[
M_{x_\p}=\inj_L(\GL_{r_1}^{\triangle},\dots,\GL_{r_k}^{\triangle},\GL_{p_1} \times\GL_{q_1},\dots,\GL_{p_l} \times\GL_{q_l};H_\gamma)
\]
where $\GL_r^{\triangle}=\{\diag(g,g): g\in \GL_r\}$ and
\[
L_{x_\p}=\inj_L(\inn_{r_1},\dots,\inn_{r_k},\GL_{p_1} \times\GL_{q_1},\dots,\GL_{p_l} \times\GL_{q_l};(\Sp_v)_{z_\gamma})
\]
where $\inn_r=\{\sm{a}{b}{b}{a}\in\GL_{2r}: a,\,b\in \operatorname{Mat}_{r\times r}\}$.
Note that for
\begin{equation} \label{def: eta}
\eta_r:=\sm{I_r}{I_r}{I_r}{-I_r}
\end{equation}
we have
\[
\eta_r^{-1}\inn_r\eta_r=\{\sm{g_1}{}{}{g_2}:g_1,\,g_2\in\GL_r\}.
\]

Consider now, more generally, a Levi subgroup $M'$ of $G$ and $w\in \WR_{M'}$ an $M'$-minimal involution.
Let $L'$ be the Levi subgroup containing $M'$ such that $w=w_{M'}^{L'}$. By Lemma \ref{lem: min is basic} \eqref{part: min basic} there exists
$\sigma\in W(L')\cap W^{M_{(2n;0)}}$ such that
$(M,L)=(\sigma M' \sigma^{-1},\sigma L' \sigma^{-1})$ is a standard \rlvt\ pair.
For $(M,L)$ we use the notation of Definition \ref{def: std basic}.
Let $b=l+k$ and $(m_1,\dots,m_b)=(2r_1,\dots,2r_k,s_1,\dots,s_l)$ so that $L=M_{(m_1,\dots,m_b;v)}$.
As in the proof of Lemma \ref{lem: min is basic} we may view $\sigma$ as a permutation of $\{1,\dots,b\}$ and $L'=M_{(m_{\sigma^{-1}(1)},\dots,m_{\sigma^{-1}(b)};v)}$.

Let $m=m_1+\cdots+m_b$. There is a unique permutation matrix $n_0\in\GL_m$ such that
$n_0\diag(g_{\sigma^{-1}(1)},\dots,g_{\sigma^{-1}(b)})n_0^{-1}=\diag(g_1,\dots,g_b)$ whenever $g_i\in\GL_{m_i}$, $i=1,\dots,b$.
Then $n=\iota(n_0;I_{2v})\in \sigma$ and therefore $(M,L)=(nM'n^{-1},nL'n^{-1})$.
Furthermore, the map $x\mapsto x'=n^{-1}xn: w_M^L M\cap X\to w_{M'}^{L'} M'\cap X$ is a bijection that maps $[x]_M$ to $[x']_{M'}$ and $M'_{x'}=n^{-1} M_x n$.

Let $x'_\p=n^{-1}x_\p n$.
Then the different $M'$-orbits in $w_{M'}^{L'}M'\cap X$ are precisely $[x'_\p]_{M'}$ for data $\p$ as in \eqref{eq: pq data} satisfying \eqref{eq: signature cond}.
We further have
\begin{equation}\label{eq: stabx'}
M'_{x'_\p}=\inj_{L'}(A_1,\dots,A_b;H_\gamma)
\end{equation}
where $A_i$ is the subgroup of $\GL_{m_{\sigma^{-1}(i)}}$ given by
\[
A_{\sigma(i)}=\begin{cases} \GL_{r_i}^{\triangle} & i=1,\dots,k \\\GL_{p_{i-k}}\times\GL_{q_{i-k}} & i=k+1,\dots,b \end{cases}
\]
and
\[
L'_{x_\p'}=\inj_{L'}(\inn_1',\dots,\inn_b';(\Sp_v)_{z_\gamma})
\]
where
\[
\inn_{\sigma(i)}'=\begin{cases} \inn_{r_i} & i=1,\dots,k \\\GL_{p_{i-k}}\times\GL_{q_{i-k}} & i=k+1,\dots,b.\end{cases}
\]

\subsection{Exponents} \label{sec: exponents}
Let $M$ be a Levi subgroup of $G$ and $x\in N_G(M)\cap X$. Recall that $P_x=M_x\ltimes U_x$ (Lemma \ref{lem: factor adm}).
We study the modulus function $\modulus_{P_x}$.

Let $M'$ be a Levi subgroup of $G$, $x'\in N_G(M')\cap X$ and $L'=L(x')$ (a \emph{semistandard} Levi subgroup of $G$).
By Corollary \ref{cor: xadm} there exists $n\in N_G(T)$ such that $M=nM'n^{-1}$ is standard, $x=nx'n^{-1}$ is $M$-standard \rlvt\ and $L:=L(x)=nL' n^{-1}$.
Recall further that $(\aaa_M)_x^{+}=n((\aaa_{M'})_{x'}^{+})$ and
$(\aaa_M)_x^{-}=n((\aaa_{M'})_{x'}^{-})$.
We keep using the same notation as in \S\ref{subsec: orbits}. In particular $M=M_{(r_1,r_1,\dots,r_k,r_k,s_1,\dots,s_l,t_1,\dots,t_m;u)}$ and
$\p=(p_1,q_1,\dots,p_l,q_l;p,q)$ is the data associated to $x$ by \eqref{eq: pq data}.
Under the natural identification $\aaa_M \simeq \R^{2k+l+m}$ we have
\begin{equation}\label{eq: plus sp}
(\aaa_M)_x^{+}=\aaa_L=\{(\lambda_1,\lambda_1,\dots,\lambda_k,\lambda_k,\mu_1,\dots,\mu_l,\overbrace{0.\dots,0}^m):\lambda_1,\dots,\lambda_k,\,\mu_1,\dots,\mu_l\in\R\}
\end{equation}
and
\[
(\aaa_M)_x^{-}=\aaa_M^L=\{(\lambda_1,-\lambda_1,\dots,\lambda_k,-\lambda_k,\overbrace{0.\dots,0}^l,\mu_1,\dots,\mu_m):
\lambda_1,\dots,\lambda_k,\,\mu_1,\dots,\mu_m\in\R\}.
\]

Let $\grp P=\grp M\ltimes \grp U$ be a parabolic subgroup of $\grp G$ and let $\alpha\in \srts_P$.
Denote by $s_\alpha\in W(M)$ the elementary symmetry associated to $\alpha$ as in \cite[\S I.1.7]{MR1361168}.

We define a directed edge-labeled graph $\grph$ in the spirit of \cite[\S3.3]{MR2010737} as follows.
The vertices of $\grph$ are pairs $(M,x)$ where $M$ is a Levi subgroup of $G$ and $x\in N_G(M)\cap X$.
The (labeled) edges of $\grph$ are given by $(M,x)\xrightarrow{n_\alpha}(M',x')$ provided that:
\begin{enumerate}
\item $\alpha\in\srts_P$,
\item $n_\alpha\in s_\alpha M$,
\item $x\alpha\ne\pm\alpha$,
\item $M'=s_\alpha M s_\alpha^{-1}=n_\alpha M n_\alpha^{-1}$,
\item $x'=n_\alpha xn_\alpha^{-1}$.
\end{enumerate}
We will write $(M,x)\overset{n_\alpha}\searrow(M',x')$ if $(M,x)\xrightarrow{n_\alpha}(M',x')$ and $x\alpha<0$
(but $x\alpha\ne-\alpha$).
Note that if $(M,x)\xrightarrow{n_\alpha}(M',x')$ then also $(M',x')\xrightarrow{n_\alpha^{-1}}(M,x)$.
Moreover, either $(M,x)\overset{n_\alpha}\searrow(M',x')$ or $(M',x')\overset{n_\alpha^{-1}}\searrow(M,x)$ but not both.
For a finite sequence of edges
\[
(M,x)=(M_1,x_1)\xrightarrow{n_{\alpha_1}}(M_2,x_2)\xrightarrow{n_{\alpha_2}}\cdots\xrightarrow{n_{\alpha_k}}(M_{k+1},x_{k+1})=(M',x')
\]
in $\grph$ we will write $(M,x)\overset{n}\curvearrowright(M',x')$ where $n=n_{\alpha_k}\dots n_{\alpha_1}\in G$.
Note that $n$ conjugates $(M,x)$ to $(M',x')$.
Similarly, we write $(M,x)\overset{n}\downarrow(M',x')$ if there exists a finite sequence
\[
(M,x)=(M_1,x_1)\overset{n_{\alpha_1}}\searrow(M_2,x_2)\overset{n_{\alpha_2}}\searrow\cdots\overset{n_{\alpha_k}}\searrow (M_{k+1},x_{k+1})=(M',x').
\]

\begin{lemma}\label{lem: alpha}
Suppose that $(M,x)$ and $(M',x')$ are vertices in $\grph$ and $(M,x)\overset{n_\alpha}\searrow(M',x')$ for some $\alpha\in\srts_P$.
Let $\grp Q=\grp L\ltimes \grp V$ be the parabolic subgroup of $\grp G$ containing $\grp P$ such that $\srts_P^Q=\{\alpha\}$
and let $\grp{P'}=\grp{M'}\ltimes\grp{U'}$ be the parabolic subgroup of $\grp Q$ such that $\srts_{P'}^Q=\{-s_\alpha \alpha\}$.
Then
\begin{enumerate}
\item $\grp V_{x'}=n_\alpha \grp U_x n_\alpha^{-1}$ and in particular $n_\alpha \grp U_x n_\alpha^{-1}\subseteq \grp U'_{x'}$.
\item \label{part: exact seq} We have the following short exact sequence of subgroups normalized by $\grp M'_{x'}$:
\[
1\longrightarrow n_\alpha \grp U_x n_\alpha^{-1} \longrightarrow \grp U'_{x'} \mathop{\longrightarrow}\limits^{\pr_L} \grp L\cap \grp{U'} \longrightarrow 1.
\]
\item \label{part: int unip} For any function $f$ on $\grp V(\A)\bs \grp{U'}(\A)$ we have
\[
\int_{n_\alpha \grp U_x(\A) n_\alpha^{-1}\bs \grp U'_{x'}(\A)} f(u)\ du=\int_{\grp V(\A)\bs \grp{U'}(\A)}f(u)\ du=\int_{(\grp L\cap \grp{U'})(\A)} f(u)\ du
\]
(whenever the integral is defined).
\item \label{part: int f} $n_\alpha \grp P_x n_\alpha^{-1} \subseteq \grp P'_{x'}$ and a semi-invariant measure on
$n_\alpha \grp P_x(\A) n_\alpha^{-1}\bs \grp P'_{x'}(\A)$ is given by integration over $n_\alpha \grp U_x(\A) n_\alpha^{-1}\bs \grp U'_{x'}(\A)$.
\item We have
\begin{equation} \label{eq: pxp'x'}
 \modulus_{P_x}(m)=(\modulus_{P'_{x'}}\modulus_{P'\cap L}^{-1})(n_\alpha m  n_\alpha^{-1}),\ \ \ m\in \grp M_x(\A)
\end{equation}
and
\begin{equation} \label{eq: modulchang}
(\modulus_P^{-\frac12} \modulus_{P_x})(m)=(\modulus_{P'}^{-\frac12} \modulus_{P'_{x'}})(n_\alpha m  n_\alpha^{-1}),\ \ \ m\in \grp M_x(\A).
\end{equation}
In particular,
\begin{equation}\label{eq: exp eqv simp}
n_\alpha\rho_x=\rho_{x'}.
\end{equation}
\end{enumerate}
\end{lemma}

\begin{proof}
The first four parts are proved exactly as \cite[Lemma 4.3.1]{MR2010737}. We omit the details.
Moreover, as in the proof of \cite[Proposition 4.3.2]{MR2010737} the relation \eqref{eq: pxp'x'} follows from part \eqref{part: exact seq}.
It is also observed in the proof of [ibid.] that $s_\alpha \rho_P+2\rho_{P'}^Q=\rho_{P'}$ and therefore
$\modulus_P^{-\frac12}(m)=(\modulus_{P'}^{-\frac12}\modulus_{P'\cap L})(n_\alpha m  n_\alpha^{-1})$.
The identity \eqref{eq: modulchang} follows.
Finally, the identity \eqref{eq: exp eqv simp} follows by restricting \eqref{eq: modulchang} to $A_M^{M_x}$.
\end{proof}

A straightforward consequence of the lemma is
\begin{corollary}\label{cor: graph}
Suppose that $(M,x)\overset{n}\curvearrowright(M',x')$ in $\grph$. Then
\[
(\modulus_P^{-\frac12} \modulus_{P_x})(m)=(\modulus_{P'}^{-\frac12} \modulus_{P'_{x'}})(n m  n^{-1}),\ \ \ m\in \grp M_x(\A).
\]
In particular, $n\rho_x=\rho_{x'}$.
\end{corollary}

Using \cite[Lemma 3.2.1 and Proposition 3.3.1]{MR2010737} we get:

\begin{corollary}\label{cor: exp w}
Let $M$ be a Levi subgroup of $G$ and $x\in N_G(M)\cap X$. Then there exists $n\in G$ such that $M'=nMn^{-1}$ is standard,
$x'=nxn^{-1}$ is $M'$-minimal and $(M,x)\overset{n}\downarrow(M',x')$.
Therefore,
\[
(\modulus_P^{-\frac12} \modulus_{P_x})(m)=(\modulus_{P'}^{-\frac12} \modulus_{P'_{x'}})(n m  n^{-1}),\ \ \ m\in \grp M_x(\A).
\]
In particular, $n\rho_x=\rho_{x'}$.
\end{corollary}

\subsection{Cuspidal orbits}
The following definition will be central for the analysis of periods of (pseudo) Eisenstein series.
\begin{definition}
Let $M$ be a Levi subgroup of $G$, $x\in N_G(M)\cap X$ and $L=L(x)$ (a \emph{semistandard} Levi subgroup of $G$ containing $M$). We say that $x$ is $M$-standard cuspidal if
$(M,L)$ is a standard \rlvt\ pair such that in the notation of Definition \ref{def: std basic}
$v=0$ (i.e., $L\subseteq M_{(2n;0)}$) and there exists $0\le l_1\le l$ such that $s_1,\dots,s_{l_1}$ are even and $s_i=1$, $l_1+1\le i\le l$
and moreover the data $\p$ associated to $x$ by \eqref{eq: pq data} satisfies $p_i=q_i$, $i=1,\dots,l_1$.

More generally, we say that $x$ is $M$-cuspidal if there exists $n\in N_G(T)$ such that $nMn^{-1}$ is a standard Levi subgroup of $G$ and
$nxn^{-1}$ is $nMn^{-1}$-standard cuspidal.
\end{definition}

\begin{remark}\label{rmk: cusp std rlvt}
Suppose that $x$ is $M$-cuspidal and $M$-standard \rlvt\ (see Definition \ref{def: stdx}).
Then $M=M_{(r_1,r_1,\dots,r_k,r_k,s_1,\dots,s_l;0)}$,
$s_j$ is either even or $1$ for every $j=1,\dots,l$ and $x\in N_G(M)\cap X$ is $M$-conjugate to
$\inj(\sm{0}{I_{r_1}}{I_{r_1}}{0},\dots,\sm{0}{I_{r_k}}{I_{r_k}}{0},h_1,\dots,h_l)$
where $h_j=\sm{0}{I_{s_j/2}}{I_{s_j/2}}{0}$ if $s_j$ is even and $h_j=\pm 1$ if $s_j=1$. (Clearly, $\sm{0}{I_{s_j/2}}{I_{s_j/2}}{0}$ is
$\GL_{s_j}$-conjugate to $\diag(I_{s_j/2},-I_{s_j/2})$.)
Moreover,
\[
\#\{j:h_j=1\}=\#\{j:h_j=-1\}.
\]
\end{remark}

Denote by $\grph_{\cusp}$ the full subgraph of $\grph$ whose vertices are $(M,x)$ where $x$ is $M$-cuspidal.
Note that $\grph_{\cusp}$ is a union of connected components of $\grph$.
We will explicate the graph $\grph_{\cusp}$ and the elements $\rho_x$ where $(M,x)$ is a vertex in $\grph_{\cusp}$.

\begin{lemma}\label{lem: s,r,r}
Any connected component of $\grph_{\cusp}$ contains a vertex $(M,x)$ such that $x$ is $M$-standard \rlvt.
\end{lemma}

\begin{proof}
Let $(M,x)\in \grph_{\cusp}$. By Corollary \ref{cor: exp w} we may assume that $x$ is $M$-minimal.
Let $\gamma=(n_1,\dots,n_k;0)$ be such that $M=M_\gamma$ and assume that $(n_{i-1},n_i,n_{i+1})=(s,r,r)$ for some $1<i<k$ and that $x$ has the form
$x=\inj(x_1,x_2,x_3)$ where $x_1\in \GL_{n_1+\cdots+n_{i-2}}$, $x_3\in \GL_{n_{i+2}+\cdots+n_k}$ and
\[
x_2=\left(\begin{smallmatrix} h & 0 & 0 \\ 0 & 0 & y \\ 0 & y^{-1} & 0 \end{smallmatrix}\right),
\ \ \ \text{with}\ \ \ y\in \GL_r,\  h\in \GL_s\text{ and } h\in [\sm{0}{I_{s/2}}{I_{s/2}}{0}]_{\GL_s} \text{ if }s \text{ is even}.
\]
Let
\[
\gamma'=(n_1,\dots,n_{i-2},r,s,r,n_{i+2},\dots,n_k;0) \ \ \ \text{and}\ \ \ \gamma''=(n_1,\dots,n_{i-2},r,r,s,n_{i+2},\dots,n_k;0).
\]
Let $\alpha\in \srts_{P_\gamma}^G$ be such that $n_\alpha=\inj(I_{n_1+\cdots+n_{i-2}},\sm{0}{I_r}{I_s}{0},I_{n_{i+1}+\cdots+n_k})\in s_\alpha M_\gamma$
and similarly, let $\beta\in \srts_{P_{\gamma'}}^G$ be such that
$n_\beta=\inj(I_{n_1+\cdots+n_{i-2}+r},\sm{0}{I_r}{I_s}{0},I_{n_{i+2}+\cdots+n_k})\in s_\beta M_{\gamma'}$. Set $x'=n_\alpha x n_\alpha^{-1}$ and
$x''=n_\beta x' n_{\beta}^{-1}$.
It is easy to see that
$(M,x)\xrightarrow{n_\alpha}(M_{\gamma'},x')\xrightarrow{n_\beta}(M_{\gamma''},x'')$ and
$x''=\inj(x_1,y_2,x_3)$ where
\[
y_2=\left(\begin{smallmatrix} 0 & y & 0 \\ y^{-1} & 0 & 0 \\ 0 & 0 & h \end{smallmatrix}\right)
\]
is $M_{\gamma''}$-minimal.
The Lemma therefore follows from the analysis of \S\ref{subsec: orbits}, Remark \ref{rmk: Mmin} and Lemma \ref{lem: min is basic} \eqref{part: min basic}.
\end{proof}

We will reduce the computation of $\rho_x$ to the case where $x$ is $M$-standard cuspidal.
Indeed,

\begin{lemma}
Let $(M,x)\in \grph_{\cusp}$ be such that $x$ is $M$-standard \rlvt.
In the notation of Remark \ref{rmk: cusp std rlvt} assume that $j\in\{1,\dots,l-1\}$ is such that exactly one of $s_j$ and $s_{j+1}$ is even (and the other equals $1$).
Let $\alpha\in\srts_P$ be the projection to $\aaa_M^*$ of the simple root $\alpha_t$ where $t=2(r_1+\cdots +r_k)+s_1+\cdots+s_j$.
Let $n_\alpha\in s_\alpha M$ and $x'=n_\alpha xn_\alpha^{-1}$.
Then $s_\alpha\rho_x=\rho_{x'}$.
\end{lemma}

\begin{proof}
By symmetry, without loss of generality,  we may assume that $s_j$ is even and $s_{j+1}=1$.
Since the result depends only on the $M$-orbit of $x$ we may further assume without loss of generality that
$x=\inj(\sm{0}{I_{r_1}}{I_{r_1}}{0},\dots,\sm{0}{I_{r_k}}{I_{r_k}}{0},h_1,\dots,h_l)$ in the notation of Remark \ref{rmk: cusp std rlvt}.

Let $\grp{P'}=\grp{M'}\ltimes\grp{U'}$ and $\grp Q=\grp L\ltimes\grp V$ be the parabolic subgroups of $\grp G$ such that
$\grp{M'}=n_\alpha \grp Mn_\alpha^{-1}$, $\grp P\subseteq \grp Q$ and $\srts_P^Q=\{\alpha\}$.
Since $\grp Q_x=\grp L_x\ltimes \grp V_x$ and similarly for $x'$ (Lemma \ref{lem: factor adm}) we have $\grp U_x=\grp V_x\rtimes (\grp U_x\cap \grp L)$
and $\grp U'_{x'}=\grp V_{x'}\rtimes(\grp U'_{x'}\cap \grp L)$.
Thus, $\modulus_{P_x}|_{\grp P_x(\A)\cap\grp L(\A)}=\modulus_{P_x\cap L}\modulus_{Q_x}|_{\grp P_x(\A)\cap\grp L(\A)}$ and
$\modulus_{P'_{x'}}|_{\grp P_{x'}(\A)\cap\grp L(\A)}=\modulus_{P'_{x'}\cap L}\modulus_{Q_{x'}}|_{\grp P'_{x'}(\A)\cap\grp L(\A)}$.
Also $\modulus_P|_{\grp P(\A)\cap \grp L(\A)}=\modulus_{P\cap L}\modulus_Q|_{\grp P(\A)\cap\grp L(\A)}$.
Note that $n_\alpha\grp Vn_\alpha^{-1}=\grp V$ and $\grp V_{x'}=n_\alpha \grp V_xn_\alpha^{-1}$.
Thus,
\[
\modulus_Q(n_\alpha ln_\alpha^{-1})=\modulus_Q(l)\ \ \ \text{and}\ \ \
\modulus_{Q_{x'}}(n_\alpha ln_\alpha^{-1})=\modulus_{Q_x}(l), \ \ \ l\in \grp P_x(\A)\cap \grp L(\A).
\]
Thus, it suffices to show that
\begin{equation}\label{eq: mod in L}
(\modulus_{P\cap L}^{-\frac12} \modulus_{P_x\cap L})(a)=
(\modulus_{P'\cap L}^{-\frac12} \modulus_{P'_{x'}\cap L})(n_\alpha a  n_\alpha^{-1})
\end{equation}
for $a\in A_M^{M_x}$.
Recall that
\[
P\cap L=M \ltimes (U \cap L),\ \ \ P'\cap L=M' \ltimes (U' \cap L),
\]
\[
P_x\cap L=M_x \ltimes (U_x \cap L) \ \ \text{ and }\ \ P'_{x'}\cap L=M'_{x'} \ltimes (U'_{x'} \cap L).
\]

Let $L=\inj(L_1,L_2,L_3)$ where $L_2=\GL_{s_j+1}$ and $L_1$ (reps.~$L_3$) is the corresponding Levi subgroup of $\GL_e$ (reps.~$\GL_f$)
where we set $e=2(r_1+\cdots+r_k)+s_1+\cdots+s_{j-1}$ (resp.~$f=s_{j+2}+\cdots+s_l$).
Accordingly, $x=\inj(x_1,x_2,x_3)$ with $x_i\in L_i$ and $P\cap L=\inj(P_1,P_2,P_3)$ where $P_i=M_i \ltimes U_i$ is the corresponding
parabolic subgroup of $L_i$, $i=1,2,3$.
Thus, $M=\inj(M_1,M_2,M_3)$ and $U\cap L=\inj(U_1,U_2,U_3)$. In particular, $M_2=\GL_{s_j}\times\GL_1$.
Note that $P'\cap L=\inj(P_1,P_2',P_3)$ where $P_2'=M_2'\ltimes U_2'$ is the parabolic subgroup of $\GL_{s_j+1}$ of type $(1,s_j)$.

We have $M_x=\inj(D_1,D_2,D_3)$ where $D_i=(M_i)_{x_i}$ and in particular,
$D_2=\diag(C_{s_j},\GL_1)$ (see \S\ref{subsec: orbits} for notation).
The result is independent of the choice of $n_\alpha\in s_\alpha M$ and without loss of generality we may choose $n_\alpha=\inj(I_e,w,I_f)$ where
$w=\sm{0}{1}{I_{s_j}}{0}$. Then, $x'=\inj(x_1,x_2',x_3)$ with $x_2'=wx_2 w^{-1}$ and $M'_{x'}=\inj((D_1,D_2',D_3)$ where $D_2'=\diag(\GL_1,C_{s_j})$.

Furthermore, $U_x\cap L=\inj((U_1)_{x_1},(U_2)_{x_2},(U_3)_{x_3})$ and $U_{x'}\cap L=\inj((U_1)_{x_1},(U_2')_{x_2'},(U_3)_{x_3})$.
Finally, for $a=\inj(a_1,a_2,a_3)\in A_M^{M_x}$ with $a_i$ in (the center of) $M_i$ we have $n_\alpha an_\alpha^{-1}=\inj(a_1,a_2',a_3)$ where
$a_2'=wa_2w^{-1}$.
Note that decomposing $a=a'a''$ with $a'=\inj(I_e,a_2,I_f)$ and $a''=\inj(a_1,I_{s_j+1},a_3)$ we have $a',\,a''\in A_M^{M_x}$ and the
identity \eqref{eq: mod in L} clearly holds for $a''$. It is left to show that it also holds for $a'$.

We have
\[
(U_2)_{x_2}=\{\left(\begin{matrix}I_{s_j}&\begin{smallmatrix}z\\h_{j+1}z\end{smallmatrix}\\&1\end{matrix}\right):z\in F^{s_j/2}\}.
\]
Therefore, for $g=\left(\begin{smallmatrix}a&b&\\b&a&\\&&t\end{smallmatrix}\right)\in D_2$ with
$\left(\begin{smallmatrix}a&b\\b&a\end{smallmatrix}\right)\in C_{s_j}$ and $t\in F^*$ we have
\begin{gather*}
\modulus_{P_x\cap L}(\inj(I,g,I))=\abs{\det(a+bh_{j+1})t^{-n}},\\
\modulus_{P_x\cap L}\modulus_{P\cap L}^{-\frac12}(\inj(I,g,I))=
\abs{\det(a+bh_{j+1})}^{\frac12}\abs{\det(a-h_{j+1}b)}^{-\frac12}.
\end{gather*}
(Indeed, $\left(\begin{smallmatrix}a&b\\b&a\end{smallmatrix}\right)$ is conjugate to $\diag(a+bh_{j+1},a-bh_{j+1})$.) Similarly, for $g'=wgw^{-1}=\left(\begin{smallmatrix}t&&\\&a&b\\&b&a\end{smallmatrix}\right)$ we have $\inj(I,g',I)=n_\alpha \inj(I,g,I)n_\alpha^{-1}\in M'_{x'}$ and
\begin{gather*}
\modulus_{P'_{x'}\cap L}(\inj(I,g',I))=\abs{t^n\det(a+bh_{j+1})^{-1}},\\
\modulus_{P'_{x'}\cap L}\modulus_{P'\cap L}^{-\frac12}(\inj(I,g',I))=
\abs{\det(a-bh_{j+1})}^{\frac12}\abs{\det(a+h_{j+1}b)}^{-\frac12}.
\end{gather*}
In particular, if $b=0$ we get
\[
\modulus_{P_x\cap L}\modulus_{P\cap L}^{-\frac12}(\inj(I,g,I))=
\modulus_{P'_{x'}\cap L}\modulus_{P'\cap L}^{-\frac12}(\inj(I,g',I))=1.
\]
The Lemma follows.
\end{proof}

It follows from the lemma above that in the computation of $\rho_x$ we can assume in addition that $x$ is $M$-standard cuspidal, i.e.,
that there exists $l_1\le l$ such that $s_i$ is even for $i=1,\dots,l_1$ and $s_i=1$ for all $i>l_1$. Set $l_2=l-l_1$ and note that $l_2$ is even.

\begin{lemma} \label{lem: rhox description}
For $x\in N_G(M)\cap X$ that is $M$-standard cuspidal, with the above notation we have $\rho_x=(\overbrace{\frac12,\dots,\frac12}^{2k+l_1},
\lambda_1,\dots,\lambda_{l_2})$ where $\lambda_i=2\#\{j\ge i:h_{l_1+i}=h_{l_1+j}\}-(l_2+1-i)$, $i=1,\dots,l_2$.
\end{lemma}

Note that $(\lambda_1,\dots,\lambda_{l_2})$ is an intersection of singular hyperplanes. More precisely,
$\lambda_{l_2}=1$ and for $i=1,\dots,l_2-1$ we have $\lambda_i-\eps_i\lambda_{i+1}=1$ where $\eps_i=1$ if $h_{l_1+i}=h_{l_1+i+1}$ and $\eps_i=-1$
otherwise.

\begin{proof}
Let $L=M_{(2r_1,\dots,2r_k,s_1,\dots,s_l;0)}$ so that $x\in wM$ where $w=w_M^L$.
Since $\rho_x$ depends only on $[x]_M$ we may assume without loss of generality that
\[
x=\inj(\sm{0}{I_{r_1}}{I_{r_1}}{0},\dots,\sm{0}{I_{r_k}}{I_{r_k}}{0},\sm{0}{I_{s_1/2}}{I_{s_1/2}}{0},\dots,\sm{0}{I_{s_{l_1}/2}}{I_{s_{l_1}/2}}{0},h_{l_1+1},\dots,h_l)
\]
where $h_{l_1+i}=\pm 1$, $i=1,\dots,l_2$.
By Lemma \ref{lem: stab vx} we have $\modulus_{P_x}=\modulus_{Q_x}\rest_{P_x}$.
Since $A_M^{M_x}=A_L$, in order to compute $\rho_x$ we need to compute $\modulus_{Q_x}\modulus_P^{-\frac12}\rest_{A_L}=\modulus_{Q_x}\modulus_Q^{-\frac12}\rest_{A_L}$.

Let $R(T,V)$ be the set of roots of $T$ on $\operatorname{Lie} V$.
For any $\beta\in R(T,V)$ denote by $U^\beta$ the corresponding one-parameter root subgroup.
Note that $x$ normalises both $T$ and $V$ and acts as an involution on $R(T,V)$. We can decompose $V$ according to the orbits of $x$
\[
V=\prod_{\orb} V^\orb
\]
(the product taken in any order and the multiplication map defines an isomorphism of affine algebraic varieties)
where $V^\orb=\prod_{\beta\in\orb}U^\beta$.
(Note that $U^\beta$ commutes with $U^{x\beta}$ so that $V^\orb$ is a group.)
Thus,
\[
V_x=\prod_{\orb}V^\orb_x.
\]
If $\abs{\orb}=2$ then $\dim V^\orb_x=1$, while if $\orb=\{\beta\}$ then $V^\orb_x=U^\beta_x$
is either $1$ or equal to $U^\beta$. Altogether,
\[
\modulus_{Q_x}\modulus_Q^{-\frac12}(a)=\prod_{\beta\in R(T,V):x\beta=\beta}\abs{\beta(a)}^{\dim U^\beta_x-\frac12},\ \ \ a\in A_L.
\]
The roots $\beta$ such that $x\beta=\beta$ can be enumerated as follows. Let $R_t=2\sum_{i=1}^{t-1} r_i$, $t=1,\dots,k+1$, $R=R_{k+1}$,
$S_t=R+\sum_{i=1}^{t-1} s_i$, $t=1,\dots,l_1+1$ and $S=S_{l_1+1}$. The roots are
\begin{enumerate}
\item $e_i+e_j$ where either $i=R_t+a$ for some $t=1,\dots,k$, $1\le a\le r_t$ and $j=i+r_t$ or $i=S_t+a$ for some $t=1,\dots,l_1$,
$1\le a\le \frac{s_t}2$, $j=i+\frac{s_t}2$.
\item $e_i \pm e_j$ where $S<i\le j \le 2n$ ($e_i-e_j$ is only a root if $i<j$).
\end{enumerate}
For $\beta$ of the first type $U_x^\beta=U^\beta$ and this explains the first $2k+l_1$ coordinates of $\rho_x$.
For $\beta$ of the second type $U^\beta_x=U^\beta$ if and only if $h_{l_1+i-S}=h_{l_1+j-S}$.

Fix $1\le i\le l_2$ and set $a=\inj(I_{S+i-1},y,I_{2n-(S+i)})$, $y>0$. By definition,
\[
\prod_{\beta\in R(T,V):x\beta=\beta}\abs{\beta(a)}^{\dim U^\beta_x-\frac12}=\abs{y}^{\lambda_i}.
\]
Note that the contributions of $e_{S+j}+e_{S+i}$ and of $e_{S+j}-e_{S+i}$ to the left hand side cancel each other out for all $1\le j<i$.
The contribution of $e_{S+i}\pm e_{S+j}$ for $i<j\le l_2$ equals $\abs{y}^{\frac\epsilon 2}$ where
$\epsilon=1$ if $h_{l_1+i}=h_{l_1+j}$ and $\epsilon=-1$ otherwise.
Combining the contribution $\abs{y}$ from the case $j=i$ this yields
\[
\lambda_i=1+\#\{i<j: h_{l_1+i}=h_{l_1+j}\}-\#\{i<j:h_{l_1+i}\ne h_{l_1+j}\}.
\]
Since $\#\{i<j:h_{l_1+i}\ne h_{l_1+j}\}=l_2-i-\#\{i<j: h_{l_1+i}=h_{l_1+j}\}$ the lemma follows.
\end{proof}

The affine subspace $\rho_x+(\aaa_M^*)_x^-$ of $\R^{2k+l}$ can be described as follows.
Let
\begin{gather*}
\h_{i}=\{(\mu_1,\dots,\mu_{2k+l}): \mu_i=1/2\},\ 1\le i\le 2k+l,\\
\h_i^{\pm}=\{(\mu_1,\dots,\mu_{2k+l}): \mu_i\mp\mu_{i+1}=1\},\ 1\le i<2k+l,\\
\h_{2k+l}^+=\{(\mu_1,\dots,\mu_{2k+l}): \mu_{2k+l}=1\}.
\end{gather*}
Then we have
\[
\rho_x+(\aaa_M^*)_x^-=\left\{\cap_{i=1}^k \h_{2i-1}^-\right\}\cap \left\{\cap_{i=2k+1}^{2k+l_1} \h_i \right\}\cap
\left\{\cap_{i=2k+l_1+1}^{2k+l-1} \h_i^{(h_ih_{i+1})}\right\}\cap \h_{2k+l}^+.
\]

\section{Intertwining Periods} \label{sec: intertwining periods}
In this section we define the intertwining periods for the pair $(G,H)$.
These are certain $\grp H(\A)$-invariant linear forms defined on induced representations of $\grp G(\A)$.
They were introduced and studied in the Galois case in \cite{MR1625060} and \cite{MR2010737}.
Our treatment follows the same line but on a technical level we use a slightly different argument
for the convergence.

The intertwining periods are built from inner period integrals.
We will first study the latter.

\subsection{Vanishing pairs}
For this subsection let $\grp{G}$ be a reductive group and $\grp{H}$ a reductive subgroup both defined over $F$.
Recall the notation \eqref{def: RAS}.
By \cite[Proposition 1]{MR1233493} for every cusp form $\phi$ on $\AS{G}$ we have
\begin{equation} \label{eq: convphi}
\int_{\RAS HG}\abs{\phi(h)}\ dh<\infty.
\end{equation}
Note that $\int_{\RAS HG}\phi(h)\ dh=\int_{H\bs(\grp{H}(\A)\cap\grp{G}(\A)^1)} \phi(h)\ dh$.
Let $\grp{H^{\der}}$ be the derived group of $\grp{H}$. Then $H\grp{H^{\der}}(\A)$ is co-compact in $\grp{H(\A)}^1$.
Applying \eqref{eq: convphi} to $\grp{H^{\der}}$ we conclude that
\[
\int_{H\bs\grp{H}(\A)^1}\abs{\phi(h)}\ dh<\infty.
\]

\begin{definition}(cf.~\cite{MR1233493})
We say that $(G,H)$ is a \emph{vanishing pair} if
\[
\int_{H\bs\grp{H}(\A)^1} \phi(h)\ dh=0
\]
for every smooth cuspidal function of uniform moderate growth (hence rapidly decreasing) $\phi$ on $\AS{G}$.
\end{definition}

\begin{remark}
In \cite{MR1233493} the (a priori weaker) condition
\[
\int_{\RAS HG} \phi(h)\ dh=0
\]
is used. However, for our purposes the definition above is more convenient.
\end{remark}

We recall the following results which are special cases of results of Ash-Ginzburg-Rallis and Jacquet-Rallis.
\begin{theorem} \label{thm: AGR JR}
The following are vanishing pairs:
\begin{itemize}
\item $(\Sp_{n+m},\Sp_n\times\Sp_m)$ for all $m,\,n$ (\cite{MR1233493});
\item $(\GL_{n+m},\GL_n\times\GL_m)$ (or even $(\GL_{n+m},\SL_n \times\SL_m)$ for all $m\ne n$ (\cite{MR1233493});
\item $(\GL_{2n},\Sp_n)$ (\cite{MR1142486})\footnote{Strictly speaking, Jacquet--Rallis prove vanishing of periods only for cuspidal automorphic forms.
However, the general case can easily be deduced from it. At any rate, the technique of \cite{MR1233493} applies equally well to this case.}.
\end{itemize}
\end{theorem}

\begin{remark}
In fact, it was proved in \cite{MR1241129} that $(\GL_n,\SL_m)$ is a vanishing pair if $m>n/2$.
However, we will not use this fact.
\end{remark}

From Theorem \ref{thm: AGR JR}, Corollary \ref{cor: xadm} and the description of the stabilizers in \S\ref{subsec: orbits} we infer:
\begin{corollary} \label{cor: xvanish}
Let $(G,H)=(\Sp_{2n},\Sp_n\times\Sp_n)$ and let $M$ be a Levi subgroup of $G$.
Then for any $x\in N_G(M)\cap X$, $(M,M_x)$ is a vanishing pair unless $x$ is $M$-cuspidal.
\end{corollary}


In the case where $(G,H)$ is not a vanishing pair we say that a cuspidal automorphic representation $\pi$ of $\grp G(\A)$
whose central character is trivial on $A_G$ is $H$-distinguished if there is $\varphi$ in the space of $\pi$ such that
\[
\int_{\RAS HG}\varphi(h)\ dh\ne0.
\]

\subsection{} \label{sec: GL2nGLnGLn}
For this subsection let $\grp{G}=\grp{GL}_{2n}$,
$\grp{P}=\grp{M}\ltimes \grp{U}$ the standard maximal parabolic subgroup of $\grp{G}$ with Levi subgroup
$\grp{M}=\{\sm{g_1}{}{}{g_2}:g_1,\,g_2\in\grp{GL}_n\}$ and unipotent radical $\grp{U}$
and $K$ the standard maximal compact subgroup of $\grp{G}(\A)$.
Recall that $H_{M}:\grp{M}(\A)\to \aaa_{M}\simeq \R^2$ is extended to $\grp{G}(\A)=\grp{P}(\A)K$ via the Iwasawa decomposition.

\begin{lemma} \label{lem: bndinner}
For any $\lambda\in(\aaa_M^G)^*$ there exists $N$ such that
\[
\int_{\RAS MG}e^{\sprod{\lambda}{H_{P}(m)}}\abs{\phi(m)}\ dm\ll_\lambda\sup_{g\in\siegel_{G}^1}\abs{\phi(g)}\norm{g}^N
\]
for any continuous function $\phi$ on $\AS{G}$.
\end{lemma}

\begin{proof}
It follows from \cite[Proposition 6]{MR1142486} and its proof that for any $N'>0$ there exists $N$ such that
\[
\sup_{m\in\grp{M}(\A)}\abs{\phi(m)}\modulus_{P}(m)^{N'}\ll_{N'}\sup_{g\in\siegel_{G}^1}\abs{\phi(g)}\norm{g}^N.
\]
Applying this also to $wmw^{-1}$ and the translate of $\phi$ by $w=\sm{}{I_n}{I_n}{}$ we get that
\begin{equation}\label{eq: rel ineq}
\sup_{m\in \grp{M}(\A)}\abs{\phi(m)}\max(\modulus_{P}(m),\modulus_{P}(m)^{-1})^{N'}\ll_{N'}\sup_{g\in\siegel_{G}^1}\abs{\phi(g)}\norm{g}^N.
\end{equation}
Clearly, there exists $N_\lambda$ such that
\[
e^{\sprod{\lambda}{H_{P}(m)}}\le \max(\modulus_{P}(m),\modulus_{P}(m)^{-1})^{N_\lambda},\ \ \ m\in \grp{M}(\A).
\]
Applying the inequality \eqref{eq: rel ineq} with $N'=N_\lambda+N''$ it remains to note that
\[
\int_{\RAS MG}\max(\modulus_{P}(m),\modulus_{P}(m)^{-1})^{-N''}\ dm
\]
converges for $N''\gg1$.
\end{proof}

We also require the convergence of an auxiliary integral associated with the pair $(\grp{G},\grp{H})$
where $\grp{H}$ is the centralizer of $w_M^G=\sm{}{I_n}{I_n}{}$.
Let $\grp{M_H}=\grp{M}\cap \grp{H}$ and define
\[
\norm{h}_{M_H\bs H}=\inf_{m\in \grp{M_H}(\A)}\norm{mh}, \ \ \ h\in \grp{H}(\A).
\]
Note that $H$ consists of the matrices in $G$ of the form $\sm abba$ and $M_H=\GL_n^\triangle=\{\diag(g,g):g\in\GL_n\}$.

\begin{lemma}\label{lem: aux bd}
For any $N>0$ there exists $t_0$ such that the integral
\[
\int_{\grp{M_H}(\A)\bs\grp{H}(\A)} e^{\sprod{(t,-t)}{H_M(h)}}\norm{h}_{M_H\bs H}^N\ dh
\]
converges uniformly for $t$ in any compact subset of $(t_0,\infty)$.
\end{lemma}

\begin{proof}
Let $\eta_n=\sm{I_n}{I_n}{I_n}{-I_n}$. Then $\eta_n^{-1}H\eta_n=M$ and $\eta_n$ centralizes $M_H$.
Note that for $g\in \grp{GL}_n(\A)$ we have
\[
\eta_n\sm{I_n}{}{}{g} \eta_n^{-1}=\frac12\sm{I_n+g}{I_n-g}{I_n-g}{I_n+g}
\]
from which it follows that $\norm{\eta_n\sm{I_n}{}{}{g} \eta_n^{-1}}_{M_H\bs H} \ll \|g\|$.

Applying the change of variable $h\mapsto \eta_n h\eta_n^{-1}$ and \eqref{eq: exp bd} we reduce to the convergence of the integral
\[
\int_{\grp{GL}_n(\A)} e^{\sprod{(t,-t)}{H_M(\eta_n\sm{I_n}{}{}{g})}}\norm{g}^N\ dg.
\]
Observe that if $g\in K\diag(t_1,\dots,t_n)K$ then
\[
\abs{\det \eta_n}e^{\sprod{(-1,1)}{H_M(\eta_n\sm{I_n}{}{}{g})}}=\prod_{i=1}^n\max(\abs{t_i},\abs{t_i}^{-1})\ge\max_i(\abs{t_i},\abs{t_i}^{-1})=\norm{g}.
\]
Thus, the Lemma follows from the convergence of
\[
\int_{\grp{GL}_n(\A)}\norm{g}^{-t}\ dg
\]
for $t\gg1$ which is a standard fact (it follows e.g. from \cite[Proposition 7]{MR1142486}).
\end{proof}

\subsection{Definition of the intertwining period}
We go back to the setup of \S\ref{sec: setup}.
Let $\grp P=\grp M\ltimes \grp U$ be a parabolic subgroup of $\grp G$ and let $x\in N_G(M)\cap X$.
For $\varphi\in\AF_P(G)$ and $\lambda\in \rho_x+(\aaa_{M,\C}^*)_x^-$ we define, whenever convergent,
\[
J(\varphi,x,\lambda)=\int_{A_M^{M_x}\grp U_x(\A)M_x\bs \grp G_x(\A)}\varphi_\lambda(h\eta)\ dh
\]
where $\eta\in G$ is such that $x=\eta\epsilon\eta^{-1}$.
Note that the integral formally makes sense by Lemma \ref{lem: exp iso}, \eqref{eq: rhox} and Lemma \ref{lem: factor adm}
and does not depend on the choice of $\eta$, since $G_x\eta$ is determined by $x$.
In fact, it is easy to see that whenever defined,  $J(\varphi,x,\lambda)$ depends only on the $M$-orbit of $x$.
Moreover, we have
\begin{multline*}
J(\varphi,x,\lambda)=
\int_{\grp P_x(\A)\bs \grp G_x(\A)}\int_{\RAS{M_x}M}\modulus_{P_x}^{-1}(m)\varphi_\lambda(mh\eta)\ dm\ dh=\\
\int_{\grp P_x(\A)\bs \grp G_x(\A)}e^{\sprod{\lambda}{\Ht_P(h\eta)}}\int_{\RAS{M_x}M}\modulus_{P_x}^{-1}(m)e^{\sprod{\rho_x}{\Ht_M(m)}}\varphi(mh\eta)\ dm\ dh.
\end{multline*}

\subsection{Convergence of the intertwining periods}
Let $\Sigma_{P,x}=\{\alpha\in\Sigma_P:x\alpha<0\}$. For $\gamma>0$ define
\[
\domain_x(\gamma)=\rho_x+\{\lambda\in (\aaa_M^*)_x^-: \sprod{\lambda}{\alpha^\vee} >\gamma,\, \forall\alpha\in \Sigma_{P,x}\}.
\]
If $x$ is $M$-standard cuspidal and $L=L(x)$ then in the notation of Definition \ref{def: std basic} we have
\[
(\aaa_{M,\C}^*)_x^-=(\aaa_M^L)^*_\C=\{\lambda=(\lambda_1,-\lambda_1,\dots,\lambda_k,-\lambda_k,\overbrace{0,\dots,0}^l):\lambda_1,\dots,\lambda_k\in\C\}
\]
and $\domain_x(\gamma)=\rho_x+\{\lambda\in (\aaa_M^*)_x^-: \lambda_i>\gamma\}$.
Taking \eqref{eq: exp eqv simp} into account, as in \cite[Lemma 5.2.1]{MR2010737} we have
\begin{lemma}\label{lem: domain}
Let $(M,x)$ and $(M',x')$ be vertices in the graph $\grph$ defined in \S\ref{sec: exponents} such that $(M,x)\overset{n_\alpha}\searrow(M',x')$ for some
$\alpha\in \srts_P$ and $n_\alpha\in s_\alpha M$.\footnote{The condition $x\alpha\ne -\alpha$ was mistakenly omitted in [loc. cit.].}
Then
\[
\domain_x(\gamma)=s_\alpha^{-1}\domain_{x'}(\gamma)
\cap\left(\rho_x+\{\lambda\in  (\aaa_M^*)_x^-: \sprod{\lambda}{\alpha^\vee}>\gamma\}\right).
\]
\end{lemma}

We will prove the convergence of $J(\varphi,x,\lambda)$ for $\varphi\in\rdp_P(G)$ for $x$ which is $M$-cuspidal.
(We will not consider more general $x$ since they will not play any role in what follows.)

\begin{theorem} \label{thm: J conv}
There exists $\gamma>0$ such that for any $M$-cuspidal $x=\eta\epsilon\eta^{-1}\in N_G(M)\cap X$ and $\varphi\in\rdp_P(G)$ the integral
defining $J(\varphi,x,\lambda)$ is absolutely convergent for $\Re\lambda\in \domain_x(\gamma)$.
Moreover, for any compact set $D$ of $\domain_x(\gamma)$ there exists $N>0$ such that
\[
\int_{A_M^{M_x}\grp U_x(\A)M_x\bs \grp G_x(\A)}\abs{\varphi_\lambda(h\eta)}\ dh
\ll_D\sup_{m\in \siegel_M^1, k\in K}\abs{\varphi(mk)}\norm{m}^N
\]
for all $\lambda
\in D+\iii(\aaa_M^*)_x^-$.
\end{theorem}

In the rest of this section we will prove Theorem \ref{thm: J conv}. Recall the notation of \S\ref{sec: rddef}.

\begin{proposition}\label{prop: alt J conv}
There exist $R>0$ and $\gamma>0$ such that for any $M$-cuspidal $x=\eta\epsilon\eta^{-1}$ the integral
\[
J(\theta^M_f,x,\lambda)=\int_{A_M^{M_x}\grp U_x(\A)M_x\bs \grp G_x(\A)}(\theta^M_f)_\lambda(h\eta)\ dh
\]
is absolutely convergent for any $f\in\clss_R(\aaa_0^M)$ uniformly for $\Re\lambda$ in a compact subset of $\domain_x(\gamma)$.
\end{proposition}

Before proving the proposition let us explain how it implies Theorem \ref{thm: J conv}.
Let $R>0$ be as in Proposition \ref{prop: alt J conv} and let $f=e^{-R\norm{\cdot}}$.
It follows from \eqref{eq: bdexpnm} that there exists $N>0$ such that
\[
\norm{m}^{-N}\ll f(\Ht_0^M(m))\ll \theta^M_f(mk),\ \ \ m\in\siegel_M^1,\,k\in K.
\]
Hence, for any $\varphi\in \rdp_P(G)$ we have
\[
\abs{\varphi(g)}\ll\sup_{m\in \siegel_M^1, k\in K}\abs{\varphi(mk)}\norm{m}^N\abs{\theta^M_f(g)},\ g\in \grp G(\A).
\]
Thus Proposition \ref{prop: alt J conv} implies Theorem \ref{thm: J conv}.

We first prove Proposition \ref{prop: alt J conv} in the case where $x$ is $M$-minimal.
Let $\grp L=\grp L(x)$ and let $\grp Q=\grp L\ltimes \grp V$ be the parabolic subgroup of $\grp G$ with Levi subgroup $\grp L$.
It suffices to consider the function $f(v)=e^{-R\norm{v}}$ and $\lambda$ real.
In particular, we may assume without loss of generality that $f$ is non-negative.
We have
\begin{multline*}
J(\theta^M_f,x,\lambda)=\int_{\grp Q_x(\A)\bs \grp G_x(\A)} \int_{\grp P_x(\A)\bs \grp Q_x(\A)}\modulus_{Q_x}^{-1}(q)
e^{\sprod{\lambda}{H_P(qh\eta)}}\\
\int_{\RAS{M_x}M}\modulus_{P_x}^{-1}(m)e^{\sprod{\rho_x}{H_P(m)}}\theta^M_f(mqh\eta)\ dm \ dq \ dh.
\end{multline*}
Since $x\in L$, we have $xQx^{-1}=Q$ and $Q_x$ is a parabolic subgroup of $G_x$.
Thus, the variable $h$ is integrated over a compact set and it is enough to show that the two inner integrals converge
uniformly for $h$ in a compact set.
Recall that by Lemma \ref{lem: stab vx} we have $\grp U_x=\grp V_x$ and $\modulus_{Q_x}\rest_{\grp P_x(\A)}=\modulus_{P_x}$.
Therefore we can rewrite the two inner integrals as
\begin{equation} \label{eq: inner integral}
\int_{\grp M_x(\A)\bs \grp L_x(\A)} e^{\sprod{\lambda-\rho_x}{H_P(lh\eta)}} \int_{\RAS{M_x}M}
\modulus_{Q_x}^{-1}(l)e^{\sprod{\rho_x}{H_P(mlh\eta)}}\modulus_{P_x}^{-1}(m)\theta^M_f(mlh\eta)\ dm \ dl.
\end{equation}
Note that $e^{\sprod{\lambda-\rho_x}{H_P(\cdot h\eta)}}$ as well as the inner integral are left $\grp M_x(\A)$-invariant.
By the description of $M_x$ in \S\ref{subsec: orbits}, $\RAS{M_x}M$ is a product of adelic quotients of finite volume and adelic quotients of the form
$\RAS{GL_r\times GL_r}{\GL_{2r}}$. The integral over the finite volume quotients is bounded by the sup norm of the integrand times the volume.
For the quotients of the second kind we use the bounds of Lemma \ref{lem: bndinner}.
Together with \eqref{eq: bdexpnm} and \eqref{eq: exp bd}, we conclude that the inner integral is bounded by a constant multiple of
\[
\modulus_{Q_x}^{-1}(l)e^{\sprod{\rho_x}{H_P(lh\eta)}}\sup_{m\in \siegel_M^1}\theta^M_f(mlh\eta)\norm{m}^N\ll
\norm{l}^N\sup_{m\in \siegel_M^1}\theta^M_f(mlh\eta)\norm{m}^N
\]
for a suitable $N$. It follows from Lemma \ref{lem: pseudobnd} that for suitable $R$ and $N'$, the latter is bounded by
a constant multiple of $\norm{l}^N\norm{lh\eta}^{N'}\ll\norm{l}^{N'+N}$ (by \eqref{eq: submult}, since $h$ ranges over a compact set).
Furthermore, by \eqref{eq: exp bd} $e^{\sprod{\lambda-\rho_x}{H_P(lh\eta)}}\ll e^{\sprod{\lambda-\rho_x}{H_P(l)}}$ for $\lambda$ in a compact.
We conclude that \eqref{eq: inner integral} is bounded by a constant multiple (which is independent of $h$ and $\lambda$ if they both lie
in compact sets) of
\begin{equation} \label{eq: finalexp}
\int_{\grp M_x(\A)\bs \grp L_x(\A)} e^{\sprod{\lambda-\rho_x}{H_P(l)}}\norm{l}_{M_x\bs L_x}^{N+N'}\ dl
\end{equation}
where $\norm{l}_{M_x\bs L_x}=\inf_{m\in \grp M_x(\A)}\norm{ml}$, $l\in \grp L_x(\A)$.
Finally, the uniform convergence (for $\lambda$ in a compact) of \eqref{eq: finalexp} for suitable $\gamma$
follows from Lemma \ref{lem: aux bd}, the description of $M_x$ and $L_x$ in \S\ref{subsec: orbits}
and the $M$-cuspidality of $x$.

In order to complete the proof of Proposition \ref{prop: alt J conv} it suffices in view of Corollary \ref{cor: exp w}
to prove the following.

\begin{lemma}\label{lem: alt min red}(Cf.~\cite[Lemma 5.3.1]{MR2010737})
Suppose that $(M,x)$ and $(M',x')$ are vertices in $\grph$ and
$(M,x)\overset{n_\alpha}\searrow(M',x')$ for some $\alpha\in\srts_P$.
Assume that Proposition \ref{prop: alt J conv} holds for $(M',x')$.
Then it also holds for $(M,x)$.
Moreover, there exist $\gamma>0$ and $R>0$ such that for $\Re\lambda\in\domain_x(\gamma)$ and $f\in\clss_R(\aaa_0^M)$ we have
\begin{equation}\label{eq: scalarFE}
J(\theta^M_f,x,\lambda)=J(M(s_\alpha,\lambda)\theta^M_f,x',s_\alpha\lambda).
\end{equation}
\end{lemma}

\begin{proof}[Proof of Lemma \ref{lem: alt min red}]
It follows from Lemma \ref{lem: domain} that if $\Re\lambda\in\domain_x(\gamma)$ then
$\Re s_\alpha\lambda\in\domain_{x'}(\gamma)$.
We first show the identity \eqref{eq: scalarFE} for $f(v)=e^{-R\sqrt{1+\norm{\cdot}^2}}$ and $\lambda\in\domain_x(\gamma)$.

We have
\[
M(s_\alpha,\lambda)\theta^M_f=\theta^{M'}_{f'}
\]
where $f'$ is the function on $\aaa_0^{M'}$ such that $\widehat{f'}(s_\alpha\mu)=c_{s_\alpha}(\lambda+\mu)\hat f(\mu)$ and
\[
c_w(\nu)=\prod_{\beta\in\Sigma_{B,w}}\frac{\zeta_F^*(\sprod{\nu}{\beta^\vee})}{\zeta_F^*(\sprod{\nu}{\beta^\vee}+1)}
\]
where $\zeta_F^*(s)$ is the completed Dedekind zeta function of $F$.
In particular, by Lemma \ref{lem: rdequ}, for any $R_0< R$ we have $f'\in\clss_{R_0}(\aaa_0^{M'})$ provided that $\gamma$
(and hence $\sprod{\lambda}{\alpha^\vee}$) is sufficiently large with respect to $R_0$
(so that $\sprod{\lambda+\mu}{\beta^\vee}>2$ for all $\norm{\mu}<{R_0}$ and $\beta\in\Sigma_{B,s_\alpha}$).

Thus by assumption, $J(M(s_\alpha,\lambda)\theta^M_f,x',s_\alpha\lambda)=J(\theta^{M'}_{f'},x',s_\alpha\lambda)$ converges
provided that $\gamma$ is sufficiently large.

Let $\grp Q=\grp L\ltimes \grp V$ be the parabolic subgroup of $\grp G$ containing $\grp P$ such that $\srts_P^Q=\{\alpha\}$.

Set $\eta'=n_\alpha\eta$. Recall that $U'\cap s_\alpha Us_\alpha^{-1}=V$. We have
\begin{multline} \label{eq: tripjms}
J(M(s_\alpha,\lambda)\theta^M_f,x',s_\alpha\lambda)=\\ \int_{\grp{P'}_{x'}(\A)\bs \grp G_{x'}(\A)}
\int_{\RAS{M'_{x'}}{M'}} \int_{\grp V(\A)\bs \grp{U'}(\A)}(\theta^M_f)_\lambda(n_\alpha^{-1} umh\eta')\ du \ \modulus_{P'_{x'}}(m)^{-1}\ dm \ dh
\end{multline}
where the triple integral converges absolutely since the integrand is non-negative.
By a change of variable $u\mapsto mum^{-1}$, an exchange of the order of integration
and \eqref{eq: pxp'x'} we rewrite the above integral as
\[
\int_{\grp{P'}_{x'}(\A)\bs \grp G_{x'}(\A)}\int_{\grp V(\A)\bs \grp{U'}(\A)}\int_{\RAS{M'_{x'}}{M'}}
(\theta^M_f)_\lambda(n_\alpha^{-1} muh\eta')\modulus_{P_{x}}(n_\alpha^{-1}mn_\alpha)^{-1}\ dm\ du \ dh.
\]
By Lemma \ref{lem: alpha} \eqref{part: int unip} we get
\[
\int_{\grp{P'}_{x'}(\A)\bs \grp G_{x'}(\A)}\int_{(n_\alpha \grp U_x n_\alpha^{-1})(\A) \bs \grp{U'}_{x'}(\A)}
\int_{\RAS{M'_{x'}}{M'}}(\theta^M_f)_\lambda(n_\alpha^{-1} muh\eta') \ \modulus_{P_{x}}(n_\alpha^{-1}mn_\alpha)^{-1}\ dm \ du \ dh.
\]
By Lemma \ref{lem: alpha} \eqref{part: int f} this equals
\[
\int_{n_\alpha \grp P_x(\A) n_\alpha^{-1} \bs \grp G_{x'}(\A)}
\int_{\RAS{M'_{x'}}{M'}}(\theta^M_f)_\lambda(n_\alpha^{-1} mh\eta') \ \modulus_{P_{x}}(n_\alpha^{-1}mn_\alpha)^{-1}\ dm\ dh.
\]
Applying the change of variables $h\mapsto n_\alpha hn_\alpha^{-1}$, $m\mapsto n_\alpha mn_\alpha^{-1}$ this becomes
\[
\int_{\grp P_x(\A) \bs \grp G_x(\A)}
\int_{\RAS{M_x}{M}}(\theta^M_f)_\lambda(mh\eta) \ \modulus_{P_x}(m)^{-1}\ dm\ dh=J(\theta^M_f,x,\lambda)
\]
as required.

To obtain the relation \eqref{eq: scalarFE} for general $f\in\clss_R(\aaa_0^M)$ and $\lambda\in\rho_x+(\aaa_{M,\C}^*)_x^-$ with
$\Re\lambda\in\domain_x(\gamma)$ we use the same argument as above where now, the absolute convergence of
the triple integral on the right-hand side of \eqref{eq: tripjms} is guaranteed by its convergence in the previously considered case
(since $f\ll_{f,R} e^{-R\sqrt{1+\norm{\cdot}^2}}$ for any $f\in\clss_R(\aaa_0^M)$).
\end{proof}

This completes the proof of Proposition \ref{prop: alt J conv} and therefore also of Theorem \ref{thm: J conv}.

We mention the following consequence.

\begin{corollary}\label{cor: fe}
Suppose that $(M,x)\overset{n_\alpha}\searrow(M',x')$ in $\grph_{\cusp}$ for some $\alpha\in\srts_P$.
Then for suitable $\gamma>0$ and for any $\varphi\in\rdp_P(G)$ we have
\[
J(\varphi,x,\lambda)=J(M(s_\alpha,\lambda)\varphi,x',s_\alpha\lambda), \ \ \Re\lambda\in\domain_x(\gamma).
\]
\end{corollary}

\begin{proof}
The argument of the proof of Lemma \ref{lem: alt min red} shows that
\begin{multline*}
J(M(s_\alpha,\lambda)\varphi,x',s_\alpha\lambda)\\=
\int_{\grp{P'}_{x'}(\A)\bs \grp G_{x'}(\A)}\int_{\RAS{M'_{x'}}{M'}}
\int_{(\grp{U'}\cap s_\alpha \grp Us_\alpha^{-1})(\A)\bs \grp{U'}(\A)}
\varphi_\lambda(n_\alpha^{-1} umh\eta')\ du \ \modulus_{P'_{x'}}(m)^{-1}\ dm \ dh
\\=\int_{\grp P_x(\A) \bs \grp G_x(\A)}\int_{\RAS{M_{x}}M}
\varphi_\lambda(mh\eta) \ \modulus_{P_x}(m)^{-1}\ dm\ dh=J(\varphi,x,\lambda).
\end{multline*}
It is justified by the absolute convergence of $J(\varphi,x,\lambda)$.
\end{proof}

\subsection{An unramified formula}\label{sec: unram}
Let us go back to the setup of \S\ref{sec: GL2nGLnGLn} (using the notation of that section) and carry out the unramified local computation pertaining
to the pair $(\grp G,\grp H)=(\grp{GL}_{2n},\grp{GL}_n\times\grp{GL}_n)$.
For this section let $F$ be a local field with normalized absolute value $\abs{\cdot}$.

Let $\grp A_n$ be the torus of diagonal matrices in $\grp{GL}_n$, $\grp U_n$ the subgroup of upper-unitriangular matrices, $\grp R_n=\grp A_n\ltimes\grp U_n$
the Borel subgroup of upper triangular matrices and $K_n$ the standard maximal compact subgroup of $\GL_n$,
i.e., $K_n=\grp{GL}_n(\OOO_F)$ for $F$ non-archimedean $O(n)$ if $F=\R$ and $U(n)$ if $F=\C$. Let
\[
\rho_n=\frac12(n-1,n-3,\dots,1-n)\in \aaa_{A_n}^*.
\]
For $\lambda=(\lambda_1,\dots,\lambda_n)\in \C^n$ let $\Xi_\lambda$ be the function on $G$ given by
\[
\Xi_\lambda(g)=\int_{R_n\bs\GL_n}e^{\sprod{(\lambda,-\lambda)+\rho_{2n}}{\Ht_0(\sm m{}{}mg)}}\ dm.
\]
Thus, $\Xi_\lambda$ is left invariant under $\{\sm mX{}m:m\in\GL_n,X\in\operatorname{Mat}_{n\times n}\}$, right $K_{2n}$-invariant and the restriction of
$\modulus_P^{-\frac12}\Xi_\lambda$ to $\{\sm m{}{}{I_n}:m\in\GL_n\}$ is the spherical function on $\GL_n$ with parameter $\lambda$.
For convenience take $\theta_n=\sm{}{I_n}{I_n}{}\sm{I_n}{I_n}{}{I_n}$ (so that $\theta_n\in P\eta_nH$ where $\eta_n$ is defined in \eqref{def: eta}) and set
\begin{multline*}
J_n(\lambda)=\int_{\GL_n^\triangle\bs\GL_n\times\GL_n}\Xi_\lambda(\theta_n\sm{m_1}{}{}{m_2})\ dm_1\ dm_2\\=
\int_{R_n^\triangle\bs\GL_n\times\GL_n}e^{\sprod{(\lambda,-\lambda)+\rho_{2n}}{H_0(\theta_n\sm{m_1}{}{}{m_2})}}\ dm_1\ dm_2.
\end{multline*}
By the Iwasawa decomposition for $\GL_n$ we have
\[
J_n(\lambda)=\int_{A_n}\int_{U_n} e^{\sprod{-2\rho_n}{H_{A_n}(a)}+\sprod{(\lambda,-\lambda)+\rho_{2n}}{H_0(\theta_n\sm{I_n}{}{}{ua})}}\ du \ da.
\]

Let $w\in S_{2n}$ be the permutation defined by
\[
w(i)=2i-1,\ w(n+i)=2i,\ i=1,\dots,n,
\]
viewed also as a permutation matrix $(\delta_{i,w(j)})$ in $\GL_{2n}$.
Set $\xi=w\theta_n w^{-1}=\diag(\theta_1,\dots,\theta_1)$ and note that
\[
w(\lambda,-\lambda)=(\lambda_1,-\lambda_1,\dots,\lambda_n,-\lambda_n)
\]
and for $u=(u_{i,j})\in U_n$ we have
\[
\xi w\sm{I_n}{}{}{u}w^{-1}\xi^{-1}=\left(\begin{smallmatrix} I_2 & & \beta_{i,j} \\ & \ddots & \\ & & I_2\end{smallmatrix}\right)
\ \ \text{ where }\ \
\beta_{i,j}=\left(\begin{smallmatrix} u_{i,j} & 0 \\ u_{i,j} & 0 \end{smallmatrix}\right).
\]
On the other hand,
\[
U_{2n}\cap wU_{2n}w^{-1}=\{u=(u_{i,j})\in U_{2n}:u_{2i,2j-1}=0\ \forall i=1,\dots,n\text{ and }\ j=i+1,\dots,n\}.
\]
Therefore, $\xi w\diag(I_n,U_n)w^{-1}\xi^{-1}$ is a set of representatives for $(U_{2n}\cap wU_{2n}w^{-1})\bs U_{2n}$.
Writing
\[
\theta_n\sm{I_n}{}{}{u}=w^{-1} (\xi w \sm{I_n}{}{}{u}w^{-1} \xi^{-1})\xi w
\]
we get that
\[
J_n(\lambda)=\int_{A_n} e^{\sprod{-2\rho_n}{H_{A_n}(a)}}\int_{(U_{2n}\cap wU_{2n}w^{-1})\bs U_{2n}}
e^{\sprod{(\lambda,-\lambda)+\rho_{2n}}{H_0(w^{-1} u\xi w\sm{I_n}{}{}{a})}} \ du \ da.
\]
The inner integral is a standard intertwining operator applied to the unramified section.
By a familiar computation for any $\mu=(\mu_1,\dots,\mu_{2n})\in \aaa_{A_{2n},\C}^*$ we have
\[
\int_{(U_{2n}\cap wU_{2n}w^{-1})\bs U_{2n}} e^{\sprod{\mu+\rho_{2n}}{H_0(w^{-1}ug)}} \ du=
c_w(\mu)e^{\sprod{w\mu+\rho_{2n}}{H_0(g)}}, \ g\in\grp{GL}_{2n}
\]
where
\[
c_w(\mu)=\prod_{i<j;w(i)>w(j)}\frac{L(\mu_i-\mu_j,\trivchar_F)}{L(\mu_i-\mu_j+1,\trivchar_F)}.
\]
The integral converges provided that $\Re\mu_i>\Re\mu_j$ for all $i<j$ such that $w(i)>w(j)$.

Note that $\{1\le i< j\le 2n: w(i)>w(j)\}=\{(i,n+j):1\le i<j\le n\}$ and therefore
\[
c_w(\lambda,-\lambda)=\prod_{i<j}\frac{L(\lambda_i+\lambda_j,\trivchar_F)}{L(\lambda_i+\lambda_j+1,\trivchar_F)}
\]
and the integral over $(U_{2n}\cap wU_{2n}w^{-1})\bs U_{2n}$ converges for $\Re\lambda_i+\Re\lambda_j>0$, $i<j$.
It follows that
\[
J_n(\lambda)=c_w(\lambda,-\lambda)\int_{A_n}e^{\sprod{-2\rho_n}{H_{A_n}(a)}+\sprod{w(\lambda,-\lambda)+\rho_{2n}}{H_0(\xi w\sm{I_n}{}{}{a})}} \ da.
\]
Note that $w\sm{I_n}{}{}{a} w^{-1}=\diag(1,a_1,1,a_2,\dots,1,a_n)$ for $a=\diag(a_1,\dots,a_n)\in A_n$ and
\[
e^{\sprod{2\rho_n}{H_{A_n}(a)}}=\prod_{i=1}^n \abs{a_i}^{n+1-2i}=\prod_{i=1}^n \abs{\det(\theta_1 \sm{1}{}{}{a_i})}^{n+1-2i}=
e^{\sprod{\nu}{H_0(\xi w\sm{I_n}{}{}{a})}}
\]
where $\nu=(n-1,n-1,n-3,n-3\dots,1-n,1-n)\in \aaa_{A_{2n}}^*$. Since
\[
\rho_{2n}-\nu=(\frac12,-\frac12,\dots,\frac12,-\frac12)=(\rho_2,\dots,\rho_2)
\]
we get that
\[
J_n(\lambda)=c_w(\lambda,-\lambda)\prod_{j=1}^n J_1(\lambda_j).
\]
The convergence and computation of $J_n(\lambda)$ reduces therefore to the case $n=1$.
As in \cite[Lemma 5.2]{MR2060496} it is easy to calculate the integral
\[
J_1(s)=\int_{F^*} e^{\sprod{(s+\frac12,-s-\frac12)}{H_{A_2}(\sm 0a1a)}}\ da.
\]
Assume first that $F$ is non-archimedean and let $q$ be the size of its residual field.
Then
\[
e^{\sprod{(s+\frac12,-s-\frac12)}{H_{A_2}(\sm 0a1a)}}=(\min\{\abs{a},\abs{a}^{-1}\})^{s+\frac12}.
\]
It follows that
\[
J_1(s)=\left[1+2\int_{\abs{a}<1}\abs{a}^{s+\frac12}\ da\right]=
\left[1+2\sum_{n=1}^\infty q^{-n(s+\frac12)}\ da\right]=\frac{1+q^{-s-\frac12}}{1-q^{-s-\frac12}}
\]
is convergent for $\Re s>-\frac12$.

Consequently, the integral defining $J_n(\lambda)$ is convergent whenever $\Re\lambda_i>-\frac12$, $i=1,\dots,n$ and
$\Re\lambda_i+\Re\lambda_j>0$ for all $i<j$ and
\[
J_n(\lambda)=
\left[\prod_{i<j}\frac{1-q^{-(\lambda_i+\lambda_j+1)}}{1-q^{-(\lambda_i+\lambda_j)}}\right]\left[\prod_{i=1}^n
\frac{1+q^{-\lambda_i-\frac12}}{1-q^{-\lambda_i-\frac12}}\right].
\]

Suppose now that $F$ is archimedean. Note that
\[
e^{\sprod{(s+\frac12,-s-\frac12)}{H_{A_2}(\sm 0a1a)}}=\begin{cases} \left(\frac{\abs{a}}{1+\abs{a}^2}\right)^{s+\frac12}  & F=\R \\\left(\frac{\abs{a}}{(1+\abs{a})^2}\right)^{s+\frac12} & F=\C \end{cases}
\]
and therefore once again the integral converges for $\Re s>-\frac12$. Up to normalization of measures (independently of $s$) we have
\[
J_1(s)=2\int_0^\infty \left(\frac{x}{1+x^2}\right)^{(s+\frac12)[F:\R]} d^\times x.
\]
Since
\[
2\int_0^\infty \left(\frac{x}{1+x^2}\right)^t\ d^\times x=
\int_0^\infty\frac{z^{\frac t2-1}}{{1+z}^t}\ dz=B(\frac t2,\frac t2)=\frac{\Gamma(\frac t2)^2}{\Gamma(t)}
\]
where $B(x,y)$ is the beta function, we conclude that for any local field $F$ we have, for a suitable normalization of measures,
\[
J_1(s)=\frac{L(s+\frac12,\trivchar_F)^2}{L(2s+1,\trivchar_F)}
\]
and
\[
J_n(\lambda)=\left[\prod_{i<j}\frac{L(\lambda_i+\lambda_j,\trivchar_F)}{L(\lambda_i+\lambda_j+1,\trivchar_F)}\right]
\left[\prod_{i=1}^n\frac{L(\lambda_i+\frac12,\trivchar_F)^2}{L(2\lambda_i+1,\trivchar_F)}\right].
\]

We remark that if $\pi$ is the unramified principal series representation of $\GL_n$ induced from $e^{\sprod{\lambda}{H_{A_n}(\cdot)}}$ then we get that
\begin{equation}\label{eq: unram}
J_n(\lambda)=L(\frac12,\pi)^2\frac{L(0,\pi,\wedge^2)}{L(1,\pi,\Sym^2)}.
\end{equation}

\subsection{}
We go back to the setup of \S\ref{sec: setup}.
Let $\grp{P}=\grp{M}\ltimes\grp{U}$ be a parabolic subgroup of $\grp{G}$, $x\in X\cap N_G(M)$ an $M$-cuspidal element and $\sigma$ an irreducible,
cuspidal, automorphic representation of $\grp{M}(\A)$.
Let $I(\sigma)$ be the space of smooth functions $\varphi$ on $\grp U(\A)M\bs \grp G(\A)$ such that
$m\mapsto\modulus_P(m)^{-\frac12}\varphi(mg)$ belongs to the space of $\sigma$ for all $g\in \grp G(\A)$.
Thus $I(\sigma)\subset \rdp_P(G)$ and we can identify it with the parabolic induction $\Ind_{\grp{P}(\A)}^{\grp{G}(\A)}(\sigma)$.
Denote the restriction of $J(x,\lambda)$ to $I(\sigma)$ by $J(x,\sigma,\lambda)$.
The analytic continuation and functional equation of $J(x,\sigma,\lambda)$ can in principle be inferred from those of the Eisenstein series,
as in \cite{MR2010737, MR2254544}. We will not carry this out here since our focus is somewhat different.

We can also factorize $J(x,\sigma,\lambda)$ and evaluate the local factors at the unramified places as follows.
By Lemma \ref{lem: s,r,r} and Corollary  \ref{cor: fe} we may assume that $x$ is $M$-standard \rlvt.
That is (see Remark \ref{rmk: cusp std rlvt}) $M=M_{(r_1,r_1,\dots,r_k,r_k,s_1,\dots,s_l;0)}$, $s_j$ is either even or $1$ for every $j=1,\dots,l$ and $x\in N_G(M)\cap X$ is
$M$-conjugate to $\inj(\sm{0}{I_{r_1}}{I_{r_1}}{0},\dots,\sm{0}{I_{r_k}}{I_{r_k}}{0},h_1,\dots,h_l)$
where $h_j=\sm{0}{I_{s_j/2}}{I_{s_j/2}}{0}$ if $s_j$ is even and $h_j=\pm 1$ if $s_j=1$.
Let $Q=L\ltimes V$ with $L=L(x)=M_{(2r_1,\dots,2r_ks_1,\dots,s_l;0)}$.
In this case, $J(x,\sigma,\lambda)$ is identically zero unless $\sigma$ is of the form
\[
\sigma\simeq \sigma_1\otimes \tilde\sigma_1\otimes \cdots\otimes \sigma_k\otimes\tilde\sigma_k\otimes \tau_1\otimes\cdots\otimes \tau_l
\]
where for all $i$, $\sigma_i$ is a cuspidal representation of $\GL_{r_i}(\A)$, $\tilde\sigma_i$ is the contragredient of $\sigma_i$,
and for all $j$, $\tau_j$ is either the trivial character of $\GL_{s_j}(\A)$ if $s_j=1$
or a cuspidal $\GL_{s_j/2}\times \GL_{s_j/2}$-distinguished automorphic representation of $\GL_{s_j}$ otherwise (i.e., if $s_j$ is even).
In this case, for $\varphi\in I(\sigma)$ we have
\[
J(\varphi,x,\sigma,\lambda)=\int_{\grp Q_x(\A)\bs \grp G_x(\A)} \int_{\grp M_x(\A)\bs \grp L_x(\A)}\modulus_{Q_x}^{-1}(l)
\int_{\RAS{M_x}M}\modulus_{P_x}^{-1}(m)\varphi_\lambda(mlh\eta)\ dm \ dl \ dh.
\]
We can factorize this integral into local integrals.
This will involve the factorization of the inner integral which in turn is a product of Petersson inner products
for the $\sigma_i$'s, $\GL_{s_j/2}\times \GL_{s_j/2}$-periods for $\tau_j$, considered in \cite{MR1159108, MR1241129}, for $s_j$ even
and volume factors for $s_j=1$.
The integral over $h$ does not affect the unramified computation.
The integral over $l$ affects the unramified computation only in the $\GL_{2r_1}\times\dots\times\GL_{2r_k}$ component of $L_x$.


Using the unramified computation of the previous section we conclude that
for a sufficiently large finite set of places $S$ of $F$,
there is a linear form $J_S=\otimes_{v\in S}J_v$ on $\otimes_{v\in S} I(\sigma_v)$ such that
if we denote $\lambda-\rho_x=(\lambda_1,-\lambda_1,\dots,,\lambda_k,-\lambda_k,0,\dots,0)$
then for every pure tensor $\varphi\in I(\sigma)^{K^S}$ we have
$J(\varphi,x,\sigma,\lambda)=L_\sigma^S(\lambda) J_S(\varphi_S)$ where
\begin{multline*}
L_\sigma(\lambda) =\left[\prod_{i=1}^k
L(\lambda_i+\frac12,\sigma_i)^2\frac{L(2\lambda_i,\sigma_i,\wedge^2)\Res_{s=1}L(s,\sigma_i\times\tilde\sigma_i)}{L(2\lambda_i+1,\sigma_i,\Sym^2)}\right]\times
\\\prod_{j=1}^l
\begin{cases} \Res_{s=1}\zeta_F(s) & s_j=1, \\ L(\frac12,\tau_j) \Res_{s=1}L(s,\tau_j,\wedge^2) & s_j \text{ is even.}\end{cases}
\end{multline*}

\section{$H$-periods of pseudo Eisenstein series} \label{sec: Hperiods}
We can now state and prove the formula for the $H$-period of pseudo Eisenstein series.
Let $\grp P=\grp M\ltimes \grp U$ be a parabolic subgroup of $\grp G$.
Recall the notation of \S\ref{sec: rddef}.

\begin{theorem} \label{thm: pesudoperiod}
There exists $R>0$ such that the integral
\[
\int_{\AS{H}}\theta_\phi(h)\ dh
\]
converges absolutely for any $\phi\in\rd{R}$ and vanishes unless $M\subset\inj(\GL_{2n})$.
Moreover, there exist $\gamma>0$ and $R>0$ such that for any $\phi\in\rdsmth{R}$ we have
\[
\int_{\AS{H}}\theta_\phi(h)\ dh
=\sum_{x}\int_{\lambda_x+\iii(\aaa_M^*)_x^-}J(\phi[\lambda],x,\lambda)\ d\lambda,
\]
a finite sum of absolutely convergent integrals
where $x$ ranges over a set of representatives of the $M$-cuspidal orbits in $X$,
and for any $x$ we fix $\lambda_x\in \domain_x(\gamma)$ such that $\norm{\lambda_x}<R$.
\end{theorem}

\begin{proof}
By Lemma \ref{lem: bndpseudo} there exists $R>0$ such that $g\mapsto\sum_{\gamma\in P\bs G}\abs{\phi(\gamma g)}$
is bounded on $\grp G(\A)$, and in particular integrable over $\AS{H}$, for all $\phi\in\rd{R}$.
Therefore we can write
\[
\int_{\AS{H}}\theta_\phi(h)\ dh=\sum_{x}I_x(\phi)
\]
as a (finite) sum of absolutely convergent integrals where $x$ ranges over a set of representatives for the $P$-orbits in $X$ and for each $x\in X$
\[
I_x(\phi)=\int_{P_x\bs \grp G_x(\A)}\phi(h\eta_x)\ dh=
\int_{\grp P_x(\A)\bs \grp G_x(\A)}\int_{P_x\bs \grp P_x(\A)}
\phi(p h\eta_x)\modulus_{P_x}(p)^{-1}\ dp\ dh
\]
where $\eta_x\in G$ is such that $x=\eta_x\epsilon\eta_x^{-1}$.
By the cuspidality condition on $\phi$, it follows from Lemma \ref{lem: nonadmissible} \eqref{part: 0int} that $I_x(\phi)=0$ unless $x$ is $M$-admissible.
By Lemma \ref{lem: M orbit map}, we get
\[
\int_{\AS{H}}\theta_\phi(h)\ dh=\sum_{x}I_x(\phi).
\]
where $x$ ranges over a set of representatives of the $M$-orbits in $N_G(M)\cap X$.
Suppose that $x\in N_G(M)\cap X$. By Lemma \ref{lem: factor adm} we have $\grp P_x=\grp M_x\ltimes \grp U_x$. Since $\phi$ is $\grp U(\A)$-invariant
we get that
\[
I_x(\phi)=\int_{\grp P_x(\A)\bs \grp G_x(\A)}\int_{M_x\bs \grp M_x(\A)}\phi(mh\eta_x)\modulus_{P_x}^{-1}(m)\ dm \ dh.
\]
Clearly, $\modulus_{P_x}$ is trivial on $\grp M_x(\A)^1$. By Corollary \ref{cor: xvanish}
\[
\int_{M_x\bs \grp M_x(\A)^1}\phi(mg)\ dm=0,\ g\in \grp G(\A)
\]
and hence also $I_x(\phi)=0$, unless $x$ is $M$-cuspidal.
In particular,
\[
\int_{\AS{H}}\theta_\phi(h)\ dh=0
\]
unless $M\subset\inj(\GL_{2n})$.
Assume therefore that $M\subset\inj(\GL_{2n})$ and $x$ is $M$-cuspidal.
By Lemmas \ref{lem: factor adm} and \ref{lem: exp iso}, we can write
\begin{multline*}
I_x(\phi)=\int_{A_M^{M_x}\grp U_x(\A)M_x\bs \grp G_x(\A)}\int_{A_M^{M_x}}
\phi(ah\eta_x)\modulus_{P_x}^{-1}(a)\ da\ dh\\=
\int_{A_M^{M_x}\grp U_x(\A)M_x\bs \grp G_x(\A)}\int_{(\aaa_M)_x^+}
\phi(e^\nu h\eta_x)e^{-\sprod{\rho_P+\rho_x}{\nu}} \ d\nu\ \ dh.
\end{multline*}

Now assume that $\phi\in\rdsmth{R}$.
By partial Fourier inversion formula with respect to the subspace $(\aaa_M)_x^+$ of $\aaa_M$ we have
\[
I_x(\phi)=\int_{A_M^{M_x}\grp U_x(\A)M_x\bs \grp G_x(\A)}
\left( \int_{\lambda_x+\iii(\aaa_M^*)_x^-}
\phi[\lambda]_\lambda(h\eta_x) \ d\lambda\right)\ \ dh
\]
for any $\lambda_x\in \rho_x+(\aaa_{M,\C}^*)_x^-$ such that $\norm{\lambda_x}<R$.
By Theorem \ref{thm: J conv} and \eqref{eq: bndonFT} the double integral converges
provided that $\Re\lambda_x\in \domain_x(\gamma)$ for suitable $\gamma$ and $R$.
Changing the order of integration we obtain
\[
I_x(\phi)=  \int_{\lambda_x+\iii(\aaa_M^*)_x^-}J(\phi[\lambda],x,\lambda)\ d\lambda.
\]
The theorem follows.
\end{proof}

\begin{remark}
Theorem \ref{thm: pesudoperiod} is stated for a general $\phi$ but in the case where $\phi$ is spectrally supported
on a single irreducible cuspidal representation $\pi$ of $\grp M(\A)$,
the terms in the formula for the $H$-period depend heavily on $\pi$.
Theorem \ref{thm: pesudoperiod} is an analogue of the inner product formula for pseudo Eisenstein series
(\cite[Theorem II.2.1]{MR1361168}), which is the point of departure for Langlands's spectral decomposition of
$L^2(\AS{G})$.
The natural next step would be to perform a residue calculus on the formula provided by Theorem \ref{thm: pesudoperiod}.
This would no doubt sharpen the results of \S\ref{sec: main result} below.
However, we will not pursue this matter here.
\end{remark}

\section{The distinguished spectrum} \label{sec: dist spec}
In this section let $\grp G$ be a reductive group over a number field $F$ and $\grp H$ a reductive subgroup defined over $F$.

\subsection{}
We will define a subspace of $L^2(\AS{G})$, denoted $L^2_{H\dist}(\AS{G})$,
which measures the part of the spectrum which is distinguished with respect to $H$.
First, let $L^2(\AS{G})_{\Hconv}$ be the subspace of $L^2(\AS{G})$ consisting of $\varphi$ such that
the integral $\int_{\RAS HG}\abs{(f*\varphi)(h)}\ dh$ converges for any $f\in C_c(\grp G(\A))$
(where the latter denotes the space of continuous, compactly supported functions on $\grp G(\A)$).
The space $L^2(\AS{G})_{\Hconv}$ contains the space of rapidly decreasing functions on $\AS{G}$
(cf., \cite[Proposition 1]{MR1233493}).
(If $\grp H\cap\grp G^{\der}$ is semisimple then $L^2(\AS{G})_{\Hconv}$ contains the space of bounded measurable functions on $\AS{G}$.)
In particular, $L^2(\AS{G})_{\Hconv}$ is dense in $L^2(\AS{G})$.
Let
\[
L^2(\AS{G})_{\Hconv}^\circ=
\{\varphi\in L^2(\AS{G})_{\Hconv}:\\\int_{\RAS HG}(f*\varphi)(h)\ dh=0\text{ for all }f\in C_c(\grp G(\A))\}.
\]
We can define the `strong' $H$-distinguished spectrum $L^2_{H\dist}(\AS{G})^{\stng}$ to be the orthogonal complement in
$L^2(\AS{G})$ of $L^2(\AS{G})_{\Hconv}^\circ$.
It is not clear to what extent is this definition sensible in general
(especially if $L^2(\AS{G})_{\Hconv}\ne L^2(\AS{G})$). However, we will see in the next section
that it is at least useful in one case.

More generally, for any subspace $\mathcal{C}$ of $L^2(\AS{G})_{\Hconv}$ define
\[
\mathcal{C}_H^\circ=\{\phi\in\mathcal{C}:\int_{\RAS HG}(f*\phi)(h)=0\text{ for every }f\in C_c(\grp G(\A))\}.
\]
Note that if for any $\phi\in\mathcal{C}$, $\phi$ is continuous and the
integral $\int_{\RAS HG}\abs{\phi(hg)}\ dh$ converges for all $g\in \grp G(\A)$ then
\[
\mathcal{C}_H^\circ=\{\phi\in\mathcal{C}:\int_{\RAS HG}\phi(hg)=0\text{ for all }g\in \grp G(\A)\}.
\]
To define $L^2_{H\dist}(\AS{G})$ we will take $\mathcal{C}$ to be the space of pseudo Eisenstein series.
To make this more precise we recall some standard facts and terminology from \cite{MR1361168}.

\subsection{}
\emph{For the rest of this section all external references below are from \cite{MR1361168}.}
We denote by $\Pi_{\cusp}(A_G\bs\grp G(\A))$ the set of equivalence classes of irreducible cuspidal representations of $\grp G(\A)$ whose central
character is trivial on $A_G$.
Recall that a cuspidal datum for $G$ (II.1.1) is a pair $(M,\pi)$
consisting of a Levi subgroup $M$ of $G$ and $\pi\in\Pi_{\cusp}(A_M\bs\grp M(\A))$.
Two cuspidal data $(M,\pi)$ and $(M',\pi')$ are equivalent if they are conjugate (in the obvious sense) by an element of $G$.
Let $\cuspdata$ be the set of equivalence classes of cuspidal data.

For any cuspidal data $(M,\pi)$ let $L^2_{\cusp,\pi}(\AS{M})$ be the $\pi$-isotypic component
of the cuspidal spectrum $L^2_\cusp(\AS{M})$ and let
\begin{multline*}
L^2_{\cusp,\pi}(\grp U(\A)M\bs\grp G(\A))=\{\varphi:\grp U(\A)M\bs\grp G(\A)\rightarrow\C\text{ measurable}\,|\\
\modulus_P^{-\frac12}\varphi(\cdot g)\in L^2_{\cusp,\pi}(\AS{M})\text{ for all }g\in\grp G(\A),
\int_{A_M\grp U(\A)M\bs\grp G(\A)}\abs{\varphi(g)}^2\ dg<\infty\}.
\end{multline*}
For any finite set of $K$-types $\Ktypes$, the space $L^2_{\cusp,\pi}(\grp U(\A)M\bs\grp G(\A))^{\Ktypes}$ (the direct sum of $K$-isotypic components
pertaining to $\Ktypes$) is finite-dimensional and consists of smooth functions.

Let $\cuspdatum\in\cuspdata$ and $\Ktypes$ a finite set of $K$-types.
For $R\gg1$ and $(M,\pi)\in\cuspdatum$ let $P_{(M,\pi)}^{R,\Ktypes}$ be the space defined in (V.2.1)
(cf.~(II.1.4)) namely (in the notation of \S\ref{sec: P^r})
\[
P_{(M,\pi)}^{R,\Ktypes}=P^R((\aaa_M^G)^*;L^2_{\cusp,\pi}(\grp U(\A)M\bs\grp G(\A))^{\Ktypes})\simeq
P^R((\aaa_M^G)^*)\otimes L^2_{\cusp,\pi}(\grp U(\A)M\bs\grp G(\A))^{\Ktypes}.
\]
We denote the value of $\phi\in P_{(M,\pi)}^{R,\Ktypes}$ at $\lambda$ by $\phi[\lambda]$.
This is consistent with the notation of \S\ref{sec: rddef} if we view $\phi$ (by Mellin inversion, as in \eqref{eq: invMeltrns})
as a function in the space $\rdsmth{R}$.
In particular, the pseudo Eisenstein series $\theta_\phi$ (II.1.10) is defined for any $\phi\in P_{(M,\pi)}^{R,\Ktypes}$.
Let $P_{\cuspdatum}^{R,\Ktypes}=\dsum_{(M,\pi)\in\cuspdatum}P_{(M,\pi)}^{R,\Ktypes}$ and extend the map $\phi\mapsto\theta_\phi$ to a map
\[
\theta^\cuspdatum:P_{\cuspdatum}^{R,\Ktypes}\rightarrow L^2(\AS{G})
\]
by linearity. Let $\pseudospace_{\cuspdatum,\Ktypes}(\AS{G})$ be the image of $\theta^\cuspdatum$.
Let $L^2_\cuspdatum(\AS{G})_\Ktypes$ be the closure of $\pseudospace_{\cuspdatum,\Ktypes}(\AS{G})$ in $L^2(\AS{G})$.
Let also
\[
\pseudospace_{\cuspdatum}(\AS{G})=\cup_{\Ktypes\subset\hat K\text{ finite}}\pseudospace_{\cuspdatum,\Ktypes}(\AS{G})
\]
and let $L^2_\cuspdatum(\AS{G})$ be the closure of $\pseudospace_{\cuspdatum}(\AS{G})$ in $L^2(\AS{G})$ (II.2.4).
Note that it follows from Theorem II.2.1 that $L^2_\cuspdatum(\AS{G})$ is independent of $R$ for $R\gg1$.

The space $L^2(\AS{G})$ admits a coarse decomposition
\begin{equation} \label{eq: coarse decomposition}
L^2(\AS{G})=\hat\dsum_{\cuspdatum\in\cuspdata} L^2_\cuspdatum(\AS{G})
\end{equation}
(II.2.4).
In particular the space
\[
\pseudospace(\AS{G})=\dsum_{\cuspdatum\in\cuspdata}\pseudospace_{\cuspdatum}(\AS{G})
\]
of all pseudo Eisenstein series is dense in $L^2(\AS{G})$.

By definition, the $H$-distinguished spectrum $L^2_{H\dist}(\AS{G})$ is the orthogonal complement of
$\pseudospace(\AS{G})_H^\circ$ in $L^2(\AS{G})$.
Note that in principle this space may depend on the choice of $R$, although we do not expect it to be so.
By abuse of notation we will suppress this a priori dependence.
Alternatively, we may work instead with pseudo Eisenstein series built from the space of Paley-Wiener functions $P_{(M,\pi)}$
of II.1.2. The ensuing discussion will carry over with minor changes.

Clearly $L^2_{H\dist}(\AS{G})^{\stng}\subset L^2_{H\dist}(\AS{G})$.

\begin{remark} \label{rem: otherspaceHdist}
One may consider the orthogonal complement in $L^2(\AS{G})$ of $\mathcal{C}_H^\circ$
for other classes of functions $\mathcal{C}$ in $L^2(\AS{G})_{\Hconv}$, for instance the class
of rapidly (or alternatively, sufficiently rapidly) decreasing functions on $\AS{G}$ (perhaps with all their derivatives).
Conceivably this would coincide with $L^2_{H\dist}(\AS{G})$.
However, we will not address this question in this paper.
\end{remark}

We set
\[
L^2_{\disc,H\dist}(\AS{G})=L^2_{H\dist}(\AS{G})\cap L^2_{\disc}(\AS{G}).
\]

\subsection{} \label{sec: finerspectraldecomposition}
We turn to the finer decomposition of $L^2(\AS{G})$ which is the crux of Chapter V.
Fix $\cuspdatum\in\cuspdata$ and $\Ktypes$ as before.
By definition, a \emph{root hyperplane} in $(\aaa_M^G)^*$ is an affine hyperplane given by an equation $\sprod{\lambda}{\alpha^\vee}=c$
for some co-root $\alpha^\vee$ corresponding to $\alpha\in\srts_P$ and $c\in\R$.
There is a certain finite set $S_\cuspdatum^\Ktypes$ consisting of triples $(M,\pi,\sing)$ where $(M,\pi)\in\cuspdatum$
and $\sing$ is an affine subspace of $(\aaa_M^G)^*$ which is an intersection of root hyperplanes.
(In fact, the set $S_\cuspdatum^\Ktypes$ defined in (V.1.1) is only locally finite.
However, for our purposes we can replace it by the finite set (which implicitly depends on $\cuspdatum$) denoted by $\operatorname{Sing}^{G,\Ktypes}$ in (V.3.13)
-- see (VI.1.8).) Let $S_\cuspdatum=\cup_{\Ktypes\subset\hat K\text{ finite}}S_\cuspdatum^\Ktypes$. (This set is probably also finite
but we do not need to know this.)
The set $S_\cuspdatum$ contains the triples $(M,\pi,\sing)$ where $\sing$ is a singular hyperplane $\sprod{\lambda}{\alpha^\vee}=c$, $c>0$
for the intertwining operator corresponding to $(M,\pi)$ and $\alpha\in\srts_P$.
Let $[S_\cuspdatum^\Ktypes]$, $[S_\cuspdatum]$ be the sets of equivalence classes
of $S_\cuspdatum^\Ktypes$, $S_\cuspdatum$ respectively under the equivalence relation defined in (V.3.1).
The finer decomposition alluded to above is
\[
L^2_\cuspdatum(\AS{G})_\Ktypes=\dsum_{\singcls\in[S_\cuspdatum^\Ktypes]}L^2_\cuspdatum(\AS{G})_{\singcls,\Ktypes}
\]
and correspondingly
\[
L^2_\cuspdatum(\AS{G})=\hat\dsum_{\singcls\in[S_\cuspdatum]}L^2_\cuspdatum(\AS{G})_\singcls
\]
(Corollaries V.3.13 and V.3.14) where
\[
L^2_\cuspdatum(\AS{G})_\singcls=\overline{\sum_{\Ktypes}L^2_\cuspdatum(\AS{G})_{\singcls,\Ktypes}}.
\]
The subspaces $L^2_\cuspdatum(\AS{G})_{\singcls,\Ktypes}$
are defined in Chapter V, initially in an ad hoc fashion,
and subsequently as integrals of certain residual Eisenstein series.
For our purposes it will be useful to have a description of the subspaces $L^2_\cuspdatum(\AS{G})_{\singcls,\Ktypes}$
using a slight variant of the recipe given in (V.3.3) -- see \eqref{eq: closure} below.
For $\singcls,\singcls'\in[S_\cuspdatum]$ we write $\singcls'\succeq\singcls$ if for
$(M,\pi,\sing)\in\singcls$ we can choose $(M',\pi',\sing')\in\singcls'$ such that $(M,\pi)=(M',\pi')$ and $\sing'\supset\sing$.
(Of course, this condition does not depend on the choice of $(M,\pi,\sing)$.)
We write $\singcls'\succ\singcls$ if $\singcls'\succeq\singcls$ but $\singcls'\ne\singcls$.

Fix $m\gg1$ and for any $\singcls\in S_\cuspdatum$ let
\begin{multline*}
\tilde P_{\not\succeq\singcls}^{R,\Ktypes}=\{\phi=(\phi_{(M,\pi)})_{(M,\pi)\in\cuspdatum}\in P_\cuspdatum^{R,\Ktypes}:
\text{for any $\singcls'=[(M,\pi,\sing')]\in S_\cuspdatum$ such that }\singcls'\not\succeq\singcls,\\
\phi_{(M,\pi)}\text{ together with its derivatives of order $\le m$ vanishes on }\sing'\}.
\end{multline*}
Similarly define $\tilde P_{\not\succ\singcls}^{R,\Ktypes}$.
The space $\tilde P_{\not\succeq\singcls}^{R,\Ktypes}$ is contained (possibly strictly) in the corresponding space $P_{\singcls,T'}^{R,\Ktypes}$
defined in (V.3.3) where the partial order $\succeq$ is replaced by an essentially arbitrary total order refinement.\footnote{We recall that a posteriori
$T'$ is superfluous because of the finiteness of $\sings^{G,\Ktypes}$.}
Nevertheless, if we replace $P_{\singcls,T'}^{R,\Ktypes}$ by $\tilde P_{\not\succeq\singcls}^{R,\Ktypes}$
throughout V.3 then all the statements and proofs remain valid verbatim.
(The induction is on the codimension of $\sing$ where $(M,\pi,\sing)\in\singcls$.)
Indeed, it all boils down to the simple statement (i) on the bottom of [p. 203]
which still holds for $\tilde P_{\not\succeq\singcls}^{R,\Ktypes}$.
Consequently, we have (cf. Corollary V.3.13):
\begin{subequations}
\begin{gather}
\label{eq: closure}
L^2_\cuspdatum(\AS{G})_{\singcls,\Ktypes}=\overline{\{\theta_\phi:\phi\in\tilde P_{\not\succeq\singcls}^{R,\Ktypes}\}}\cap
\{\theta_\phi:\phi\in\tilde P_{\not\succ\singcls}^{R,\Ktypes}\}^\perp,\\
\dsum_{\singcls'\succeq\singcls}L^2_\cuspdatum(\AS{G})_{\singcls',\Ktypes}=
\overline{\{\theta_\phi:\phi\in\tilde P_{\not\succeq\singcls}^{R,\Ktypes}\}}.
\end{gather}
\end{subequations}
Here we used the fact that the right-hand sides are invariant under $q_T$ -- cf.~ [p. 200 (3)].

Suppose now that $(G,H)$ is a pair for which the analogue of Theorem \ref{thm: pesudoperiod} is applicable.
More precisely, assume that (in the notation of \S\ref{sec: Hperiods}) for any Levi $M$ there exists a finite collection $\affines^H_M$
of affine subspaces of $(\aaa_M^G)^*$, and for each $\sing\in\affines^H_M$
an element $\lambda_\sing\in\sing$ and holomorphic functions $J_\sing(\varphi,\cdot)$, $\varphi\in\rdp_P(G)$
in a neighborhood (in $\sing_\C$) of $\Re\lambda=\lambda_\sing$ such that for any $\phi\in\rdsmth{R}$ we have
\begin{equation} \label{eq: Hperiodgen}
\int_{\RAS HG}\theta_\phi(h)\ dh
=\sum_{\sing\in\affines^H_M}\int_{\lambda\in\sing_{\C}:\Re\lambda=\lambda_\sing}J_{\sing}(\phi[\lambda],\lambda)\ d\lambda
\end{equation}
where the integrals are absolutely convergent.
For any $\cuspdatum\in\cuspdata$ and $(M,\pi)\in\cuspdatum$ let $\affines^H_{(M,\pi)}$ be the set of $\sing\in\affines_M^H$ for which
$J_\sing(\varphi,\cdot)$ does not vanish identically for some $\varphi\in\rdp_P(G)\cap L^2_{\cusp,\pi}(\grp U(\A)M\bs\grp G(\A))$.
We write $\affines^H_{\cuspdatum}=\{(M,\pi,\sing):(M,\pi)\in\cuspdatum,\sing\in\affines^H_{(M,\pi)}\}$ and assume
(without loss of generality, by enlarging $S_\cuspdatum$ if necessary) that $\affines^H_{\cuspdatum}$ forms a union of equivalence classes of $S_\cuspdatum$.
Let $[\affines^H_{\cuspdatum}]$ be the set of equivalence classes of $\affines^H_{\cuspdatum}$.
Let
\[
[S_\cuspdatum]_H^\circ=\{\singcls\in[S_\cuspdatum]:\singcls'\not\succeq\singcls\text{ for any }\singcls'\in[\affines_{\cuspdatum}^H]\}
\]
and let $[S_\cuspdatum]_{H\dist}$ be its complement in $[S_\cuspdatum]$, i.e.,
\[
[S_\cuspdatum]_{H\dist}=\{\singcls\in[S_\cuspdatum]:\text{there exists }\singcls'\in[\affines_{\cuspdatum}^H]\text{ such that }\singcls'\succeq\singcls\}.
\]

Note that if $\singcls\in [S_\cuspdatum]_H^\circ$ then by \eqref{eq: Hperiodgen} we have
$\int_{\RAS HG}\theta_\phi(h)\ dh=0$ for any $\phi\in\tilde P_{\not\succeq\singcls}^{R,\Ktypes}$.
From \eqref{eq: closure} we conclude:

\begin{corollary} \label{cor: Hdist}
For any $\cuspdatum\in\cuspdata$ we have
\[
\dsum_{\singcls\in[S_\cuspdatum]_H^\circ}L^2_\cuspdatum(\AS{G})_\singcls\subset\overline{\pseudospace_{\cuspdatum}(\AS{G})_H^\circ}.
\]
Therefore,
\[
L^2_{H\dist}(\AS{G})\subset\hat\dsum_{\cuspdatum\in\cuspdata,\singcls\in[S_\cuspdatum]_{H\dist}}L^2_\cuspdatum(\AS{G})_\singcls.
\]
\end{corollary}

\begin{remark}
In general, the inclusions in Corollary \ref{cor: Hdist} are likely to be strict.
\end{remark}

\subsection{} \label{sec: discrete data}
In the next two sections we will further explicate Corollary \ref{cor: Hdist} in the cases at hand.
First we introduce some more notation. (See Chapter VI for more details.)
Let $\Pi_{\disc}(A_G\bs\grp G(\A))$ be the set of equivalence classes of irreducible representations which occur in the discrete spectrum
$L^2_{\disc}(\AS{G})$ of $L^2(\AS{G})$.
For any $\pi\in\Pi_{\disc}(A_G\bs\grp G(\A))$ let $L^2_{\disc,\pi}(\AS{G})$ be the isotypic component of $\pi$ in $L^2_{\disc}(\AS{G})$.

Let $\Discdata=\Discdata^G$ be the set of all equivalence classes $[(L,\delta)]$ of pairs $(L,\delta)$ up to association where
$L$ is a Levi subgroup of $G$ and $\delta\in\Pi_{\disc}(A_L\bs\grp L(\A))$.
Consider
\begin{multline*}
L^2(\iii(\aaa_L^G)^*;\Ind L^2_{\disc}(\AS{L}))^{N_G(L)}=\{\varphi:\iii(\aaa_L^G)^*\rightarrow\Ind L^2_{\disc}(\AS{L})\ |
\\ \varphi(w\lambda)=M(w,\lambda)\varphi(\lambda)\text{ for all }w\in N_G(L),\text{ for almost all }\lambda\in\iii(\aaa_L^G)^*,\\
\int_{\iii(\aaa_L^G)^*/N_G(L)}\norm{\varphi(\lambda)}^2\ d\lambda<\infty\}/\{\varphi|\varphi(\lambda)=0\text{ almost everywhere}\}.
\end{multline*}
Here $\Ind$ stands for (normalized) parabolic induction to $\grp G(\A)$ from the parabolic subgroup with Levi part $L$
and $M(w,\lambda)$ are the intertwining operators (which are unitary for $\lambda\in\iii\aaa_L^*$).

We have a map
\[
\conteisen_L:L^2(\iii(\aaa_L^G)^*;\Ind L^2_{\disc}(\AS{L}))^{N_G(L)}\rightarrow L^2(\AS{G})
\]
given by
\[
\conteisen_L(\varphi)=\int_{\iii(\aaa_L^G)^*/N_G(L)}E(\cdot,\varphi(\lambda),\lambda)\ d\lambda
\]
where $E(\cdot,\varphi,\lambda)$ denotes the corresponding Eisenstein series.
For any $\discdata\in\Discdata$ fix $L$ such that $(L,\delta')\in\discdata$ for some $\delta'$ and let $L^2_{\discdata}(\AS{G})$ be the image of
\[
L^2(\iii(\aaa_L^G)^*;\Ind\dsum_{(L,\delta)\in\discdata}L^2_{\disc,\delta}(\AS{L}))^{N_G(L)}
\]
under $\conteisen_L$. (It depends only on $\discdata$, not on the choice of $L$.)
One should not confuse the spaces $L^2_{\discdata}(\AS{G})$ with the spaces $L^2_{\cuspdatum}(\AS{G})$ defined above
(for cuspidal data $\cuspdatum$).
We have
\begin{equation} \label{eq: discrete data}
L^2(\AS{G})=\hat\dsum_{\discdata\in\Discdata}L^2_{\discdata}(\AS{G}).
\end{equation}
Note that the decomposition \eqref{eq: discrete data} is in general not finer than
\eqref{eq: coarse decomposition} -- it is conceivable that a cuspidal representation is equivalent to a non-cuspidal representation
occurring in the discrete spectrum (cf.~the notational convention on [II.1.1, p. 79]).

\section{The case of $(\GL_{2n},\Sp_n)$} \label{sec: GL2nSpn}

In this section only, let $\grp G=\grp{GL}_{2n}$ and $\grp H=\grp{Sp}_n$.
We will study the distinguished automorphic spectrum of $\grp G$ with respect to $\grp H$.
Note that $\RAS HG=\AS H$ in this case.

\subsection{}
First we explicate Corollary \ref{cor: Hdist} in the case at hand using
the description of the discrete spectrum of $L^2(\AS{GL_n})$ due to M\oe glin-Waldspurger \cite{MR1026752} which we now recall.

Let $M=M_{(n_1,\dots,n_k)}\simeq \GL_{n_1}\times\cdots\times\GL_{n_k}$ be the Levi subgroup of $\GL_n$
corresponding to a composition $n=n_1+\cdots+n_k$ and let $P=M\ltimes U$ be the corresponding parabolic subgroup.
For a representation $\pi$ of $\grp M(\A)$ and $\lambda\in \aaa_{M,\C}^*$ let $I(\pi,\lambda)$ be the representation of
$\grp{GL}_n(\A)$ parabolically induced from $\pi \otimes e^{\sprod{\lambda}{H_P(\cdot)}}$.

Suppose that $n=mr$, $M=M_{(m,\dots,m)}$ and $\pi=\tau\otimes\dots\otimes\tau$ ($r$ times) where $\tau\in\Pi_{\cusp}(A_{\GL_m}\bs\grp{GL}_m(\A))$.
Let $\mu_M=((r-1)/2,\dots,-(r-1)/2)\in\aaa_M^*$ so that $(\mu_M)_i-(\mu_M)_{i-1}=1$, $i=1,\dots,r-1$.
Then the limit
\[
\lim_{\lambda\to\mu_M}\left(\prod_{i=1}^{r-1}(\lambda_i-\lambda_{i+1}-1)\right)E(\varphi,\lambda)
\]
exists and as $\varphi$ varies in $I(\pi,\lambda)$ it spans an irreducible subrepresentation
of $L^2(\AS{GL_n})$ which is isomorphic to the Langlands quotient $\Speh(\tau,r)$ of $I(\pi,\mu_M)$.
Moreover, as we vary $m$ and $\tau$ we get the entire discrete spectrum of $L^2(\AS{GL_n})$ this way, namely:
\[
\Pi_{\disc}(A_{\GL_n}\bs\grp{GL}_n(\A))=\{\Speh(\tau,r):n=mr,\ \tau\in\Pi_{\cusp}(A_{\GL_m}\bs\grp{GL}_m(\A))\}.
\]
In particular, $L^2_{\disc}(\AS{GL_n})$ is multiplicity free.

More generally, any $\discdata\in\Discdata$ is of the form $[(L,\delta)]$ with $L=M_{(n_1,\dots,n_l)}$ and
$\delta=\Speh(\tau_1,r_1)\otimes\dots\otimes\Speh(\tau_l,r_l)$ where $n_i=r_im_i$ and $\tau_i\in\Pi_{\cusp}(A_{\GL_{m_i}}\bs\grp{GL}_{m_i}(\A))$, $i=1,\dots,l$. If
\[
M=M_{(\underbrace{m_1,\dots,m_1}_{r_1},\dots,\underbrace{m_l,\dots,m_l}_{r_l})},\
\pi=\underbrace{\tau_1\otimes\dots\otimes\tau_1}_{r_1}\otimes\dots\otimes\underbrace{\tau_l\otimes\dots\otimes\tau_l}_{r_l}
\]
and $\mu_\delta=((r_1-1)/2,\dots,(1-r_1)/2,\dots,(r_l-1)/2,\dots,(1-r_l)/2)$ then
in the notation of \S\ref{sec: finerspectraldecomposition} we have
\begin{equation} \label{eq: discdata=cuspdata}
L^2_{\discdata}(\GL_n\bs \grp{GL}_n(\A))=L^2_\cuspdatum(\GL_n\bs \grp{GL}_n(\A))_\singcls
\end{equation}
where $\cuspdatum=[(M,\pi)]$ and $\singcls=[(M,\pi,\mu_\delta+(\aaa_L^G)^*)]$.
Note that $\mu_\delta$ depends only on $M$ and $L$.

Let $\Pi_{\disc,H\type}(A_G\bs\grp G(\A))$ be the subset of $\Pi_{\disc}(A_G\bs\grp G(\A))$ consisting of $\Speh(\sigma,r)$ with $r$ even
(the ``even'' Speh representations) and let
\[
L^2_{\disc,H\type}(\AS{G})=\dsum_{\pi\in\Pi_{\disc,H\type}(A_G\bs\grp G(\A))}L^2_{\disc,\pi}(\AS{G}).
\]

We say that a Levi subgroup $L$ (or its associate class) is \emph{even} if $L=M_{(2n_1,\dots,2n_k)}$ where $n_1+\dots+n_k=n$.

Let
\begin{multline*}
\Discdata_{H\dist}=\{[(L,\delta)]:L=M_{(2n_1,\dots,2n_k)},n_1+\dots+n_k=n,\delta=\delta_1\otimes\dots\otimes\delta_k,\\
\delta_i\in\Pi_{\disc,\Sp_{n_i}\type}(A_{\GL_{2n_i}}\bs\grp{GL}_{2n_i}(\A)),i=1,\dots,k\}.
\end{multline*}

We denote by $L^2_{\disc,\Sp\type}(\AS{L})$ the image of
\[
\otimes_{i=1}^kL^2_{\disc,\Sp_{n_i}\type}(\AS{GL_{2n_i}})
\]
under the isomorphism
\[
\otimes_{i=1}^kL^2(\AS{GL_{2n_i}})\rightarrow L^2(\AS{L}).
\]

The formula for the period of pseudo Eisenstein series in this case was considered in \cite{MR2254544}.\footnote{The main result in
[loc. cit.] is stated for pseudo Eisenstein series built from Paley-Wiener sections,
but the argument is applicable equally well to the spaces $P^{R,\Fam}_\cuspdatum$ considered in \S\ref{sec: dist spec}.
Alternatively, we can argue as in \S\ref{sec: Hperiods}.}
For each class in $[\affines^H_{\cuspdatum}]$ we can take a representative of the form $(M,\pi,\sing)$
where $M=M_{(n_1,n_1,\dots,n_l,n_l)}$, $\pi=\tau_1\otimes\tau_1\otimes\dots\otimes\tau_l\otimes\tau_l$ and
\[
\sing=\{(\lambda_1+\frac12,\lambda_1-\frac12,\dots,\lambda_l+\frac12,\lambda_l-\frac12):\lambda_1,\dots,\lambda_l\in\R,
\lambda_1+\dots+\lambda_l=0\}.
\]
We can conclude:

\begin{proposition} \label{prop: uprbndGLn}
We have
\begin{equation} \label{eq: uprbndGLn}
L^2_{H\dist}(\AS{G})\subset\hat\dsum_{\discdata\in\Discdata_{H\dist}}L^2_{\discdata}(\AS{G})
=\bigoplus_{[L]\text{ even}}\conteisen_L(L^2(\iii(\aaa_L^G)^*;\Ind L^2_{\disc,\Sp\type}(\AS{L}))^{N_G(L)}).
\end{equation}
\end{proposition}

\begin{proof}
Let $\cuspdatum\in\cuspdata$.
Using Corollary \ref{cor: Hdist} and \eqref{eq: discdata=cuspdata} we need to show that if $(M,\pi)\in\cuspdatum$ and $(L,\delta)$ is discrete data such that
$L\supset M$ and $[(M,\pi,\mu_\delta +(\aaa_L^G)^*)]\in [S_\cuspdatum]_{H\dist}$ (i.e., $\singcls'\succeq[(M,\pi,\mu_\delta +(\aaa_L^G)^*)]$
for some $\singcls'\in[\affines_{\cuspdatum}^H]$) then $\delta$ is the tensor product of even Speh representations.

Let $M=M_{(n_1,\dots,n_k)}$ and $L=M_{(n'_1,\dots,n'_{k'})}$ where $n_1+\dots+n_k=n'_1+\dots+n'_{k'}=2n$.
For any $j=0,\dots,k'$ there exists $l'_j\in\{0,\dots,k\}$ (with $l'_0=0$ and $l'_{k'}=k$) such that $n_1+\dots+n_{l'_j}=n'_1+\dots+n'_j$.
By the above description of $\affines^H_{\cuspdatum}$, the condition $[(M,\pi,\mu_\delta +\aaa_L^*)]\in [S_\cuspdatum]_{H\dist}$ is that $k=2l$ is even and
for a suitable Weyl element $w$ we have $w\pi=\tau_1\otimes\tau_1\otimes\dots\otimes\tau_l\otimes\tau_l$ and
\[
w(\mu_\delta+(\aaa_L^G)^*)\subset \{\nu_1+\frac12,\nu_1-\frac12,\dots,\nu_l+\frac12,\nu_l-\frac12:\nu_1,\dots,\nu_l\in\R,\ \nu_1+\cdots+\nu_l=0\}.
\]
Thus, there exists a permutation $\sigma$ of $\{1,\dots,k\}$ such that $\pi_{\sigma(2i-1)}=\pi_{\sigma(2i)}$ and
$x_{\sigma(2i-1)}=x_{\sigma(2i)}$, $i=1,\dots,l$ for all
$(x_1,\dots,x_k)\in\aaa_L^*$ (where we view $\aaa_L^*$ as a subspace of $\aaa_M^*\simeq\R^k$).
It follows that for every $i=1,\dots,l$ there exists $j=1,\dots,k'$ such that
$l'_{j-1}<\sigma(2i-1),\sigma(2i)\le l'_j$.
Thus, $l'_j-l'_{j-1}$ is even for all $j$. This means that $[(L,\delta)]\in\Discdata_{H\dist}$.
\end{proof}

We will show below that the inclusion in \eqref{eq: uprbndGLn} is in fact an equality.

\subsection{}
Let $B$ be the Borel subgroup of $G$ of upper triangular matrices, $T$ the diagonal torus of $G$ and $T_H=T\cap H$, a maximal torus of $H$.
The group $H$ is obtained as the fixed point set of the involution $\theta$ given by
$\theta(g)=J_n^{-1}\,^tg^{-1}J_n$. Note that $\theta$ preserves both $T$ and the subgroup of upper unitriangular matrices.
Let $\srts_0^H$ be the set of simple roots of $T_H$ in $\Lie(H)$ with respect to $B_H=B\cap H$.
The restriction of characters from $T$ to $T_H$ gives rise to a surjective map $\pr_H:\aaa_0^*\rightarrow(\aaa_0^H)^*$.
In the standard coordinates we have
\[
\pr_H(x_1,\dots,x_{2n})=(x_1-x_{2n},\dots,x_n-x_{n+1}).
\]
In particular $\pr_H(\srts_0)=\srts_0^H$ and
\begin{equation} \label{eq: prHrho_0}
\pr_H(\rho_0)=(2n-1,\dots,1)=2\rho_0^H-(1,\dots,1).
\end{equation}
Note that $\theta$ induces the involution $(x_1,\dots,x_{2n})\mapsto -(x_{2n},\dots,x_1)$ on $\aaa_0^*$.

We can take the sets $\siegel_G$ and $\siegel_H$ such that $\siegel_H=\siegel_G\cap \grp H(\A)$.

Let $\phi$ be an automorphic form on $\AS{G}$ and let $P=M\ltimes U$ be a parabolic subgroup.
Recall that the set of cuspidal exponents of $\phi$ along $P$ is a finite subset of $\aaa_{M,\C}^*$ defined in
\cite[\S I.3]{MR1361168}. These sets (as we vary $P$) determine the growth of $\phi$ along the cusps
[ibid., Lemma I.4.1]. In particular, we can study the integrability over $\AS{H}$ in terms of the cuspidal exponents.
(See also \cite[Lemma 2.4]{Ya}.)

\begin{lemma} \label{lem: conv critGLn}
Let $\phi$ be an automorphic form on $\AS{G}$.
Suppose that for any parabolic subgroup $P=M\ltimes U$ of $G$ and any cuspidal exponent $\lambda$ of $\phi$ along $P$,
the coordinates $(x_\alpha)_{\alpha\in\srts_0^H}$ of $\pr_H(\Re\lambda)-(1,\dots,1)$ with respect to the basis
$\srts_0^H$ satisfy $x_\alpha<0$ for all $\alpha\notin\pr_H(\srts_0^M)$.
Then $\phi$ is absolutely integrable over $\AS{H}$.
\end{lemma}

\begin{proof}
Fix $\delta>0$ sufficiently small and let
\[
A_0^H(\delta)=\{a\in A_{T_H}:e^{\sprod{\alpha}{\Ht_{T_H}(a)}}>\delta,\ \forall\alpha\in \srts_0^H\}.
\]
Also, let $K_H=K\cap \grp H(\A)$.
Then
\[
\int_{\AS{H}}\abs{\phi(h)}\ dh\le\int_{K_H}\int_{B_H\bs \grp{B_H}(\A)^1}\int_{A_0^H(\delta)}\abs{\phi(bak)}\modulus_{B_H}(a)^{-1}\ db\ da\ dk.
\]
Observe that
\[
A_0^H(\delta)\subset A_0(\delta):=\{a\in A_T:e^{\sprod{\alpha}{\Ht_T(a)}}>\delta,\ \forall\alpha\in \srts_0\}.
\]
By \cite[Lemma I.4.1]{MR1361168} there exists $N$ such that for any choice of $\mu^P\in(\aaa_0^M)^*$, $P=M\ltimes U$ parabolic,
we have
\[
\abs{\phi(g)}\ll_{\{\mu^P\}_P}\sum_{(P,\lambda)}e^{\sprod{\Re\lambda+\mu^P+\rho_0}{\Ht_0(g)}}(1+\norm{\Ht_P(g)})^N,
\]
for any $g\in \grp G(\A)$ such that $\sprod{\alpha}{\Ht_0(g)}>\delta$ for all $\alpha\in\srts_0$,
where the sum ranges over the pairs consisting of a parabolic subgroup $P$ and a cuspidal exponent $\lambda$ of $\phi$
along $P$.
Therefore, to show the convergence of $\int_{\AS{H}}\abs{\phi(h)}\ dh$ it suffices to prove that
\[
\int_{A_0^H(\delta)}e^{\sprod{\Re\lambda+\mu^P+\rho_0}{\Ht_T(a)}}(1+\norm{\Ht_T(a)})^N\modulus_{B_H}(a)^{-1}\ da<\infty
\]
for any $(P,\lambda)$ and a suitable choice of $\mu^P\in(\aaa_0^M)^*$.
Equivalently, $\pr_H(\Re\lambda+\mu^P+\rho_0)-2\rho_0^H$ is a linear combination of $\srts_0^H$ with negative coefficients.
Clearly, by \eqref{eq: prHrho_0} this is equivalent to the condition stated in the lemma.
\end{proof}

\begin{remark} \label{rem: notconv2}
Conversely, by an argument similar to that of \cite[Lemma I.4.11]{MR1361168} (for the criterion of square-integrability)
one can show that if $\phi$ is an automorphic form on $\AS{G}$ such that
\[
\int_{\AS{H}}\abs{\phi(hg)}\ dh<\infty
\]
for all $g\in \grp G(\A)$ then the cuspidal exponents of $\phi$ satisfy the conditions of Lemma \ref{lem: conv critGLn}.
We omit the details since we will not need to use this fact.
\end{remark}

We also have the following fact.

\begin{lemma} \label{lem: convconv}
We have $L^2(\AS{G})_{\Hconv}=L^2(\AS{G})$.
Moreover, the map $(f,\varphi)\mapsto\int_{\AS{H}}f*\varphi(h)\ dh$ is a continuous bilinear form on $C_c(\grp G(\A))\times L^2(\AS{G})$
with the compact open topology on $C_c(\grp G(\A))$.
\end{lemma}

\begin{proof}
Let $\Xi$ be the function on $\siegel_G^1$ given by $\Xi(g)=e^{\sprod{\rho_0}{\Ht(g)}}$.
By \cite[Corollary 2.3]{MR3156857} we have
\[
\abs{f*\varphi(x)}\ll_R\norm{f}_\infty\Xi(x)\norm{\varphi}_{L^2(\AS{G})}, \ \ \ x\in\siegel_G^1
\]
for any $\varphi\in L^2(\AS{G})$ and $f\in C_c(\grp G(\A))$ such that $\operatorname{supp} f\subset\{g\in \grp G(\A):\norm{g}\le R\}$.
It remains to note that $\Xi$ is integrable over $\siegel_H$ since the exponent $\pr_H(\rho_0)-2\rho_0^H=(-1,\dots,-1)$
is a negative sum of roots of $\srts_0^H$ (cf.~the proof of Lemma \ref{lem: conv critGLn}).
\end{proof}

Thus, in the case at hand $L^2_{H\dist}(\AS{G})^{\stng}$ is the orthogonal complement in
$L^2(\AS{G})$ of the closed subspace $L^2(\AS{G})_H^\circ$ defined by
\[
\{\varphi\in L^2(\AS{G}):\int_{\AS{H}}f*\varphi(h)\ dh=0\text{ for all }f\in C_c(\grp G(\A))\}.
\]
By Lemma \ref{lem: convconv} we also have
\[
L^2(\AS{G})_H^\circ=\{\varphi\in L^2(\AS{G}):\int_{\AS{H}}f*\varphi(h)\ dh=0\text{ for all }f\in\mathcal{C}(\grp G(\A))\}
\]
for any dense subspace $\mathcal{C}(\grp G(\A))$ of $C_c(\grp G(\A))$.

\subsection{}
We will now use the results of Yamana \cite{Ya} (extending those of Jacquet--Rallis \cite{MR1142486} and Offen \cite{MR2254544, MR2248833})
on the symplectic periods of Eisenstein series to prove a strong form of the opposite inclusion of Proposition \ref{prop: uprbndGLn}.

For the next result, let $P=M\ltimes U$ be a parabolic subgroup and $\pi\in\Pi_{\disc}(A_M\bs\grp M(\A))$.
We denote by $\AUT_{\pi}(\grp U(\A)M\bs \grp G(\A))$ the space of automorphic forms $\varphi$ on $\grp U(\A)M\bs \grp G(\A)$
such that $m\mapsto\modulus_P(m)^{-\frac12}\varphi(mg)$ belongs to the space of $\pi$ for all $g\in\grp G(\A)$.

\begin{lemma} \label{lem: symperunifconv}
Let $\lambda\mapsto\varphi(\lambda)$ be a continuous map from $\iii(\aaa_M^G)^*$ to a finite-dimensional subspace of $\AUT_{\pi}(\grp U(\A)M\bs \grp G(\A))$.
Then the integral $\int_{\AS{H}}E(h,\varphi(\lambda),\lambda)\ dh$ converges absolutely uniformly for $\lambda$
in compact subset of $\iii(\aaa_M^G)^*$.
\end{lemma}

\begin{proof}
We will prove that there exists $N$ such that for any compact subset $D$ of $\iii(\aaa_M^G)^*$ we have
\[
\abs{E(g,\varphi(\lambda),\lambda)}\ll_{\varphi,D} e^{\sprod{\rho_0}{\Ht_0(g)}}(1+\norm{\Ht_0(g)})^N,\ \ g\in\siegel_G^1.
\]
This will imply the corollary as in the proof of Lemma \ref{lem: conv critGLn}.
Recall that there exist $f_i\in C_c(G(\R))$ and $X_i\in\univ(\mathfrak{g})$, $i=1,2$ such that
$f_i$ and $X_i$ are invariant under conjugation by $K_\infty$ and $f_1*X_1+f_2*X_2=\delta_e$ (\cite[\S4]{MR518111}). Therefore
\begin{equation} \label{eq: convf_i}
\varphi(\lambda)=\sum_{i=1}^2I(f_i,\lambda)I(X_i,\lambda)\varphi(\lambda).
\end{equation}
Applying a suitable idempotent in the algebra of finite functions on $K$ we may assume that \eqref{eq: convf_i} holds
with some bi-$K$-finite $f_i\in C_c(\grp G(\A))$  (independent of $\lambda$).
Denote by $\Lambda^T$ Arthur's truncation operator \cite{MR558260}.
By the argument in \cite[Proposition 2.5]{MR2402686}, together with \cite[Lemma 2.2]{MR3156857} (with $n=0$)
(see also proof of \cite[Proposition 5.1]{MR3156857}) we get
\[
\abs{E(g,\varphi(\lambda),\lambda)}\ll\sum_{i=1}^2\norm{\Lambda^TE(\cdot,I(f_i,\lambda)I(X_i,\lambda)\varphi(\lambda),\lambda)}
_{L^2(\AS{G})}e^{\sprod{\rho_0}{\Ht_0(g)}},\ \ g\in\siegel_G^1
\]
for all $\lambda\in\iii(\aaa_M^G)^*$ where $T$ is sufficiently regular in the positive Weyl chamber and
$T-\Ht_0(g)$ lies in a fixed translate of the positive obtuse Weyl chamber.
In particular, we may take $T=tT_0$ for some fixed $T_0$ in the positive Weyl chamber where $t\ll 1+\norm{\Ht_0(g)}$.
Upon changing the function $\varphi$, it remains to show that there exists $N$ such that for any compact subset $D$ of $\iii\aaa_M^*$ we have
\begin{equation} \label{eq: innertruncbnd}
\norm{\Lambda^TE(\cdot,\varphi(\lambda),\lambda)}_{L^2(\AS{G})}\ll_{\varphi,D}(1+\norm{T})^N.
\end{equation}
This is a consequence of \cite[Corollary 9.2]{MR650368} (see also \cite{MR2767521}) together with the proof
of \cite[Lemma 5.2]{MR3156857}, which is based on the analysis of \cite[\S2 and \S3]{MR681738}.
(In fact, $N$ can be taken to be the rank of $G$.)
\end{proof}

\begin{proposition} \label{prop: Hdiststng}
We have
\[
L^2_{H\dist}(\AS{G})^{\stng}=\hat\dsum_{\discdata\in\Discdata_{H\dist}}L^2_{\discdata}(\AS{G})
=\bigoplus_{[L]\text{ even}}\conteisen_L(L^2(\iii(\aaa_L^G)^*;\Ind L^2_{\disc,\Sp\type}(\AS{L}))^{N_G(L)}).
\]
Equivalently,
\[
L^2(\AS{G})_H^\circ=\hat\dsum_{\discdata\notin\Discdata_{H\dist}}L^2_{\discdata}(\AS{G}).
\]
\end{proposition}

\begin{proof}
Since $\grp G(\A)$ is type I \cite[appendix]{MR2331344} and $L^2(\AS{G})$ is multiplicity free\footnote{in the sense
that the commuting algebra of the regular representation in the space of bounded operators on $L^2(\AS{G})$ is commutative}
(which follows from \cite{MR1026752} and \cite{MR623137}, together with \cite[Theorem 8.6.5 and \S18.7.6]{MR0458185})
any $\grp G(\A)$-invariant closed subspace of $L^2(\AS{G})$ is of the form
\[
\widehat\bigoplus_{\data}\conteisen_L(L^2(A_\data;\Ind \dsum_{(L,\delta)\in\discdata}L^2_{\disc,\delta}(\AS{L}))^{N_G(L)})
\]
where for each $\data$ we choose $L$ such that $(L,\delta')\in\data$ for some $\delta'$
and $A_\data$ is a Lebesgue measurable $N_G(L)$-invariant subset of $\iii(\aaa_L^G)^*$, determined up to a set of zero Lebesgue measure.
(This follows from \cite[Theorem 8.6.5, Proposition 8.4.5 and \S18.7.6]{MR0458185}.)
Consider the above decomposition for the subspace $L^2(\AS{G})_H^\circ$. By \cite[Theorem 3.2]{Ya} we have $A_\data=\iii(\aaa_L^G)^*$
(up to a measure $0$ set) unless $\data\in\Discdata_{H\dist}$.
(Alternatively, this also follows from Proposition \ref{prop: uprbndGLn}.)
We need to show that $A_\data$ is of zero measure if $\data\in\Discdata_{H\dist}$.
Suppose on the contrary that $A_\data$ has positive measure for some $\data\in\Discdata_{H\dist}$.
Thus,
\[
L^2(\AS{G})_H^\circ\supset\conteisen_L(L^2(A_\data;\Ind\dsum_{(L,\delta)\in\data}L^2_{\disc,\delta}(\AS{L}))^{N_G(L)}).
\]
Let $A$ be a bounded $N_G(L)$-invariant subset of $A_\data$ of positive measure.
Let $B$ be a ball centered at $0$ in $\iii(\aaa_L^G)^*$ containing $A$.
Take an arbitrary $f\in C_c(\grp G(\A))$.
Then for any $\varphi\in L^2(A;\Ind\dsum_{(L,\delta)\in\data}L^2_{\delta}(\AS{L}))^{N_G(L)}$ we have
\[
\int_{\AS{H}}\left(f*\int_{A/N_G(L)}E(\varphi(\lambda),\lambda)\ d\lambda\right)(h)\ dh=0.
\]
Therefore, by Lemma \ref{lem: symperunifconv}, for any bi-$K$-finite $f\in C_c(\grp G(\A))$ and a continuous function $\varphi$
from $B$ into a finite-dimensional subspace of $\sum_{(L,\delta)\in\data}\AUT_{\delta}(\grp V(\A)L\bs \grp G(\A))$ such that
$\varphi(w\lambda)=M(w,\lambda)\varphi(\lambda)$ for all $w\in N_G(L)$ we have
\[
\int_{A/N_G(L)}\int_{\AS{H}}E(h,I(f,\lambda)\varphi(\lambda),\lambda)\ dh\ d\lambda=0.
\]
Here of course $Q=L\ltimes V$ is the parabolic subgroup of $G$ with Levi part $L$.
Fixing $(L,\delta')\in\data$, it follows that $\int_{\AS{H}}E(h,I(f,\lambda)\varphi,\lambda)\ dh=0$
for any $\varphi\in\AUT_{\delta'}(\grp V(\A)L\bs \grp G(\A))$ and almost all $\lambda\in A$.
Thus $\int_{\AS{H}}E(h,\varphi,\lambda)\ dh=0$ for any $\varphi\in\AUT_{\delta'}(\grp V(\A)L\bs \grp G(\A))$
and almost all $\lambda\in A$.
On the other hand, by \cite{Ya} $\int_{\AS{H}}E(h,\varphi,\lambda)\ dh$ extends to a meromorphic function
which is not identically zero for some $\varphi\in\AUT_{\delta'}(\grp V(\A)L\bs \grp G(\A))$.
We get a contradiction (e.g., using the Weierstrass preparation theorem).\footnote{Of course, the proof
in \cite{Ya} gives more information about the possible zeros of $\int_{\AS{H}}E(h,\varphi,\lambda)\ dh$.}
\end{proof}

Combining Propositions \ref{prop: uprbndGLn} and \ref{prop: Hdiststng} and the fact that $L^2_{H\dist}(\AS{G})\supset L^2_{H\dist}(\AS{G})^{\stng}$
we obtain

\begin{corollary} \label{cor: Spndist}
We have
\[
L^2_{H\dist}(\AS{G})=L^2_{H\dist}(\AS{G})^{\stng}=
\hat\dsum_{\discdata\in\Discdata_{H\dist}}L^2_{\discdata}(\AS{G}).
\]
In particular,
\[
L^2_{\disc,H\dist}(\AS{G})=L^2_{\disc,H\type}(\AS{G}).
\]
\end{corollary}

\begin{remark} \label{rem: functoriality}
The results of this section suggest a close relationship between $L^2_{H\dist}(\AS{G})$ and $L^2(\AS{\grp{GL}_n})$.
For the generic part of $L^2(\AS{\grp{GL}_n})$ this relationship can be explicated by taking the Langlands quotient of $\Ind\pi\abs{\det\cdot}^\frac12\otimes\pi\abs{\det\cdot}^{-\frac12}$,
which is an instance of Langlands functoriality.
However, this recipe breaks down for the non-generic part of $L^2(\AS{\grp{GL}_n})$ and it remains to be seen whether this can be circumspectly phrased
in terms of Langlands functoriality.
\end{remark}

\section{The main result} \label{sec: main result}
We go back to the case $\grp G=\grp{Sp}_{2n}$ and $\grp H=\grp{Sp}_n\times\grp{Sp}_n$.
As before, we have $\RAS HG=\AS H$.
The analysis of the pair $(G,H)$ is more difficult than for the pair $(\GL_{2n},\Sp_n)$ considered before.
One of the reasons is that the analogue of Lemma \ref{lem: convconv} is no longer valid in the case at hand.
Another reason is that the description of the discrete automorphic spectrum of $\grp G$ is much more involved than that of $\grp{GL}_{2n}$
(cf.~\cite{MR1815141}).

As in the previous case we will reformulate Corollary \ref{cor: Hdist} -- see Theorem \ref{thm: reformulate upper bound} below.
Unlike in the case of \S\ref{sec: GL2nSpn} we do not expect Corollary \ref{cor: Hdist} to be tight (although we do not know how to prove it).
In any case the statement of Theorem \ref{thm: reformulate upper bound} is more elaborate than the corresponding
Proposition \ref{prop: uprbndGLn}.

\subsection{} \label{sec: fineL2dist}
Once again, we draw some facts from \cite{MR1361168} (which is implicit in all references below).
Recall that (VI) for each $\cuspdatum\in\cuspdata$ and $\singcls\in[S_\cuspdatum]$ such that $L^2_\cuspdatum(\AS{G})_\singcls\ne0$
we can attach a pair $(L,V_\delta)$ (up to association) consisting of a Levi subgroup $L$ of $G$ and an admissible (not necessarily finite length)
subrepresentation $(\delta,V_\delta)$ of $L^2_{\disc}(\AS{L})$.
Moreover, $\singcls=[(M,\pi,\mu+\aaa_L^*)]$ for some $(M,\pi)\in\cuspdatum$ with $M\subset L$, $\mu\in(\aaa_M^L)^*$ and
$L^2_\cuspdatum(\AS{G})_\singcls$ is the image of $L^2(\iii\aaa_L^*;\Ind \sum_{w\in N_G(L)}wV_\delta)^{N_G(L)}$ under $\conteisen_L$.
The discrete spectrum of $G$ itself (i.e., the case $L=G$) arises from singleton $\singcls$'s (as we vary $\cuspdatum$).

Let
\[
\alldata=\{(\cuspdatum,\singcls):\cuspdatum\in\cuspdata,\singcls\in[S_\cuspdatum]\}.
\]
Recall the decomposition
\[
L^2(\AS{G})=\hat\dsum_{(\cuspdatum,\singcls)\in\alldata}L^2_{\cuspdatum}(\AS{G})_{\singcls}.
\]

We consider the following subsets of $\alldata$:
\begin{gather*}
\pointcls=\{(\cuspdatum,\singcls)\in\alldata:\dim\singcls=0\},\\
\alldata_{H\dist}=\{(\cuspdatum,\singcls)\in\alldata:\singcls\in [S_\cuspdatum]_{H\dist}\},\\
\pointcls_{H\dist}=\pointcls\cap\alldata_{H\dist},\\
\tilde\alldata=\{(\cuspdatum,\singcls)\in\alldata:\cuspdatum\in\tilde\cuspdata\},\\
\tilde\pointcls=\pointcls\cap\tilde\alldata,
\end{gather*}
where $\tilde\cuspdata$ denotes the set of cuspidal data represented by $(M,\pi)$ where $M$ is contained in the Siegel Levi.
We have $\alldata_{H\dist}\subset\tilde\alldata$ and in particular, $\pointcls_{H\dist}\subset\tilde\pointcls$.

We will set
\[
\tilde L^2(\AS{G})=\hat\dsum_{\cuspdatum\in\tilde\cuspdata}L^2_\cuspdatum(\AS{G})
\]
and $\tilde L^2_{\disc}(\AS{G})=L^2_{\disc}(\AS{G})\cap\tilde L^2(\AS{G})$.
Since $\alldata_{H\dist}\subset\tilde\alldata$, it follows from Corollary \ref{cor: Hdist} that
\begin{equation} \label{eq: expcorHdist}
L^2_{H\dist}(\AS{G})\subset\hat\dsum_{(\cuspdatum,\singcls)\in\alldata_{H\dist}}L^2_\cuspdatum(\AS{G})_{\singcls}
\subset\tilde L^2(\AS{G}).
\end{equation}
For $\pi\in\Pi_{\disc}(\grp G(\A))$ let $\tilde L^2_{\disc,\pi}(\AS{G})=L^2_{\disc,\pi}(\AS{G})\cap \tilde L^2(\AS{G})$
and let $\tilde\Pi_{\disc}(\grp G(\A))$ be the set of representations which occur in $\tilde L^2_{\disc}(\AS{G})$.

Let
\[
L^2_{\disc,H\type}(\AS{G})=\hat\dsum_{(\cuspdatum,\singcls)\in\pointcls_{H\dist}}L^2_\cuspdatum(\AS{G})_\singcls\subset
\tilde L^2_{\disc}(\AS{G}).
\]
Once again by Corollary \ref{cor: Hdist} we have
\[
L^2_{\disc,H\dist}(\AS{G})\subset L^2_{\disc,H\type}(\AS{G}).
\]
(We expect the inclusion to be strict, but we do not know how to show this.)

Let $\Pi_{\disc,H\type}(\grp G(\A))$ be the subset of $\tilde\Pi_{\disc}(\grp G(\A))$ consisting of representations
which occur in $L^2_{\disc,H\type}(\AS{G})$.

\subsection{} \label{sec: L2distdiscdata}
We observe the following

\begin{lemma} \label{lem: lcltoglblGLn}
Suppose that $(M_i,\pi_i)\in\cuspdatum_i\in\tilde\cuspdata$ and $\lambda_i\in\aaa_{M_i,\C}^*$, $i=1,2$.
Assume that for almost all $v$ the unramified irreducible subquotients $\Ind((\pi_i)_v,\lambda_i)^{\unr}$
of $\Ind((\pi_i)_v,\lambda_i)$ coincide. Then there exists $w\in W$ such that $w(M_1,\pi_1,\lambda_1)=(M_2,\pi_2,\lambda_2)$.
\end{lemma}

\begin{proof}
Since $\pi_i$, $i=1,2$ is generic, $(\pi_i)_v$ and hence $\Ind((\pi_i)_v,\mu_i)$, is fully induced from
an unramified character of the torus $\grp T(F_v)$ for almost all $v$.
On the other hand, if $\Ind(\chi_1)^{\unr}=\Ind(\chi_2)^{\unr}$ for unramified characters $\chi_1,\chi_2$ of $\grp T(F_v)$
then $\Ind(\chi_1)=\Ind(\chi_2)$ in the Grothendieck group.
Thus, $\Ind((\pi_1)_v,\mu_1)=\Ind((\pi_2)_v,\mu_2)$ in the Grothendieck group for almost all $v$.
Let $\pi'_i=\Ind^{M_{2n;0}}(\pi_i,\mu_i)$, and similarly for $(\pi'_i)_v$, so that $\Ind((\pi_1')_v)=\Ind((\pi_2')_v)$ in the Grothendieck group
for almost all $v$.
Let $\tau_i$ be the representation of $\grp{GL}_{4n}(\A)$ induced from $\pi_i'\otimes\pi_i'^\vee$, and similarly for $(\tau_i)_v$.
Then $(\tau_1)_v=(\tau_2)_v$ in the Grothendieck group of $\grp{GL}_{4n}(F_v)$ for almost all $v$ and hence
the unramified subquotiets of $(\tau_i)_v$, $i=1,2$ are equal.
It easily follows from the classification theorem of Jacquet--Shalika for $\GL_{4n}$ \cite[Theorem 4.4]{MR623137} that
$w(M_1,\pi_1,\lambda_1)=(M_2,\pi_2,\lambda_2)$ for some $w\in W$.
\end{proof}

\begin{corollary} 
For any $(\cuspdatum,\singcls)\in\tilde\pointcls$, $L^2_{\cuspdatum}(\AS{G})_{\singcls}$
is a sum of isotypic components of $\tilde L^2_{\disc}(\AS{G})$.
Equivalently, for any distinct elements $(\cuspdatum_i,\singcls_i)$, $i=1,2$ of $\tilde\pointcls$ we have
\[
\Hom_{\grp G(\A)}(L^2_{\cuspdatum_1}(\AS{G})_{\singcls_1},L^2_{\cuspdatum_2}(\AS{G})_{\singcls_2})=0.
\]
\end{corollary}

Indeed, if $(\cuspdatum,\singcls)\in\pointcls$ then we can choose $(M,\pi)\in\cuspdatum$ and
$\{\mu\}\in\singcls$ such that for any irreducible subrepresentation $\sigma$ of $L^2_\cuspdatum(\AS{G})_{\singcls}$,
$\sigma_v$ is the unramified subquotient $\Ind(\pi_v,\mu)^{\unr}$ of $\Ind(\pi_v,\mu)$ for almost all $v$.
(In particular, it is independent of $\sigma$.)
The Corollary therefore follows from Lemma \ref{lem: lcltoglblGLn}.

\begin{corollary} \label{cor: Htype components}
We have
\[
L^2_{\disc,H\type}(\AS{G})=\hat\dsum_{\pi\in\Pi_{\disc,H\type}(\grp G(\A))}\tilde L^2_{\disc,\pi}(\AS{G}).
\]
That is, $L^2_{\disc,H\type}(\AS{G})$ is a sum of isotypic components in $\tilde L^2_{\disc}(\AS{G})$.
\end{corollary}

We say that a Levi subgroup $L$ (or its associate class) is \emph{even} if $L=M_{2n_1,\dots,2n_k;2m}$
where $n_1+\dots+n_k+m=n$.
For such $L$ let $\openstab=L_{z_\gamma}=\Sp_{n_1}\times\dots\times\Sp_{n_k}\times(\Sp_m\times\Sp_m)$
where $z_\gamma$ is given in \eqref{eq: zgamma} with $\gamma=(2n_1,\dots,2n_k;m,m)$.

Recall the set $\Discdata$ of discrete data defined in \S\ref{sec: discrete data}.
We define a subset $\Discdata_{H\type}$ of $\Discdata$ by
\begin{multline*}
\Discdata_{H\type}=\{[(L,\delta)]:L=M_{(2n_1,\dots,2n_k;2m)},n_1+\dots+n_k+m=n,\delta=\delta_1\otimes\dots\otimes\delta_k\otimes\sigma,\\
\delta_i\in\Pi_{\disc,\Sp_{n_i}\dist}(A_{\GL_{2n_i}}\bs\grp{GL}_{2n_i}(\A)),i=1,\dots,k,\sigma\in\Pi_{\disc,\Sp_m\times\Sp_m\type}(\grp{Sp}_{2m}(\A))\}.
\end{multline*}
We denote by $L^2_{\disc,\openstab\type}(\AS{L})$ the image of
\[
\otimes_{i=1}^kL^2_{\disc,\Sp_{n_i}\dist}(\AS{GL_{2n_i}})
\otimes L^2_{\disc,\Sp_m\times\Sp_m\type}(\AS{Sp_{2m}})
\]
under the isomorphism
\[
\otimes_{i=1}^kL^2(\AS{GL_{2n_i}})\otimes L^2(\AS{Sp_{2m}})\rightarrow L^2(\AS{L}).
\]
For any $\discdata\in\Discdata$ define $\tilde L^2_{\discdata}(\AS{G})=L^2_{\discdata}(\AS{G})\cap\tilde L^2(\AS{G})$.

Finally, we can state the upper bound result on $L^2_{H\dist}(\AS{G})$.

\begin{theorem} \label{thm: reformulate upper bound}
\[
L^2_{H\dist}(\AS{G})\subset\hat\dsum_{\discdata\in\Discdata_{H\type}}\tilde L^2_{\discdata}(\AS{G})
=\bigoplus_{[L]\text{ even}}\conteisen(L^2(\iii\aaa_L^*;\Ind L^2_{\disc,\openstab\type}(\AS{L}))^{N_G(L)}).
\]
\end{theorem}

Roughly speaking, the assertion is that for the discrete data which occurs in the $H$-distinguished spectrum,
the $\GL$-part is as in the case of $(\GL_{2r},\Sp_r)$ considered in the previous section,
while on the $\Sp$-part all we can say is that it arises from data in $\pointcls_{\Sp_m\times\Sp_m\dist}$.

\begin{proof}
The equality on the right-hand side follows from Corollary \ref{cor: Htype components}.
We will deduce the inclusion on the left-hand side from \eqref{eq: expcorHdist}.
Suppose that $(\cuspdatum,\singcls)\in\alldata_{H\dist}$.
Let $(M,\pi=\pi_1\otimes\dots\otimes\pi_k)\in\cuspdatum$ where $M=M_{(n_1,\dots,n_k;0)}$ with $n_1+\dots+n_k=2n$
and $(M,\pi,\mu +\aaa_L^*)\in\singcls\in [S_\cuspdatum]_{H\dist}$ with $L=M_{(n'_1,\dots,n'_{k'};m')}\supset M$.
For any $j=0,\dots,k'$ there exists $l'_j=0,\dots,k$ (with $l'_0=0$) such that $n_1+\dots+n_{l'_j}=n'_1+\dots+n'_j$.
Let $s=l'_{k'}$. Let $M'=M_{(n_{s+1},\dots,n_k;0)}$, $\pi'=\pi_{s+1}\otimes\dots\otimes\pi_k\in\Pi_{\cusp}(A_{M'}\bs\grp M'(\A))$ and
$\mu'$ the last $k-s$ coordinates of $\mu$.
We need to show that $l'_j-l'_{j-1}$ is even for all $j=1,\dots,k'$, $m'$ is even and
$[(M',\pi',\{\mu'\})]\in [S_{\cuspdatum'}]_{\Sp_{m'/2}\times\Sp_{m'/2}\dist}$ where $\cuspdatum'=[(M',\pi')]$.

By the explicit description of $\affines^H_{\cuspdatum}$, the condition $[(M,\pi,\mu +\aaa_L^*)]\in [S_\cuspdatum]_{H\dist}$
is that for some $w\in W(M)$ we have $w\pi=\tau_1\otimes\tau_1\otimes\dots\otimes\tau_l\otimes\tau_l\otimes\tau_1'\otimes\dots\otimes\tau_r'$
where $2l+r=k$ and for some $l_1\le r$ we have
\begin{itemize}
\item $\tau_i'\in\Pi_{\cusp}(A_{\GL_{2m_i'}}\bs\grp{GL}_{2m_i'}(\A))$ is $(\GL_{m_i'}\times\GL_{m_i'})$-distinguished for $i=1,\dots,l_1$,
\item $\tau_i'$ is the trivial character of $\grp{GL}_1(\A)$ for $i>l_1$,
\end{itemize}
and $w(\mu+\aaa_L^*)$ is contained in
\[
\{(\nu_1+\frac12,\nu_1-\frac12,\dots,\nu_l+\frac12,\nu_l-\frac12,\underbrace{\frac12,\dots,\frac12}_{l_1},\lambda_1,\dots,\lambda_{l_2}):
\nu_1,\dots,\nu_l\in\R\}
\]
where $\lambda_1,\dots,\lambda_{l_2}$ (with $l_1+l_2=r$) are as in Lemma \ref{lem: rhox description} (for some $x$).
Thus, there exists a permutation $\sigma$ of $\{1,\dots,k\}$ and signs $\epsilon_1,\dots,\epsilon_l$ such that
$x_{\sigma(2i-1)}+\epsilon_ix_{\sigma(2i)}$, $i=1,\dots,l$ and $x_{\sigma(i)}$, $i>2l$ are constant for all
$(x_1,\dots,x_k)\in\aaa_L^*$ (viewed as a subspace of $\aaa_M^*\simeq\R^k$).
It follows that $\sigma(i)>s$ for $i>2l$ while for every $i=1,\dots,l$ either
\[
\text{both $\sigma(2i-1)$ and $\sigma(2i)$ are bigger than $s$}
\]
or
\[
\text{$\epsilon_i=-1$ and there exists $j=1,\dots,k'$ such that $l'_{j-1}<\sigma(2i-1),\sigma(2i)\le l'_j$.}
\]

Thus, $l'_j-l'_{j-1}$ is even for all $j$.
Moreover, if $i_1,\dots,i_t$ are the indices $i=1,\dots,l$ such that $\sigma(2i-1),\sigma(2i)>s$ then
$[(M',\pi',\{\mu'\})]$ contains the representative $(M'',\pi'',\mu'')$ where
\[
\pi''=\tau_{i_1}\otimes\tau_{i_1}\otimes\dots\otimes\tau_{i_t}\otimes\tau_{i_t}\otimes\tau_1'\otimes\dots\otimes\tau_r'
\]
and
\[
\mu''=(\nu_{i_1}+\frac12,\nu_{i_1}-\frac12,\dots,\nu_{i_t}+\frac12,\nu_{i_t}-\frac12,\underbrace{\frac12,\dots,\frac12}_{l_1},\lambda_1,\dots,\lambda_{l_2}))
\]
for some $\nu_{i_1},\dots,\nu_{i_t}\in\R$.
It follows that $m'$ is even and once again by Lemma \ref{lem: rhox description} that $[(M',\pi',\{\mu'\})]\in [S_{\cuspdatum'}]_{\Sp_{m'/2}\times\Sp_{m'/2}\dist}$ as required.
\end{proof}

Theorem \ref{thm: reformulate upper bound} is not completely satisfactory since the upper bound it provides is unlikely to be tight.
We denote by $L^2_{\disc,\openstab\dist}(\AS{L})$ the image of
\[
\otimes_{i=1}^kL^2_{\disc,\Sp_{n_i}\dist}(\AS{GL_{2n_i}})
\otimes L^2_{\disc,\Sp_m\times\Sp_m\dist}(\AS{Sp_{2m}})
\]
under the isomorphism
\[
\otimes_{i=1}^kL^2(\AS{GL_{2n_i}})
\otimes L^2(\AS{Sp_{2m}})\rightarrow L^2(\AS{L}).
\]
In analogy with the case considered in the previous section, it is natural to make the following hypothesis.

\begin{conjecture} \label{conj: main}
We have
\[
L^2_{H\dist}(\AS{G})=\bigoplus_{[L]\textrm{ even}}\conteisen_L(L^2(\iii\aaa_L^*;\Ind L^2_{\disc,\openstab\dist}(\AS{L}))^{N_G(L)}).
\]
\end{conjecture}

At this stage however we can prove neither the inclusion $\subseteq$ nor the other.
Also, we do not have a precise conjecture about the space $L^2_{\disc,H\dist}(\AS{G})$ itself.
It is not even clear whether a simple description of $L^2_{\disc,H\dist}(\AS{G})$ is realistic.
See also Remark \ref{rem: functoriality2} below.

\subsection{}
For completeness we give a criterion for the convergence of $H$-period integrals of automorphic forms on $\AS{G}$
in terms of their cuspidal exponents, in analogy with Lemma \ref{lem: conv critGLn}.

\begin{lemma} \label{lem: conv crit}
Let $\phi$ be an automorphic form on $\AS{G}$.
Suppose that for any parabolic subgroup $P=M\ltimes U$ and any cuspidal exponent $\lambda$ of $\phi$ along $P$,
the coordinates $(x_\alpha)_{\alpha\in\srts_0}$ of $\Re\lambda+(0,1,\dots,n-1,-n,\dots,-1)$ with respect to the basis
$\srts_0$ satisfy $x_\alpha<0$ for all $\alpha\notin\srts_0^M$.
Then $\phi$ is absolutely integrable over $\AS{H}$.
\end{lemma}

\begin{proof}
It is more convenient to work with the centralizer $\grp H'$ of $\inj_{(n,n;0)}(I_n,-I_n)$, which is conjugate to $\grp H$.

Fix $\delta>0$ sufficiently small and let
\[
A'_0(\delta)=\{a\in A_T:e^{\sprod{\alpha}{\Ht_T(a)}}>\delta\ \forall\alpha\in \srts_0^{H'}\}
\]
where $\srts_0^{H'}$ is the set of simple roots of $T$ in $\Lie(H')$ with respect to $B_{H'}=B\cap H'$. Also, let $K_{H'}=K\cap \grp H'(\A)$.
Then
\[
\int_{\AS{H'}}\abs{\phi(h)}\ dh\le\int_{K_{H'}}\int_{B_{H'}\bs \grp{B_{H'}}(\A)^1}\int_{A'_0(\delta)}\abs{\phi(bak)}\modulus_{B_{H'}}(a)^{-1}\ db\ da\ dk.
\]
Observe that
\[
A'_0(\delta)\subset\cup_{w\in W^{G'}}A_0(\delta)^w
\]
where $W^{G'}$ is the image under $\inj$ of the set of permutation matrices corresponding to the permutations $\sigma$ of $\{1,\dots,2n\}$ such that
$\sigma(1)<\dots<\sigma(n)$ and $\sigma(n+1)<\dots<\sigma(2n)$ and
\[
A_0(\delta)=\{a\in A_T:e^{\sprod{\alpha}{\Ht_T(a)}}>\delta\ \forall\alpha\in \srts_0\}.
\]
Thus,
\[
\int_{\AS{H'}}\abs{\phi(h)}\ dh\le\sum_{w\in W^{G'}}\int_{K_{H'}}\int_{B_{H'}\bs\grp{B_{H'}}(\A)^1}
\int_{A_0(\delta)}\abs{\phi(ba^wk)}\modulus_{B_{H'}}(a^w)^{-1}\ db\ da\ dk.
\]

Recall that by \cite[Lemma I.4.1]{MR1361168} there exists $N$ such that for any choice of $\mu^P\in(\aaa_0^M)^*$, $P=M\ltimes U$ parabolic subgroup,
we have
\[
\abs{\phi(g)}\ll_{\{\mu^P\}_P}\sum_{(P,\lambda)}e^{\sprod{\Re\lambda+\mu^P+\rho_0}{\Ht_0(g)}}(1+\norm{\Ht_P(g)})^N,
\]
for any $g\in \grp G(\A)$ such that $\sprod{\alpha}{\Ht_0(g)}>\delta$ for all $\alpha\in\srts_0$,
where the sum ranges over the pairs consisting of a parabolic subgroup $P$ and a cuspidal exponent $\lambda$ of $\phi$
along $P$.
Observe that $B_{H'}\subset B^w$ for any $w\in W^{G'}$.
Thus, for any $b\in\grp{B_{H'}}(\A)^1$, $a\in A_0(\delta)$, $w\in W^{G'}$ and $k\in K$ we have
\[
\abs{\phi(ba^wk)}=\abs{\phi(b^{w^{-1}}awk)}\ll_{\{\mu^P\}_P}\sum_{(P,\lambda)}e^{\sprod{\Re\lambda+\mu^P+\rho_0}{\Ht_T(a)}}
(1+\norm{\Ht_P(a)})^N.
\]
Therefore, to show the convergence of $\int_{\AS{H'}}\abs{\phi(h)}\ dh$ it suffices to prove that
\[
\int_{A_0(\delta)}e^{\sprod{\Re\lambda+\mu^P+\rho_0}{\Ht_T(a)}}(1+\norm{\Ht_T(a)})^N\modulus_{B_{H'}}(a^w)^{-1}\ da<\infty
\]
for any $w\in W^{G'}$, $(P,\lambda)$ and a suitable choice of $\mu^P\in(\aaa_0^M)^*$.
Equivalently, the projection of $\Re\lambda+\rho_0-2w\rho_0^{H'}$ to $\aaa_M^*$ is a linear combination of $\srts_P$ with negative coefficients
where $\rho_0^{H'}$ corresponds to $\modulus_{B_{H'}}^{\frac12}$.
Note that $\rho_0^{H'}=(n,\dots,1,n,\dots,1)$ and hence $w\rho_0^{H'}-\rho_0^{H'}$ is a sum of positive
roots of $G$ with non-negative coefficients for any $w\in W^{G'}$.
Thus, it suffices to check the condition for $w=1$. The lemma therefore follows from
the fact that $\rho_0-2\rho_0^{H'}=(0,1,\dots,n-1,-n,\dots,-1)$.
\end{proof}

\begin{remark} \label{rem: notconv}
Conversely, one can show that if $\phi$ is an automorphic form on $\AS{G}$ such that
\[
\int_{\AS{H}}\abs{\phi(hg)}\ dh<\infty
\]
for all $g\in \grp G(\A)$ then the cuspidal exponents of $\phi$ satisfy the conditions of Lemma \ref{lem: conv crit}.
The argument is similar to that of \cite[Lemma I.4.11]{MR1361168} and will be omitted.
\end{remark}

\subsection{Lower bound for $L^2_{H\dist}(\AS{G})$}

We end up with an important example of representations which occur in $L^2_{\disc,H\dist}(\AS{G})$.
First we need a lemma.

\begin{lemma} \label{lem: decompcusp}
We have
\[
\pseudospace(\AS{G})_H^\circ=\dsum_{\cuspdatum}\pseudospace_{\cuspdatum}(\AS{G})_H^\circ.
\]
Thus,
\[
L^2_{H\dist}(\AS{G})=\hat\dsum_{\cuspdatum}L^2_{H\dist,\cuspdatum}(\AS{G})
\]
where $L^2_{H\dist,\cuspdatum}(\AS{G})=L^2_{H\dist}(\AS{G})\cap L^2_{\cuspdatum}(\AS{G})$. Furthermore,
$L^2_{H\dist,\cuspdatum}(\AS{G})$ is the orthogonal complement of $\pseudospace_{\cuspdatum}(\AS{G})_H^\circ$ in $L^2_{\cuspdatum}(\AS{G})$.
\end{lemma}

\begin{proof}
Given a cuspidal data $[(M,\pi)]\in\cuspdata$ and a finite set $S$ of places of $F$ including the archimedean ones,
such that $\pi_v$ is unramified for all $v\notin S$ let
\[
h_{(M,\pi)}^S(\lambda)=I(\pi^S,\lambda)^{\unr}, \ \ \lambda\in\aaa_{M,\C}^*
\]
be the unramified irreducible subquotient of $I(\pi^S,\lambda)$.
It follows from Lemma \ref{lem: lcltoglblGLn} that
\begin{equation} \label{eq: disjointimage}
\text{if $[(M_i,\pi_i)]\in\tilde\cuspdata$, $i=1,2$ are distinct then the images of $h_{(M_i,\pi_i)}^S$ are disjoint.}
\end{equation}

Fix $\Ktypes$ and let $\cuspdatum_i$ be distinct elements of $\cuspdata$, $i=1,\dots,r$.
Assume that $\sum_i\theta_{\phi_i}\in\pseudospace(\AS{G})_H^\circ$ for $\phi_i\in P_{\cuspdatum_i}^{R,\Ktypes}$.
We have to show that $\theta_{\phi_i}\in\pseudospace(\AS{G})_H^\circ$ for all $i$.
We may assume that $\cuspdatum_i\in\tilde\cuspdata$ since $\pseudospace_\cuspdatum(\AS{G})_H^\circ=\pseudospace_\cuspdatum(\AS{G})$
if $\cuspdatum\notin\tilde\cuspdata$.
If $S$ is sufficiently large then $\sum_if*\theta_{\phi_i}\in\pseudospace(\AS{G})_H^\circ$ for any bi-$K^S$-invariant function $f$.
For $R>0$ let $Y^S_{\le R}$ be the unramified part of the admissible dual
of $\grp G(\A^S)$ with parameters of real part of norm $\le R$.
Then $Y^S_{\le R}$ is a compact Hausdorff space and by the Stone--Weierstrass Theorem
the algebra $\{\hat f:f\text{ bi-$K^S$-invariant}\}$ is dense in the space of continuous functions on $Y^S_{\le R}$.
Now, it follows from Theorem \ref{thm: pesudoperiod} that for suitable $R$
\[
\int_{\AS{H}}(f*\theta_{\phi_i})(h)\ dh=[\sum_{(M_j,\pi_j)\in\cuspdatum_i}(h^S_{(M_j,\pi_j)})_*\mu_{(M_j,\pi_j)}](\hat f)
\]
for all $i$ and for some (complex-valued) measure $\mu_{(M_j,\pi_j)}$ on $(\aaa_{M_j,\C}^*)_{\le R}=\{\lambda\in\aaa_{M_j,\C}^*:\norm{\Re\lambda}\le R\}$
where $h_*$ denotes the push-forward of $h$.
(The image of $(\aaa_{M_j,\C}^*)_{\le R}$ under $h^S_{(M_j,\pi_j)}$ is contained in $Y_{\le R'}^S$ for suitable $R'$.)
Thus, $\sum_i\sum_{(M_j,\pi_j)\in\cuspdatum_i}(h^S_{(M_j,\pi_j)})_*\mu_{(M_j,\pi_j)}=0$ and by \eqref{eq: disjointimage}
$\sum_{(M_j,\pi_j)\in\cuspdatum_i}(h^S_{(M_j,\pi_j)})_*\mu_{(M_j,\pi_j)}=0$ for all $i$, i.e.
$\int_{\AS{H}}\theta_{\phi_i}(h)\ dh=0$. The lemma follows.
\end{proof}

For the rest of the section assume that $n=n_1+\dots+n_k$ is a composition of $n$
and $\pi_i\in\Pi_{\cusp}(A_{\GL_{2n_i}}\bs\grp{GL}_{2n_i}(\A))$, $i=1,\dots,k$ are pairwise inequivalent
and $\GL_{n_i}\times\GL_{n_i}$-distinguished. Equivalently (\cite{MR1159108, MR1241129})
$L(\frac12,\pi_i)L(1,\pi_i,\wedge^2)=\infty$ for all $i$. In particular, $\pi_i$ is self-dual for all $i$.
Let $P=M\ltimes U$ be the parabolic subgroup of $G$ with Levi part $M=M_{(2n_1,\dots,2n_k;0)}$
and let $\tau=\pi_1\otimes\dots\otimes\pi_k\in\Pi_{\cusp}(A_M\bs\grp M(\A))$.
We identify $\Ind\tau$ with the space of smooth functions in $L^2_{\cusp,\tau}(\grp U(\A)M\bs\grp G(\A))$.
Consider the Eisenstein series $\eisen(\varphi,\lambda)$ for $\varphi\in\Ind\tau$.\footnote{See \cite{MR2402686}.
Alternatively, it is enough to consider $K$-finite sections.}
Since the $\pi_i$'s are distinct, $(\lambda_1-\frac12)\dots(\lambda_k-\frac12)\eisen(\varphi,\lambda)$
is holomorphic in a neighborhood of $\Re\lambda_1\ge\dots\ge\Re\lambda_k\ge0$.
Let
\[
\eisen_*\varphi=\lim_{\lambda\rightarrow\lambda^0}(\lambda_1-\frac12)\dots(\lambda_k-\frac12)\eisen(\varphi,\lambda)
\]
where $\lambda^0=(\frac12,\dots,\frac12)$.
Similarly let $M(\lambda)$ be the corresponding intertwining operator (with respect to the longest Weyl element)
and $M_*=\lim_{\lambda\rightarrow\lambda^0}(\lambda_1-\frac12)\dots(\lambda_k-\frac12)M(\lambda)$.

It is known that $\eisen_*\varphi\in L^2(\AS{G})$ and these span an irreducible representation
which we denote by $\Pi_\tau$ \cite[Theorem 2.1]{MR2848523}.
Moreover, $\Pi_\tau=\Pi_{\tau'}$ if and only if $\tau'$ is obtained from $\tau$ by a permutation.

\begin{remark}\label{rmk: not conv}
In the case $k=1$, $\eisen_*\varphi$ is integrable over $\AS{H}$ and its $H$-period was computed explicitly in
\cite[Theorem 2]{MR1740991}. In contrast, for $k>1$ we cannot expect $\eisen_*\varphi$ to be integrable over $\AS{H}$.
Indeed, (cf.~Remark \ref{rem: notconv}) if $n=n_1+\dots+n_k$ with $n_1\le\dots\le n_k$ then $\lambda=(-\frac12,\dots,-\frac12)$
is a cuspidal exponent of $\eisen_*\varphi$ with respect to $P_{(2n_1,\dots,2n_k;0)}$ and (since $2n_1\le n$) the first coordinate of
the projection of $\lambda+(0,\dots,n-1,-n,\dots,-1)$ to $\aaa_M^*$ is $n_1-1\ge0$.
\end{remark}

For the next result, let $\grp{G'}=\grp M_{(2n;0)}$, $\grp{H'}=\grp H\cap \grp{G'}\simeq\grp{GL}_n\times\grp{GL}_n$
and let $\grp{P'}=\grp P\cap \grp{G'}$ be the parabolic subgroup of $\grp{G'}\simeq\grp{GL_{2n}}$ of type $(2n_1,\dots,2n_k)$.
Set $\grp{P_{H'}}=\grp{P'}\cap \grp H=\grp P\cap \grp{H'}$, $\grp{M_H}=\grp M\cap\grp H=\grp M\cap\grp{H'}$ and
$\grp{U_{H'}}=\grp{U}\cap\grp{H'}$ so that $\grp{P_{H'}}=\grp{M_H}\ltimes\grp{U_{H'}}$.
Denote by $I^{G'}(\tau)$ the parabolic induction to $\grp{G'}(\A)$ and let $W^{G'}(M)=W(M)\cap G'$
(which is in natural bijection with the set of permutations on $\{1,\dots,k\}$).

\begin{lemma} \label{lem: GLnFE}
Let $\pi_i$, $i=1,\dots,k$ and $\tau$ be as above.
Then for any $\varphi\in I^{G'}(\tau)$ and $w\in W^{G'}(M)$ we have
\[
\int_{A_{M^w}\grp{U_{H'}^w}(\A)M_H^w\bs \grp{H'}(\A)}M^{G'}(w,0)\varphi(h)\ dh=
\int_{A_M\grp{U_{H'}}(\A)M_H\bs\grp{H'}(\A)}\varphi(h)\ dh
\]
where $P^w=M^w\ltimes U^w$ is the parabolic subgroup of $G$ with Levi $M^w=wMw^{-1}$,
$M_H^w=M^w\cap H=M^w\cap H'$, $U_{H'}^w=U^w\cap H'$ and finally $M^{G'}(w,\lambda)$ is the intertwining operator
$I^{G'}(\tau,\lambda)\rightarrow I^{G'}(w\tau,w\lambda)$.
\end{lemma}

\begin{proof}
By writing $w$ as a product of simple reflections we immediately reduce to the case where $w$ is a simple reflection.
In this case it suffices to check the lemma for $k=2$.
For $\varphi\in I^{G'}(\tau)$ and $\Re s\gg1$ let $E^{G'}(\varphi,s)$ be the corresponding Eisenstein series (on $\grp{G'}(\A)$)
\[
E^{G'}(g,\varphi,s)=\sum_{\gamma\in P'\bs G'}\varphi_s(\gamma g).
\]
The truncated Eisenstein series $\Lambda^TE^{G'}(g,\varphi,s)$ is given by
\[
\Lambda^TE^{G'}(g,\varphi,s)=\sum_{\gamma\in P'\bs G'}\varphi_s(\gamma g)\chi_{\le T}(H_{P'}^{G'}(\gamma g))-
\sum_{\gamma\in P^\circ\bs G'}(M^{G'}(s)\varphi)_{-s}(\gamma g)\chi_{>T}(H_{P^\circ}^{G'}(\gamma g))
\]
where $P^\circ$ is the parabolic subgroup of $G'$ of type $(n_2,n_1)$
and $M^{G'}(s):I^{G'}(\tau)\rightarrow I^{G'}(\tau^\circ)$ is the corresponding intertwining operator where $\tau^\circ=\pi_2\otimes\pi_1$.
Here we identified the one-dimensional space $\aaa_{P'}^{G'}$ with $\R$ and $\chi_{\le T}$ is the characteristic function of the corresponding ray in $\R$.
Similar notation is assumed also for $P^\circ$.
It follows from Lemma \ref{lem: bndpseudo} that for any $N$ there exists $s_0>0$ such that
\[
\sum_{\gamma\in P'\bs G'}\abs{\varphi_s(\gamma g)}\chi_{\le T}(H_{P'}^{G'}(\gamma g))+
\sum_{\gamma\in P^\circ\bs G'}\abs{(M^{G'}(s)\varphi)_{-s}(\gamma g)}\chi_{>T}(H_{P^\circ}^{G'}(\gamma g))\ll_{s,T,N}\norm{g}^{-N}
\]
for any $g\in\siegel^1_{G'}$ and $s\in\C$ with $\Re s>s_0$.
Thus, we can compute $\int_{\RAS{H'}{G'}}\Lambda^TE^{G'}(h,\varphi,s)\ dh$ using unfolding. Only the trivial orbit contributes;
the other orbits whose contribution does not factor through a constant term involve vanishing inner periods --
either diagonally embedded $\GL_{n_1}\subset\GL_{n_1}\times\GL_{n_2}$ with $n_1=n_2$ or $\GL_{k_1}\times\GL_{k_2}\subset\GL_{k_1+k_2}$ with $k_1\ne k_2$.
Therefore, we get an identity of meromorphic functions:
\begin{align*}
\int_{\RAS{H'}{G'}}\Lambda^TE^{G'}(h,\varphi,s)\ dh=&\frac{e^{sT}}s\int_{A_M\grp{U_{H'}}(\A)M_H\bs \grp{H'}(\A)}\varphi(h)\ dh-\\
&\frac{e^{-sT}}s\int_{A_{M^\circ}\grp{U^{\circ}_{H'}}(\A)M^{\circ}_H\bs\grp{H'}(\A)}M^{G'}(s)\varphi(h)\ dh
\end{align*}
where $P^\circ=M^\circ\ltimes U^\circ$, $M^{\circ}_H=M^{\circ}\cap H$ and $U^\circ_{H'}=U^{\circ}\cap H'$.
Since the left-hand side is holomorphic at $s=0$ we conclude the functional equation.
\end{proof}

\begin{theorem} \label{thm: lwrbnd for distspec}
The representation $\Pi_\tau$ is a subrepresentation of $L^2_{\disc,H\dist}(\AS{G})$.
\end{theorem}

\begin{proof}
Let $\cuspdatum=[(M,\tau)]=\{w(M,\tau):w\in W^{G'}(M)\}$. Write $\phi=(\phi^w)_{w\in W^{G'}(M)}\in P_{\cuspdatum}^{R,\Ktypes}$.
By Theorem \ref{thm: pesudoperiod} we have
\begin{equation} \label{eq: Hperpseudo}
\int_{\AS{H}}\theta_\phi(h)\ dh=\sum_{w\in W^{G'}(M)}\int_{A_{M^w}\grp{U_H^w}(\A)M_H^w\bs \grp H(\A)}\phi^w[\lambda^0]_{\lambda^0}(h)\ dh
\end{equation}
where $P^w=M^w\ltimes U^w$ is the parabolic subgroup with Levi $M^w=wMw^{-1}$,
$M_H^w=M^w\cap H$ and $U_H^w=U^w\cap H$.
Note that by Lemma \ref{lem: GLnFE} we have
\begin{equation} \label{eq: indGLFE}
\int_{A_{M^w}\grp{U_H^w}(\A)M_H^w\bs\grp H(\A)}\phi^w[\lambda^0]_{\lambda^0}(h)\ dh=
\int_{A_M\grp{U_H}(\A)M_H\bs\grp H(\A)}(M(w^{-1},\lambda^0)\phi^w[\lambda^0])_{\lambda^0}(h)\ dh
\end{equation}
for any $w\in W^{G'}(M)$.
Note that (by direct calculation) $\modulus_P^{\frac12}(m)e^{\sprod{\lambda^0}{\Ht_M(m)}}=\modulus_{P_H}(m)$ for any $m\in\grp{M_H(\A)}$.
Therefore
\[
\int_{A_M\grp{U_H}(\A)M_H\bs\grp H(\A)}\varphi_{\lambda^0}(h)\ dh=
\int_{\grp{P_H}(\A)\bs \grp H(\A)}e^{\sprod{\lambda^0}{\Ht_P(h)}}\int_{\RAS{M_H}M}\modulus_P^{-\frac12}(m)\varphi(mh)\ dm\ dh
\]
for any $\varphi\in\Ind(\tau,\lambda^0)$.
Let $\Ind(\tau,\lambda^0)^\circ$ be the subrepresentation of $\Ind(\tau,\lambda^0)$ given by
\[
\{\varphi:\int_{A_M\grp{U_H}(\A)M_H\bs\grp H(\A)}\varphi_{\lambda^0}(hg)\ dh=0\text{ for all }g\in \grp G(\A)\}.
\]
This is a proper subspace of $\Ind(\tau,\lambda^0)$ by the condition on $\tau$.
Hence, it is contained in the kernel of $M_*$, which by local considerations is the unique maximal proper subrepresentation of
$\Ind(\tau,\lambda^0)$.
In fact, most likely $\Ind(\tau,\lambda^0)^\circ=\Ker M_*$ but we will not need to know this fact.
At any rate, it follows from \eqref{eq: Hperpseudo} and \eqref{eq: indGLFE} that
\begin{multline*}
\pseudospace_{\cuspdatum}(\AS{G})_H^\circ=\{\theta_\phi:\sum_{w\in W^{G'}(M)}M(w^{-1},\lambda^0)\phi^w[\lambda^0]
\in\Ind(\tau,\lambda^0)^\circ\}\\\subset\{\theta_\phi:\sum_{w\in W^{G'}(M)}M_*(M(w^{-1},\lambda^0)\phi^w[\lambda^0])=0\}.
\end{multline*}
On the other hand, it follows from the proof of \cite[Corollary V.3.16]{MR1361168} and the simple description of the residue datum in the case at hand that
\[
(\theta_\phi,\eisen_*\varphi)_{L^2(\AS{G})}=\sum_{w\in W^{G'}(M)}(M(w^{-1},\lambda^0)\phi^w[\lambda^0],M_*\varphi)=
\sum_{w\in W^{G'}(M)}(M_*M(w^{-1},\lambda^0)\phi^w[\lambda^0],\varphi).
\]
Thus, the orthogonal complement of $\pseudospace_{\cuspdatum}(\AS{G})_H^\circ$
in $L^2_{\cuspdatum}(\AS{G})$ contains $\Pi_\tau$.
Hence, the proposition follows from Lemma \ref{lem: decompcusp}.
\end{proof}

\begin{remark}
With the above notation it further follows that
\[
L^2_{H\dist,\cuspdatum}(\AS{G})=\Pi_\tau.
\]
Indeed, by Corollary \ref{cor: Hdist}, the orthogonal complement of $\pseudospace_{\cuspdatum}(\AS{G})_H^\circ$
in $L^2_{\cuspdatum}(\AS{G})$ is contained in $\Pi_\tau$ since $\Pi_\tau=L^2_{\cuspdatum}(\AS{G})_{[(M,\tau,\{\lambda^0\})]}$.
(It is likely that in fact $L^2_{\disc,\cuspdatum}(\AS{G})=\Pi_\tau$ but we shall say no more about it.)
\end{remark}

\begin{remark}
It is conceivable that
\[
\lim_{\min_{\alpha\in\srts_0}\sprod{\alpha}T\rightarrow\infty}\int_{\AS{H}}\Lambda^T\eisen_*(h,\varphi)\ dh
\]
exists and is equal to
\[
\int_{\grp{P_H}(\A)\bs \grp H(\A)}\int_{A_MM_H\bs \grp{M_H}(\A)}\varphi(mh)\ dm\ dh.
\]
This would be a generalization of \cite[Theorem 2]{MR1740991} (for $k=1$, where no truncation is necessary).
However, we will not discuss it here.
\end{remark}

\begin{remark} \label{rem: functoriality2}
The results of the last two sections suggest a relationship, albeit vague, between $L^2_{H\dist}(\AS{G})$
and an appropriate substitute of $L^2_{\GL_n \times \GL_n \dist}(\AS{\grp{GL_{2n}}})$.
It remains to be seen whether this can be phrased more precisely and conclusively.
\end{remark}

\begin{remark}
By a more elaborate argument it is possible to prove that other representations (for instance, the identity representation)
belong to $L^2_{\disc,H\dist}(\AS{G})$. We will not pursue this matter here.
\end{remark}



\providecommand{\bysame}{\leavevmode\hbox to3em{\hrulefill}\thinspace}
\providecommand{\MR}{\relax\ifhmode\unskip\space\fi MR }
\providecommand{\MRhref}[2]{%
  \href{http://www.ams.org/mathscinet-getitem?mr=#1}{#2}
}
\providecommand{\href}[2]{#2}

\end{document}